\documentclass{amsart}
\usepackage{amsmath}
\usepackage[hyperindex,breaklinks]{hyperref}
\hypersetup{colorlinks=true,citecolor=blue,linkcolor=blue}
\newcommand{\Hilb}{{\rm{Hilb}}}

\newcommand{\VAPS}{\rm{VAPS_G}}
\newcommand{\VSP}{{\rm{VSP}}}

\newcommand{\QQ}{\mathbb {Q}}
\newcommand{\PP}{{\mathbb P}}
\newcommand{\AAA}{{\bf A}}
\newcommand{\CC}{{\mathbb C}}

\newcommand{\GG}{{\mathbb G}}

\newcommand{\U}{{\mathcal U}}
\newcommand{\ZZ}{{\mathbb Z}}
\newcommand{\oo}{{\mathcal O}}

\newcommand{\V}{{\mathcal V}}
\newtheorem{lemma}{\bf Lemma}[section]
\newtheorem{proposition}[lemma]{\bf Proposition}
\newtheorem{theorem}[lemma]{\bf Theorem}
\newtheorem{corollary}[lemma]{\bf Corollary}
\newtheorem{question}[lemma]{\bf Question}
\newtheorem{conjecture}[lemma]{\bf Conjecture}
\newtheorem{definition}[lemma]{\bf Definition}

\newtheorem{remark}[lemma]{\bf Remark}

\numberwithin{equation}{section}

\begin{document}       
\setcounter{tocdepth}{1}

\title[VAPS to quadric powers]{Variety of apolar schemes to powers of quadrics}

%%%%%%%%%%%%%%%%%%%%%%%%%%%%%%%%%%%%%%%%%%%%%%%

%\author[Grzegorz Kapustka]{Grzegorz Kapustka}
%\address{}
%\email{ } 
%\urladdr{\href{}}
%\thanks{} 
%\author[Michal Kapustka]{Michal Kapustka}
%\address{}
%\email{ } 
%\urladdr{\href{}}
%\thanks{} 
\author[G. Kapustka]{Grzegorz Kapustka}
%\thanks{G. Kapustka supported by Narodowe Centrum Nauki 2018/30/E/ST1/00530.}
\address{G. Kapustka: Faculty of Mathematics and Informatics, Jagiellonian University, ul. Łojasiewicza 6, 30-348 Krak\'ow, Poland}
\email{grzegorz.kapustka@uj.edu.pl}

\author[M. Kapustka]{Micha\l{} Kapustka}
%\thanks{M. Kapustka supported by Narodowe Centrum Nauki 2018/31/B/ST1/02857}  
\address{M. Kapustka: Institute of Mathematics of the Polish Academy of Sciences, ul. Śniadeckich 8, 00-656 Warszawa, Poland}
\email{michal.kapustka@impan.pl}

\author[Kristian Ranestad]{Kristian Ranestad}
\address{K.Ranestad: Matematisk institutt\\
         Universitetet i Oslo\\
         PO Box 1053, Blindern\\
         NO-0316 Oslo\\
         Norway}
\email{ranestad@math.uio.no}
%\urladdr{\href{}{}
%\thanks{}
%%%%%%%%%%%%%%%%%%%%%%%%%%%%%%%%%%%%%%%%%%%%%%%

\date{\today}

%\date{\today}

\begin{abstract}
%We study, when $(n,r)=(2,3),(3,2)$, the variety $\VAPS(q^r, N(n, r)),$ %$ \subset \Hilb^{N(n, r)} \mathbb{P}^{n}$,
%a Grassmannian compactification of the variety of finite schemes of length $N(n, r) = \binom{n+r}{n}$ apolar to $q^r$, where $\{q = 0\}$ is a quadric hypersurface of rank $n+1$. In particular, we describe the geometry of the variety $\VAPS(q^2, 10)$ for $q$ a quadric in four variables, which can be seen as an analog of the Mukai-Umemura threefold; $\VAPS(q^2, 6)$ with $(n,r)=(2,2)$.
We study the variety $\VAPS(q_2^3, 10)$ (resp.~$\VAPS(q_3^2, 10)),$ %$ \subset \Hilb^{N(n, r)} \mathbb{P}^{n}$,
a Grassmannian compactification of the variety of finite schemes of length $10$ apolar to $q_{2}^3$ (resp. $q_3^2$), where $\{q_{n} = 0\}\subset \PP^{n} $ is a smooth quadric hypersurface. In particular, we show that $\VAPS(q_2^3, 10)$ is the tangent developable of a rational normal curve, while 
 $\VAPS(q_3^2, 10)$ is reducible with three singular $5$-dimensional components.
 %we describe the geometry of the variety $\VAPS(q^2, 10)$ for $q$ a quadric in four variables, which can be seen as an analog of the Mukai-Umemura threefold; $\VAPS(q^2, 6)$ with $q$ a quadric of rank three.
\end{abstract}
\maketitle

\tableofcontents

\section{Introduction}\label{intro} 

Let $f\in S:= \CC[x_0,...,x_n]$ be a homogeneous form and let $f^\perp\subset T:=\CC[y_0,...,y_n]$ be its apolar ideal, the ideal of forms $g(y_0,...,y_n)$ such that % the partial differential 
$$g(\frac{\partial}{\partial x_0},...,\frac{\partial}{\partial x_n})(f)=0.$$  An ideal $I$ is apolar to $F=\{f=0\}$ if $I\subset f^\perp$, and a scheme $\Gamma\subset \PP^{n}$ is apolar to $F$, if its homogeneous ideal $I_\Gamma$ is.
Let $H: \ZZ_{\geq 0}\to \ZZ_{> 0}$ be the Hilbert function of $m$ general points in $\PP^n,$ and 
let ${\rm VAPS}(f,m) \subset \Hilb^H(\PP^n)$, the Variety of APolar Schemes, denote the closure in the multigraded Hilbert scheme $\Hilb^H(\PP^n)$, cf. \cite{HS}, of the locus of saturated ideals $I\subset f^\perp$ such that the Hilbert function of $T/I$ is $H$. %Hdegree $m$ zero-dimensional schemes $\Gamma \subset \PP^{n}$ that are apolar to $F=\{f=0\}$.  
Note that there may be components of the locus of apolar ideals in $\Hilb^H(\PP^n)$ consisting of nonsaturated ideals.  These are of less interest to us.  Being saturated is an open condition (\cite[Theorem 2.6]{JM}) in $\Hilb^H(\PP^n)$, so we restrict our attention to the components whose general member is saturated.  Furthermore, instead of the embedding in this multigraded Hilbert scheme, we consider the image of ${\rm VAPS}(f,m)$ under the forgetful map $$\Hilb^H(\PP^n)\to \GG(k_d,N_d); \quad I\mapsto I_d \subset f^\perp_d,$$
where $$k_d=\dim_\CC I_d=\binom{n+d+1}{n}-m$$ and $N_d=\dim_\CC f^{\perp}_d$.
We denote this image by ${\rm VAPS_G}(f,m)$, when the Grassmannian ${\rm G}=\GG(k_d,N_d)$ and $d$ are understood. 
We choose $m$ to be the minimal length of an apolar scheme to $f$, and  $d$ such that $$\dim_\CC I_d=\binom{n+d+1}{n}-m$$ for any saturated ideal $I$ in ${\rm VAPS}(f,m)$.
If every apolar scheme of length $m$ has the Hilbert function of $m$ general points, 
the variety ${\rm VAPS_G}(f,m)$ will be a compactification of the variety of apolar schemes of length $m$, whether considered in the multigraded Hilbert scheme $\Hilb^H(\PP^n)$ or in the ordinary Hilbert scheme $\Hilb^m(\PP^n)$.
%We shall therefore consider 
%${\rm VAPS_G}(f,H)$ in certain Grassmannians $G$.

An apolar scheme of length $m$ that is smooth, corresponds, by the {\em Apolarity lemma}, \cite[Lemma 1.15]{IK}, to a decomposition of the form $f$  as a sum of pure powers of $m$ linear forms.  A compactification in $\Hilb^m(\PP^n)$ of the set of smooth apolar schemes, is denoted ${\rm VSP}(f,m)$, the {\em Variety of Sums of Powers of $f$}.  In some cases this compactification coincides with ${\rm VAPS_G}(f,m)$.
%Allthough it is more natural to consider the closure in the Hilbert Scheme, or the multigraded Hilbert Scheme, we will consider closures in certain Grassmannians.
Varieties ${\rm VSP}(f,m)$ have allowed new descriptions of several important algebraic varieties:  % that had not be understood by other methods.  

%Let us list some of them, where ${\rm VAPS}(f,m)$ and $\VAPS(f,m)$ coincide with $\VSP(f,m)$, the closure in the Hilbert scheme of the locus of nonsingular apolar schemes of length $m$:
\begin{itemize}
\item  K3 surfaces of degree $38$ are isomorphic to $\VSP(f_6,10)$ with $\{f_6=0\}$ a general sextic curve as described in \cite{Mukai}
\item Fano threefolds of Picard rank one and of degree $22$ are isomorphic to $\VSP(f_4,6)$ for $\{f_4=0\}$ a general quartic curve \cite{Mukai}
 \item  hyper-K\"ahler fourfolds of $K3^{[2]}$-type with Picard rank one and degree $38$ are described in \cite{IR01}
 as $\VSP(f_3,10)$ where $\{f_3=0\}$ is a general cubic fourfold.
\end{itemize}

%Being apolar is not a closed condition in the Hilbert scheme, and the schemes that are not apolar but lie in ${\rm VAPS}(f,m)$ requires a careful analysis.  The set of apolar schemes sometimes have a natural embedding in a Grassmannian that is easier to access.   

%Consider a degree $k$, such that $I_{\Gamma,k}$ has rank $d_k$ and generates the $I_\Gamma$ for every apolar scheme $\Gamma$. Then  $[I_{\Gamma,k}]\in \GG(d_k, f^\perp_k)$, so we let $\VAPS(f,m)\subset \GG(d_k, f^\perp_k)$ be the closure of the set of apolar schemes of length $m$ in this Grassmannian.  On the boundary there are $d_k$-spaces of forms of degree $k$ that generate apolar ideals to $F$ that are not saturated, but are more accessible than the boundary points in the Hilbert scheme.

In this paper we consider $$f=q^r=(x_0^2+x_1^2+...+x_{n-1}^2+x_n^2)^r$$ for $q$ a quadric form of rank $n+1$. We show in Section \ref{sec1} the following proposition that extends to arbitrary $n$ a result, when $n=2$, by Cosimo Flavi \cite{flavi2}.

\begin{proposition}\label{rank}
Let $\{q=0\}\subset\PP^{n}$ be a smooth quadric hypersurface.  The minimal length of an apolar scheme to $q^r$
is $$N(n,r)=\binom{n+r}{n},$$
and coincides with its catalecticant rank, the rank of the space of partials of degree $r$.
The Hilbert function of any apolar scheme of length $N(n,r)$ to $q^r$ coincides with the Hilbert function of $N(n,r)$ general points in $\PP^n$, i.e.
$$H=\left(1,n+1,\binom{n+2}{n},\ldots,\binom{n+r}{n},\binom{n+r}{n},\ldots \right)$$
\end{proposition}
The smallest length, $N(n,r)$, of any apolar scheme to $q^r$ is called the scheme length or cactus rank of $q^r$. In the case $n=3$, any finite scheme is smoothable, so this cactus rank coincides with the border rank, as formulated by Flavi (loc. cit.). The minimal length of a smooth apolar scheme is called the rank. For $q^r$, the rank is often strictly larger than the cactus rank and even larger than the rank of a general form of degree $2r$. An upper bound on the rank was found by Reznick \cite{Rez95} and improved by Flavi in \cite{flavi1}. Buczy\'{n}ski et al. showed in \cite{BHMT18} that among forms with higher rank than the general form, the orbit of the powers $q^r$ under $PGL(n+1)$ forms an irreducible component of the smallest dimension. Although we compute the rank of $q^r$ in two cases (see Remark \ref{ranks} and Corollaries \ref{rankq3} and \ref{rankq2}), this is not our main concern in this paper.

%Since the rank is larger than the cactus rank in our main exa, we investigate 
%was first shown by  recovers and extends the lower bounds on the border rank and rank established by Cosimo Flavi \cite{flavi1,flavi2} where the main results concern the rank of \(q^r\), a question we do not discuss here.

We consider the Hilbert function of a general set of $N(n,r)$ points in $\PP^n$, and study the geometry of the Grassmannian compactification $\VAPS(q^r,N(n,r))$, with ${\rm G}=\GG(d_{r+1},f^{\perp}_{r+1})$, where $$d_{r+1}=\binom{n+r+1}{n}-\binom{n+r}{n}=\binom{n+r+1}{n}-H(r+1).$$ 
In particular,
 $d_{r+1}=\dim_\CC I_{\Gamma,r+1}$, where $\Gamma$ is any apolar scheme of length $N(n,r)$.  %So $N(n,r)$ is the smallest integer $m$ such that the variety $\VAPS(f,m)$ is non-empty. 
% We shall show that the apolar ideal $f^{\perp}$ is generated in degree $r+1$, and that the Hilbert function of any apolar scheme of length $N(n,r)$ to $q_r$ equals that of the Hilbert scheme of $N(n,r)$ general points in $\PP^n$. 
 
%In Section \ref{sec1}, we determine \(N(n,r)\), recovering and extending lower bounds on the border rank and rank established by Cosimo Flavi \cite{flavi1,flavi2} where the main results concern the rank of \(f\), a question we do not discuss here.

%\begin{proposition}
%Let $q$ be a quadratic form of rank $n$.  The minimal length of an apolar scheme to $q^r$
%is $$N(n,r)=\binom{n+r}{n}.$$
%\end{proposition}
It is interesting to compare $\VAPS(q^r,N(n,r))$ with $\VSP_{\rm G}(f_t,N(n,r))$, where $f_t$ is a general form of degree $2r$.
%and $VSP(f_t,N(n,r))$ is the closure of nonsingular apolar schemes. 
The cactus rank of $q^r$ equals the rank of $f_t$ only in a few cases: $$(n,r) = (2,2), (2,3), (2,4), (3,2), (4,2), (5,2).$$ Of these, the case \((2,2)\) was treated by Mukai and Umemura \cite{MU}, while in the cases \((2,4)\) and \((5,2)\), the variety \(\VAPS(q^r, N(n,r))\) is of higher dimension than \(\VSP_{\rm G}(f, N(n,r))\) for a general $f$ of degree $2r$. Here, we treat two of the three cases where the dimensions coincide, namely the cases \((n,r) = (2,3)\) and \((n,r) = (3,2)\).

The variety $\VSP(q, N(n, 1))$ was discussed in the paper \cite{JRS}, with reasonably complete results for $n=3$, and partial results for $n>3$. Mukai and Umemura \cite{MU} studied the geometry of $\VSP(q^2, 6)$ for $q$ a ternary quadric and proved that it is isomorphic to a special Fano threefold of Picard rank one and degree $22$, which turned out to be interesting as the one with the largest automorphism group \cite{Pro90, DKK}. This was crucial in the proof that a general threefold among the Fano threefolds of degree $22$ (which are $\VSP(f, 6)$ for $f$ a general ternary quartic form) admits a K\"ahler-Einstein metric \cite{Don07}.

The next cases to study are $\VAPS(q^3,10)$ for a ternary quadric $q$ and $\VAPS(q^2,10)$ for $q$ a quaternary quadric; this is the aim of this paper.
A main part of our arguments is computational. Both the explicit equations and the automorphisms of $\{q=0\}$ are used in our proof. In both cases, the inverse quadric $Q^{-1}$ (that is, the inverse quadric to $\{q=0\}$) plays a role. 

Our computations lead to the following description of the apolar schemes of minimal length.

\begin{proposition}\label{apolar to q3}
Let $q$ be a ternary quadric of rank three. Then any apolar scheme of length ten to $q^3$ is supported on the tangent line to $Q^{-1}$ at some point $p$ and has a component of length at least nine at $p$.

Let $q$ be a quaternary quadric of rank four. Then any apolar scheme of length ten to $q^2$ 
 is supported in the tangent plane to $Q^{-1}$ at some point $p$.  It has a component of length at least eight at $p$, while the residual scheme of length at most two is disjoint from $Q^{-1}$. 
\end{proposition}
\begin{remark}\label{ranks}
    The minimal length of a {\em smooth} apolar scheme to $q^r$, called the rank of $q^r$, is eleven when $q$ is ternary and $r=3$, and also when $q$ is quaternary and $r=2$, see Corollaries \ref{rankq3} and \ref{rankq2}.
\end{remark}
%In the ternary case, any apolar scheme of length ten has a component of length at least nine at some point $p\in Q^{-1}$ and the residual point on the tangent line to $Q^{-1}$ at $p$.  In the quaternary case, any apolar scheme of length ten has a component of length at least eight at some point $p\in Q^{-1}$ and the residual points outside $Q^{-1}$, but in the tangent plane to $Q^{-1}$ at $p$. 
In the ternary case we show:
\begin{theorem}Let $q$ be a ternary quadric of rank three. Then $$\VAPS(q^3,10)\subset \GG(5,(q^3)^\perp)$$ is the tangent developable surface of a rational normal curve of degree $20$ in the Pl\"ucker embedding.
\end{theorem}
The theorem is part of Proposition \ref{main2}.
The tangent developable surface of a rational normal curve is a degeneration of a K3 surface, so in fact $\VAPS(q^3,10)$ is a degeneration of $\VSP_G(f_t,10)$ for a general ternary form $f_t$ of degree six, see Remark \ref{rnc of degree 20}.
The method of proof is to consider an affine deformation of a particular apolar scheme, that has an embedding as an open subvariety of  $\VAPS(q^3,10)$, and then to consider the compactification of this variety. 

We approach the quaternary case in a similar way.
The variety $\VAPS(f,k)$ (or rather $\VSP(f,k)$, although they coincide)  for a quaternary quartic form $f$ of rank $k<10$ was studied in \cite{KKRSSY}.
The case $\VAPS(q^2,10)$ is the first example of a quartic of cactus rank $10$ that we are able to describe.
Thus, the main goal of this paper concerns the description and the study of the geometry of this special fivefold.
We find that it has three components, we describe those components as degeneracy loci and compute their degrees in the natural Pl\"ucker embedding.
%Two of those components are isomorphic to each other and we call the third component the main component. 

 More precisely, we denote by $Q^{-1}$ the quadric surface that is the inverse quadric to $\{q=0\}$. In Proposition \ref{mainComp} we consider the Grassmann bundle $\pi \colon \GG(3,\mathcal E)\to Q^{-1}$ for an explicitly constructed rank six bundle $\mathcal E$ on $Q^{-1}$ embedded in $Q^{-1}\times V_{16}$. Here $V_{16}=((q^2)^\bot_3)^*$, where $(q^2)^\bot_3$ is the vector space of dimension $16$ of apolar cubics to $q^2$. Let $\U$ the relative universal subbundle on $\GG(3,\mathcal E)$ and let $(P^1_{Q^{-1}})(3,3)$ be the bundle of principal parts on $Q^{-1}$ so that the projectivisation of $(P^1_{Q^{-1}})^*$ is the projective tangent bundle to $Q^{-1}$. %is a quadric in $\PP^3$. %it is a vector bundle whose projectivisation is a projective tangent bundle to $Q^{-1}$.
Let $$E_6=\oo_{Q^{-1}}(-2,3)\oplus 4\oo_{Q^{-1}}(0,0)\oplus \oo_{Q^{-1}}(2,-3)$$ then a section of $\varphi\in {\rm H}^0(\wedge^2 E_6)$ defines a section of ${\rm H}^0(\wedge^2 \mathcal Q)$ on $\GG(3,E_6)$, where $\mathcal Q$ is the relative quotient bundle, that defines a relative Lagrangian Grassmannian $$\pi\colon LG_{\varphi}(3,E_6)\to Q^{-1}$$ with relative universal subbundle $\mathcal{U}\subset \pi^{\ast}(E_6)$.

\begin{theorem}\label{three components} Let $q$ be a quaternary quadric of rank four. Then the variety $\VAPS(q^2,10)$ has three components. % two of those components are isomorphic to each other and we call the third component the main component.
      One of the components  is isomorphic to a zero locus in $\GG(3,\mathcal E)$ of a section of $$\pi^{\ast}P^1_{Q^{-1}}\otimes \wedge^2 \U^{\ast}.$$
      The degree of this component in its natural embedding in $\GG(6,V_{16})$ is 2560.
  
  Each of the two additional component of $\VAPS(q^2,10)$ is isomorphic to the corank $2$ Lagrangian degeneracy locus between subbundles  $\pi^{\ast}V_3$ and $\mathcal{U}$ of $\pi^{\ast} E_6$ on $LG_{\varphi}(3,E_6)$ where $$V_3\simeq \oo_{Q^{-1}}(-2,3)\oplus 2\oo_{Q^{-1}} (0,-2).$$ The degree of each of the component in its natural embedding in $\GG(6,V_{16})$ is 3392.
 \end{theorem}
The proof is given in Proposition \ref{three components global}, Proposition \ref{mainComp} and Proposition \ref{addComp}.
%The two remaining components are isomorphic to each other and can be described as Lagrangian degeneracy loci.
%Let $$E_6=\oo_{Q^{-1}}(-2,3)\oplus 4\oo_{Q^{-1}}(0,0)\oplus \oo_{Q^{-1}}(2,-3)$$ then a section of $\wedge^2 \mathcal Q$ on $\GG(3,E_6)$, where $\mathcal Q$ is the relative quotient bundle, defines a relative Lagrangian Grassmannian $\pi\colon LG_{\varphi}(3,E_6)\to Q^{-1}$ with relative universal sub bundle $\mathcal{U}$.
%\begin{theorem}
 %Each of the two additional component of $\VAPS(q^2,10)$ is isomorphic to the corank $2$ Lagrangian degeneracy locus between $\pi^{\ast}V_3 \to \U^{\ast}$ on $LG_{\varphi}(3,E_6)$ where $$V_3\simeq \oo_{Q^{-1}}(-2,3)\oplus 2\oo_{Q^{-1}} (0,-2).$$ The degree of the component in its natural embedding to $\GG(6,V_{16})$ is 3392.
%\end{theorem}

The structure of the paper is the following.
In Section \ref{sec1} we introduce general results and notions. In Section \ref{ternaryquadric} we describe the variety $\VAPS(q^3,10)$ for $\{q=0\}$ a plane conic curve.  In Section \ref{main3} we start the study of $\VAPS(q^2,10)$ for $\{q=0\}$ a smooth quadric surface by a discussion of 
the new phenomena appearing in four variables, in particular ``bad limits", i.e.~unsaturated ideals that are limits of apolar ideals. In section \ref{affine deformation}, using an affine deformation of a special apolar local scheme supported at a point $p$, we describe the subset of $\VAPS(q^3,10)$ parameterizing apolar schemes containing $p$ in their support.  We identify three components and describe their closure in the Grassmannian ${\rm G}$.  In Section \ref{main6} we describe a globalization of these three components in $\VAPS(q^3,10)$. 

\subsection*{Acknowledgements}
We would like to thank Joachim Jelisiejew for fruitful discussions and helpful comments during our work on this paper. M. Kapustka was supported by the Polish National Sciences Center project number 2018/31/B/ST1/02857. G. Kapustka was supported by the Polish National Sciences Center project number 2018/30/E/ST1/00530.
 
\section{The cactus rank of a power of a quadric}\label{sec1}
Let $q=x_0^2+x_1^2+...+x_{n-1}x_n$ and let $N(n,r)$ be the minimal length of any apolar scheme to $q^r$.  
%Let $$\VAPS(q^r,N(n,r))\subset \Hilb_{N(n,r)}\PP^{n}$$ be the closure of the set of schemes of length $N(n,r)$ that are apolar to $q^r$. 
The quadratic form $q$ defines a symmetric collineation:
$$\langle y_0,...,y_n\rangle \to \langle x_0,...,x_n\rangle;\quad (a_0y_0+...+a_ny_n)\mapsto (a_0\frac{\partial}{\partial x_0}+...a_n\frac{\partial}{\partial x_n})q.$$
The inverse collineation
$\langle x_0,...,x_n\rangle\to \langle y_0,...,y_n\rangle$ is defined by the {\em inverse quadric},  
$$q^{-1}=\frac{1}{4}y_0^2+\frac{1}{4}y_1^2+...+y_{n-1}y_n.$$ 

We find finite local schemes apolar to $q$ with support on the the quadric $$\{q^{-1}=0\}.$$ These are so called tautological schemes as defined in  \cite[Section 2.3]{FJM}. In \cite[Proposition 2.8]{FJM} it is shown that a local Gorenstein scheme whose ideal is the homogenization of the apolar ideal of a dehomogenization of a form $F$ is apolar to a twist of $F$.  The same proof applies to the inversion; a local Gorenstein scheme whose ideal is the homogenization of the apolar ideal of a twisted dehomogenization of a form $F$ is apolar to $F$, when the twist of the dehomogenization is the inverse of the twist of the form $F$: If 
$$F=\sum_{I} \lambda_Ix^{I}\in \CC[x_0,\ldots,x_n]_d,\quad I\in\{(i_0,\ldots,i_n)|i_0+\cdots+i_n=d\}$$ and $f:=F(x_0,\ldots,x_{n-1},1)$, then this inverse twist is
$$\operatorname{tw}(f)=\sum_{I} \lambda_I\frac{i_n!}{x_n^{i_n}}x^{I}.$$
 %\begin{proposition}\label{rank} 
%The following equality holds
%$$N(n,r)=\binom{n+r}{n}.
%%=\lceil\frac{1}{n+1}\binom{n+2r}{n}\rceil.
%$$
%Furthermore, any the Hilbert function of any apolar scheme of length $N(n,r)$ to $q^r$ has values %$(1,n+1,\binom{n+2}{n},\ldots,\binom{n+r}{n}, \binom{n+r}{n},\ldots)$.
%\end{proposition}
%\vskip .5cm
%Proof of Proposition \ref{rank}.
\begin{proof}[Proof of Proposition \ref{rank}]
Let $q_1=x_0^2+x_1^2+...+x_{n-2}^2$, so that 
$(q_1+x_{n-1})^r$ is the dehomogenization of $q^r$ with respect to $x_n$, and let 
$f_0=\operatorname{tw}((q_1+x_{n-1})^r)$, the inverse twist of the one defined in  \cite[Section 1.2]{FJM}. Then 
$$f_0=\sum_{k=0}^r\frac{r!}{(r-k)!}q_1^{r-k}x_{n-1}^k=q_1^r+rq_1^{r-1}x_{n-1}+r(r-1)q_1^{r-2}x_{n-1}^2+\ldots$$
Let $\Gamma_0$ be the finite Gorenstein scheme, whose ideal is the homogenization of the apolar ideal of $f_0$.  Then $\Gamma_0$ is apolar to $q^r$ by \cite[Proposition 2.8]{FJM}.
%$$f_0=q_0^r=(x_0^2+x_1^2+...+x_{n-1})^r=(q_1+x_{n-1})^r.$$
%$$f=tw(q_0^r)=tw((q_1+x_{n-1})^r)=\Sigma_{k=0}^r\frac{r!}{(r-k)!}q_1^{r-k}x_{n-1}^k.$$
%(UNFORTUNATELY, $\Gamma_0$ IS NOT APOLAR TO $q^r$, BUT RATHER A TWIST OF $q^r$ THAT DEPENDS ON $r$.  TO GET OUR RESULTS FOR $q^r$ WE COULD TWIST $q_0^r$, BUT THIS WOULD BE A DIFFERENT TWIST FOR EACH $r$, SO NOT EASY TO HANDLE..)

It is supported at the point $$V(y_0,...,y_{n-1})\in\{q^{-1}=0\}.$$  The length of $\Gamma_0$ is the dimension of the space $D(f_0)$ of all partials of all orders of $f_0$.  
%We set 
%$q_0=x_0^2+x_1^2+...+x_{n-1}$.  
The first order partials of $q_0^r$ have leading terms (the higher homogeneous summand):
%$$y_i(q_0^r)=2r(q_0)^{r-1}x_i, 0\leq i\leq n-2 \quad {\rm and}\quad y_{n-1}(q_0^r)=r(q_0)^{r-1} $$
$$y_i(f_0)=2rq_1^{r-1}x_i+\ldots ,\qquad 0\leq i\leq n-2, $$
 and
$$y_{n-1}(f_0)=rq_1^{r-1}+r(r-1)q_1^{r-2}x_{n-1}+\ldots, $$
and the second order partials are
%\begin{align*}
%y_i^2(q_0^r)=&4r(r-1)(q_0)^{r-2}x_i^2+2r(q_0)^{r-1},\quad 0\leq i\leq n-2,\\
%y_iy_j(q_0^r)=&4r(r-1)(q_0)^{r-2}x_ix_j,\quad 0\leq i<j\leq n-2,\\
%y_iy_{n-1}(q_0^r)=&2r(r-1)(q_0)^{r-2}x_i,\quad 0\leq i\leq n-2,\\
%y_{n-1}^2(q_0^r)=&r(r-1)(q_0)^{r-2}.
%\end{align*}
\begin{align*}
y_i^2(f_0)=&4r(r-1)q_1^{r-2}x_i^2+2rq_1^{r-1}+\ldots,\quad 0\leq i\leq n-2,\\
y_iy_j(f_0)=&4r(r-1)q_1^{r-2}x_ix_j+\ldots,\quad 0\leq i<j\leq n-2,\\
y_iy_{n-1}(f_0)=&2r(r-1)q_1^{r-2}x_i+\ldots,\quad 0\leq i\leq n-2,\\
y_{n-1}^2(f_0)=&r(r-1)q_1^{r-2}+\ldots\quad .
\end{align*}
So the space of leading terms of partials of order at most two is spanned by
$$ q_1^r, q_1^{r-1}(x_0,\ldots,x_{n-2}),q_1^{r-1},q_1^{r-2}(x_0,\ldots,x_{n-2})^2,q_1^{r-2}(x_0,\ldots,x_{n-2}),q_1^{r-2} ,$$
where $q_1^s(x_0,\ldots,x_{n-2})^m$ are all products of $q_1^s$ with monomials of degree $m$ in the $x_i$. 
Inductively, we get that the space of leading terms of partials of order at most $r$ is spanned by:
$$ q_1^r, q_1^{r-1}(x_0,...,x_{n-2}),q_1^{r-1},...,
q_1^{r-k}(x_0,...,x_{n-2})^k,...,$$
$$q_1^{r-k}(x_0,...,x_{n-2}),q_1^{r-k}, ...,
q_1(x_0,...,x_{n-2})^{r-1},...,q_1(x_0,...,x_{n-2}),q_1, 
$$
$$(x_0,...,x_{n-2})^r,...,(x_0,...,x_{n-2}),1. $$
Clearly these partials are linearly independant and span the space $D(f_0)$ of all partials of $f_0$, and counting dimensions we get:

\begin{align*}
{\rm length}\;\Gamma_0=&{\rm dim}\;D(q_0^r)=1+(n-1+1)+\\
&+\left(\binom{k+n-2}{k}+...+n-1+1\right)+
+\left(\binom{r+n-2}{r}+...+n-1+1\right)\\
&= 1+\binom{1+n-1}{1}+...+\binom{k+n-1}{k}+...+\binom{r+n-1}{r}\\
&=\binom{r+n}{n}.
\end{align*}

The space, $\CC[y_0,...,y_n]_r(q^r)$, of partials of $q^r$ of order $r$ is also $\binom{r+n}{n}$-dimensional:  In fact, 
\begin{align*}&\left(\frac{1}{4}y_0^2+\ldots+\frac{1}{4}y_{n-2}+y_{n-1}y_n\right)(q^r)\\
=&\;\;r(r-1)q^{r-2}(x_0^2+\ldots+x_{n-1}x_n)+\frac{n+1}{2}rq^{r-1}\\
=&\;\;\left(r(r-1)+\frac{n+1}{2}\right)q^{r-1},
\end{align*}
while 
\begin{align*}(y_i^2)(q^r)=&4r(r-1)q^{r-2}x_i^2+2rq^{r-1}, i=0,...,n-2\\
(y_{n-1}y_n)(q^r)=&r(r-1)q^{r-2}x_{n-1}x_n+rq^{r-1},\\
(y_{i}y_j)(q^r)=&4r(r-1)q^{r-2}x_{i}x_j, 0\leq i<j\leq n-2\\
(y_{i}y_j)(q^r)=&2r(r-1)q^{r-2}x_{i}x_j, 0\leq i\leq n-2<j\\
(y_{n-1}^2)(q^r)=&r(r-1)q^{r-2}x_{n}^2, \\
(y_{n}^2)(q^r)=&r(r-1)q^{r-2}x_{n-1}^2. \\
\end{align*}
So $q^r$ has $\binom{2+n}{n}$ dimensional linear space of partials of order $2$ which is spanned by $q^{r-2}(x_0,...,x_n) ^2$ and hence includes  in particular the form $q^{r-1}$.
Similarly, the space of partials of $q^r$ of order four is spanned by the forms 
$q^{r-4}(x_0,...,x_n) ^4$ and includes $q^{r-3}(x_0,...,x_n)^2$ and $q^{r-2}$.

Inductively, the space of partials of order $r$ of $q^r$ is spanned by $( x_0, \dots, x_n) ^r$. In particular, the space of partials of $q^{r}$ of order $r$ is $\binom{r+n}{n}$-dimensional.
The apolar ideal of $q^r$ contains no form of degree $r$. Since the ideal of any scheme of length less than $\binom{r+n}{n}$ would contain a form of degree $r$, the cactus rank of $q^r$ is $\binom{r+n}{n}$ and equals the catalecticant rank.

For the last statement, the Hilbert function $H$ for any apolar scheme is maximal up to degree $r$, since $(q^r)^\perp$ has no forms of degree less than $r+1$. If the apolar scheme has length $N(n,r) = H(r)$, then the Hilbert function $H$ has the value $H(d) = N(n,r)$ for any $d \geq r+1$, as it is non-decreasing and equals $N(n,r)$ for large $d$.
\end{proof}

%Notice that $D(f_0)$ is spanned by partials of order at most $r$, so $f$ has at least $\binom{r+n}{n}$ partials of degree $r$ equals the maximal catalacticant rank of a form of degree $2r$, so the length of any finite apolar scheme cannot be shorter. Therefore $N(n,r)$ is the cactus rank of $f$.  
The cactus rank of $q^r$ computed in Proposition \ref{rank} is, for all but a finite set of pairs $(n,r)$, lower than the rank of the general form of degree $2r$ in $n+1$ variables. Indeed, 
recall that the rank $AH(n,d)$ of a general form of degree $d$ in $n+1$ variables was determined by Alexander-Hirschowitz, cf. \cite{AH}. When $d=2r$ it is
\begin{align*}%\label{AHbound}
    AH(n,2r) = \begin{cases}
    \lceil\frac{1}{n+1}\binom{n+2r}{n}\rceil,\;\;\text{if}\;\; n,r>1,\; (n,r)\neq (2,2),(3,2),(4,2),\\
    6,(\text{resp.}\;10,15),\;\;\text{if}\;\; (n,r)=(2,2),(\text{resp}.\;(3,2),(4,2)). \\
    \end{cases}
\end{align*}

It is straightforward to check that when $n>1,r>1$, then  $$N(n,r)< AH(n,2r), \;\text{unless}\;\; (n,r)=(2,2),(2,3),(2,4),(3,2),(4,2),(5,2).$$
%$$N(n,r)=\binom{n+r}{n}
%\leq AH(2r,n)=\lceil\frac{1}{n+1}\binom{n+2r}{n}\rceil,(n,r)\neq (2,2),(3,2),(4,2)$$
%$$AH(4,2)=6, AH(4,3)=10, AH(4,4)=15$$
%except for $$(n,r)=(2,2),(2,3),(2,4),(3,2),(4,2),(5,2),$$
\begin{remark}
The cactus rank of a general form of degree $2r$ in $n+1$ variables has yet to be determined. It is of course at most $AH(n,2r)$, but for general $(n,r)$ a better upper bound is provided by the length $$2+2n+...+2\binom{n+r-2}{n-1}+\binom{n+r-1}{n-1}=\binom{n+r}{n+1}$$ of the scheme defined by the apolar ideal of a general polynomial of degree $2r$ in $n$ variables.  A lower bound is given by the catalecticant rank, i.e.  $\binom{n+r}{n}$.
\end{remark}
Notice, that  the local apolar schemes $\Gamma_0$ are parameterized by the quadric $\{q^{-1}=0\}$, so  $\VAPS(q^r, N(n,r))$ is at least $(n-1)$-dimensional.

\begin{question}
What is the dimension of $\VAPS(q^r, N(n,r))$?
    
\end{question}

The $(n-1)$-dimensional $SO(n,q)$-orbit $SO(n,q)(\Gamma_0)$ of $\Gamma_0$ is contained in $\VAPS(q^r, N(n,r))$.  
\begin{conjecture}\label{conjecture} 
%$\VAPS(q^r,N(n,r))=(n+1)\cdot N(n,r)-\binom{n+2r}{n}$ 
$$\VAPS(q^r,N(n,r))=SO(n,q)(\Gamma_0),$$ when $n=2, r>3, n=3,4, r>2$ and $n>4,r>1$. 
\end{conjecture}

When $n=2, r=2$, Mukai \cite{Mukai} showed that $\VAPS(q^2,6)$ is a Fano $3$-fold $V_{22}$, and in fact a smooth equivariant compactification of $SO(3,q)/{\rm Icosa}$ and therefore  also a compactification of $\CC^3$.
Here we deal with the boundary cases $(n,r)=(2,3),(3,2),$ which are not covered by Conjecture \ref{conjecture}.

%\section{$\VAPS(q^3,10)$ when $q$ is a ternary quadric.}\label{ternaryquadric}
\section{The variety of  schemes of length ten apolar to the cube of a ternary quadric}\label{ternaryquadric}

The variety $\VAPS(q^2,6)$ for a ternary quadric $q$ is a special Fano $3$-fold, while $\VAPS(f,6)$ is a Fano threefold $V_{22}$ for a general ternary quartic $f$.  In this section we give a similar result for the cube $q^3$ of a ternary quadric (or conic). 
The variety $\VAPS(f,10)$ for a general ternary sextic is a K3-surface of degree $38$.  We will show that $\VAPS(q^3,10)$ is a singular specialisation of such a K3-surface.

We set $q=x_0^2+x_1x_2$.
Then the ideal $(q^3)^{\bot}$ is generated by 
\begin{align*}
 y_1^4,y_0y_1^3,&3y_0^2y_1^2-2y_1^3y_2,y_0^3y_1-3y_0y_1^2y_2,y_0^4-12y_0^2y_1y_2+6y_1^2y_2^2,\\
&y_2^4,y_0y_2^3,3y_0^2y_2^2-2y_1y_2^3,y_0^3y_2-3y_0y_1y_2^2.
\end{align*}
%\[
%(q^3)^{\bot}=y_2^4,y_0y_2^3,3y_0^2y_2^2-2y_1y_2^3,y_0^3y_2-3y_0y_1y_2^2,y_1^4,y_0y_1^3,3y_0^2y_1^2-2y_1^3y_2,y_0^3y_1-3y_0y_1^2y_2,y_0^4-12y_0^2y_1y_2+6y_1^2y_2^2
%\]
The ideal $I_{\Gamma_0}$ of the scheme $\Gamma_0$ supported at $V(y_0,y_1)$ is generated by
$$ y_1^4,
y_0y_1^3,
3y_0^2y_1^2-2y_1^3y_2,
y_0^3y_1-3y_0y_1^2y_2,
y_0^4-12y_0^2y_1y_2+6y_1^2y_2^2,$$
and it is the homogenization of the apolar ideal of 
$$x_0^6+3x_0^4x_1+6x_0^2x_1^2+6x_1^3.$$

The generators of $I_{\Gamma_0}$ form a $5$-dimensional subspace of the space of quartic generators $(q^3)^{\bot}$.
Since $(q^3)^{\bot}$ contains no cubics, any apolar scheme of length $10$ contains a $5$-spaces of quartics in its ideal.  In particular, there is a natural rational map
$$v_q: \VAPS(q^3,10)\to \GG(5, (q^3)^{\bot})$$
that is defined on the locus of apolar schemes.

\begin{lemma} The set of apolar schemes of length $10$ to $q^3$ is closed in $\Hilb_{10}\PP^{2}$.
\end{lemma}
\begin{proof} The ideal $(q^3)^\bot$ is generated by the $(8\times 8)$-Pfaffians of a skew symmetric $(9\times 9)$-matrix $M_{q^3}$ with linear entries.
In fact $M_{q^3}$ is the Buchsbaum-Eisenbud matrix of syzygies of codimension three Gorenstein ideal $(q^3)^\bot$. So there are no cubics apolar to $(q^3)^\bot$, and the ideal of any apolar scheme $\Gamma$ of length $10$ is generated by five quartics among the $(8\times 8)$-Pfaffians. In particular, the Hilbert-Burch matrix $HB_\Gamma$ of $I_\Gamma$ is a $(4\times 5)$-matrix with linear entries that after possibly row and column operations is a submatrix of $M_{q^3}$.  So we may assume that the generators of
$I_\Gamma$ are the last five Pfaffians in the above basis for $(q^3)^\bot$, and therefore that the 
matrix $HB_\Gamma$ is the upper right submatrix of $M_{q^3}$ while the upper left $(4\times 4)$-submatrix of $M_{q^3}$ is the zero-matrix.  

In the Grassmannian $\GG(4,9)$ of $4$-dimensional subspaces of the row space of $M_{q^3}$ the locus $Z$ of $4$-spaces  where the skewsymmetric matrices vanish, is at least $2$-dimensional.  For each subspace $[U]\in Z$ the corresponding $5$-space of Pfaffians $U\subset (q^3)^\bot$ are the $4$-minors of the $(4\times 5)$-submatrix of syzygies in $M_{q^3}$  of these Pfaffians.  That submatrix is the  Hilbert-Burch matrix $HB_\Gamma$ of a scheme of  length $10$, necessarily apolar to $q^3$, as soon as the $4$-minors have no common factor.   To complete the proof of the lemma, it suffices  to exclude the latter possibility.   
A common factor would be a form of degree one, two or three.    

If the degree of the common factor $g$ is three, then the $5$-space of quartic forms $U\subset (q^3)^\bot$ all have the form $lg$, where $l$ is a linear form. But there is only a $3$-space of linear forms, so this is absurd.   If $g$ has degree $2$, then $g(q^3)$ is a quartic form and $q'g(q^3)=0$ for a $5$-space of quadratic forms $q'$. A ternary quartic form with a $5$-space of quadratic forms in its apolar ideal has rank one, i.e. has a pencil of linear forms in its apolar ideal.  So $g(q^3)$ is apolar to a pencil of linear forms, and hence $q^3$ has a pencil of cubic forms in its apolar ideal. This is a contradiction.  

%Considering the middle catalecticant matrix of $g(q^3)$, a computation in \cite{Macaulay_2} shows that its rank is at least four for every quadric $g$. In particular, $g(q^3)^\bot$ contains at most a pencil of quadrics $q'$. So $g$ cannot have degree $2$. 
Finally, if $g$ is linear, then $g(q^3)$ is a quintic with at least a $5$-space of cubic forms in $g(q^3)^\bot$. But then the partials of $g(q^3)$ of order three is at most a $5$-dimensional space of quadrics. This means that there is quadratic form $g''$, such that $g''(g'(g(q^3)))=0$ for every cubic form $g'$.  But then  $g'(g''(g(q^3)))=0$ for every cubic form $g'$, so $g''(g(q^3))=0$ and $g''\cdot g$ is a cubic form in $(q^3)^\perp$. Since $(q^3)^\perp$ contains no cubic forms, this is a contradiction.  
\end{proof}
By this lemma (and its proof), the map 
$$v_q: \VAPS(q^3,10)\to \GG(5, (q^3)^{\bot})$$ is a closed embedding.
In fact, the apolar schemes of length ten correspond one to one to $5$-space of quartics in $(q^3)^{\bot}$ with a $4$-space of linear syzygies.

Therefore, we identify $\VAPS(q^3,10)$ with its isomorphic image in the Grassmannian $\GG(5, (q^3)^{\bot})$ in its Pl\"ucker embedding.

\begin{proposition}\label{main2} 
Let $q$ be a ternary quadric of rank three. Then the variety 
$\VAPS(q^3,10)$ is a surface of degree $38$, it is the tangent developable of a smooth rational curve of degree $20$, the closed $SO(3)$-orbit of the local scheme $\Gamma_0$ in $\Hilb_{10}\PP^{2}$. 
The tangent line at the point  $\Gamma_{[0:1]}:=\Gamma_0$ is formed by the schemes 
\begin{align*}
\Gamma_{[s:t]}= V(&y_1^4,
y_0y_1^3,
3y_0^2y_1^2-2y_1^3y_2,y_0^3y_1-3y_0y_1^2y_2,\\
& t(y_0^4-12y_0^2y_1y_2+6y_1^2y_2^2)+s(y_0^3y_2-3y_0y_1y_2^2))
    \end{align*}
for each $[s:t]\in \PP^1$.

For $[s:t]\not= [0:1]$, it has two components $\Gamma_{[s:t]},=\Gamma_{[s:t],1}\cup \Gamma_{[s:t],9}$:
$$\Gamma_{[s:t],1}=V(y_1,ty_0+sy_2)$$
\begin{align*}
\Gamma_{[s:t],9}=V(&s^3y_0^3-9s^2ty_0^2y_1+21st^2y_0y_1^2-14t^3y_1^3-3s^3y_0y_1y_2+6s^2ty_1^2y_2,\\ &y_1^4,y_0y_1^3,3y_0^2y_1^2-2y_1^3y_2).
\end{align*}
$\Gamma_{[s:t],1}$ is a point, while $\Gamma_{[s:t],9}$ is a local Gorenstein scheme of length $9$ supported at $V(y_0,y_1)$ defined by the apolar ideal of the polynomial
$$3s^2x_0^6+30s^2x_0^4x_1-280stx_0^3x_1+60s^2x_0^2x_1^2+2800t^2x_0^2x_1-280stx_0x_1^2+60s^2x_1^3+5600t^2x_1^2. $$
\end{proposition}

\begin{proof} We consider the $SO(3)$-orbits of apolar schemes that contain the point $V(y_0,y_1)$.  First, it is straightforward to check that the local scheme $\Gamma_0$, supported at $V(y_0,y_1)$, is apolar to $q^3$.  

The Buchsbaum-Eisenbud matrix of $(q^3)^{\bot}$ is
\[
M_q=\begin{pmatrix}
0& 0& 0& 0& -14y_0& 14y_1& 0& 0& 0\\
0& 0& 0& 0& 14y_2& -21y_0& 7y_1& 0& 0\\
0& 0& 0& 0& 0& 7y_2& -y_0& 3y_1& 0\\ 
0& 0& 0& 0& 0& 0& 3y_2&-y_0& y_1\\ 
14y_0& -14y_2& 0& 0& 0& 0& 0& 0& 0\\ 
-14y_1& 21y_0& -7y_2& 0& 0& 0& 0& 0& 0\\ 
0& -7y_1& y_0& -3y_2& 0& 0& 0& 0& 0\\ 
0& 0& -3y_1& y_0& 0&0& 0& 0& -y_2\\ 
0& 0& 0& -y_1& 0& 0& 0& y_2& 0\\
\end{pmatrix},
\]
i.e. the $(8\times 8)$-pfaffians of $M_q$ are the generators $J$ of $(q^3)^{\bot}$.

The apolar ideal $I_{\Gamma_0}$ is generated by the subset
\[J(0)=(y_1^4,y_0y_1^3,3y_0^2y_1^2-2y_1^3y_2,y_0^3y_1-3y_0y_1^2y_2,y_0^4-12y_0^2y_1y_2+6y_1^2y_2^2)
\]
of generators in
%\begin{tiny}
\begin{align*}
J=(&y_1^4,y_0y_1^3,3y_0^2y_1^2-2y_1^3y_2,y_0^3y_1-3y_0y_1^2y_2,y_0^4-12y_0^2y_1y_2+6y_1^2
       y_2^2,\\
 &y_2^4,y_0y_2^3,3y_0^2y_2^2-2y_1y_2^3,y_0^3y_2-3y_0y_1y_2^2)
       \end{align*}
   %    \end{tiny}
in $(q^3)^{\bot}$.  
The syzygies of $J(0)$ are the columns of the matrix 
\[
H(0)=\begin{pmatrix}
 14y_0& -14y_2& 0& 0\\ 
-14y_1& 21y_0& -7y_2& 0\\ 
0& -7y_1& y_0& -3y_2\\ 
0& 0& -3y_1& y_0\\ 
0& 0& 0& -y_1\\
\end{pmatrix}.
\] 
This matrix is, of course, a submatrix of $M_q$.
We consider an affine unfolding of $I_{\Gamma_0}$ defined by 
$$J(a)=D(a)\cdot J$$
where
\[
D(a)=\begin{pmatrix}
a_0& 0& 0& 0& 0&a_1& a_6& a_{11}& a_{16}\\
0& a_0& 0& 0& 0&a_2& a_7& a_{12}& a_{17}\\
0& 0& a_0& 0& 0&a_3& a_8& a_{13}& a_{18}\\
0& 0& 0& a_0& 0&a_4&a_9& a_{14}& a_{19}\\
0& 0& 0& 0& a_0&a_5&a_{10}& a_{15}& a_{20}
\end{pmatrix},
\]
with $a_0=1$.
%where $G$ is the generators of $(q^3)^{\bot}$ in the order
%\begin{tiny}
%\[
%G=(y_1^4,y_0y_1^3,3y_0^2y_1^2-2y_1^3y_2,y_0^3y_1-3y_0y_1^2y_2,y_0^4-12y_0^2y_1y_2+6y_1^2
%       y_2^2, y_2^4,y_0y_2^3,3y_0^2y_2^2-2y_1y_2^3,y_0^3y_2-%3y_0y_1y_2^2)
%       \]
 %      \end{tiny}
%$$G(a)=D(a)\cdot transpose G$$
For each $a\in \AAA^{20}$ and $a_0=1$ the five quartic forms $J(a)$ generate an ideal apolar to $q^3$ with the same initial ideal as $I_{\Gamma_0}$ with respect to the reverse lexicographical order. 

Since the initial monomials of these quartics are not divisible by $y_2$, while all the additional monomials are divisible by $y_2$, any linear syzygy among the forms in $J(a)$ is obtained from a syzygy among the forms in $J(0)$ by adding appropriate multiples of $y_2$.
On the other hand, any $5$-space $J(a)$ with $4$ linear syzygies defines, by flatness, an apolar scheme of length ten. In fact, the ideal of any apolar scheme of length ten that does not intersect $\{y_2 = 0\}$ is generated by $J(a)$ for some $a$. Any such ideal cannot contain any quartic divisible by $y_2$, since the scheme then would be contained in a cubic, while $q^3$ has no apolar cubic forms.

To find the spaces of quartics $J(a)$ that define a scheme of length $10$ we therefore consider the 
 matrix $H(b)$ extending the syzygy matrix $H(0)$:  
\[
H(b)=\begin{pmatrix}
 14y_0& -14y_2& 0& 0\\ 
-14y_1& 21y_0& -7y_2& 0\\ 
0& -7y_1& y_0& -3y_2\\ 
0& 0& -3y_1& y_0\\ 
0& 0& 0& -y_1\\
\end{pmatrix}+\begin{pmatrix}
 b_1y_2& b_2y_2& b_3y_2& b_4y_2\\ 
b_5y_2& b_6y_2& b_7y_2& b_8y_2\\ 
b_9y_2& b_{10}y_2& b_{11}y_2& b_{12}y_2\\ 
b_{13}y_2& b_{14}y_2& b_{15}y_2& b_{16}y_2\\ 
b_{17}y_2& b_{18}y_2& b_{19}y_2& b_{20}y_2\\
\end{pmatrix}
\]
and ask that $H(b)$ is a syzygy matrix for $J(a)$, i.e. that $J(a)^t H(b)=0$.  We now allow $a_0\not= 0$ in $J(a)$ to get homogenous equations in the parameters $a_i$.
The entries of the product are forms of degree five in $y_0,y_1,y_2$, so $H(b)$ is a syzygy matrix for  $J(a)$ when $21$ bilinear equations in $a$ and $b$ vanish. Eliminating the parameters $b_i$ we obtain an ideal in the $a_i$ generated by $12$ linear forms and $15$ quadratic polynomials.  The linear forms allow an elimination of $12$ variables, leaving us with $15$ quadratic polynomials in $(a_0,a_1,a_6,a_{11},a_{16},a_{17},a_{18},a_{19},a_{20})$, that define a surface $S$ in $\AAA^8$  (See Appendix \ref{equations23}).  The  ideal generated by the $15$ quadratic forms  defines an irreducible surface $\bar S\subset\PP^8$. 
% By a Macaulay2 \cite{Macaulay_2} computation, it has degree $14$.  
It contains the line $$L_0=Z(a_{1},a_{6},a_{11},a_{16},a_{17},a_{18},a_{19})$$ and is singular along the rational normal curve $C$ defined by the $2$-minors of 
%\[
%\begin{pmatrix}
%5a_0&3a_5&7a_{10}&a_{15}&a_{20}&5a_{19}&a_{18}&7a_{17}\\a_5&8a_{10}&6a_{15}&8a_{20}&5a_{19}&8a_{18}&14a_{17}&32a_{16}\\
%\end{pmatrix},
%\]
\[
\begin{pmatrix}
40a_0&7a_6&12a_{11}&5a_{16}&8a_{17}&a_{18}&4a_{19}&a_{20}\\
    a_{20}&6a_1&a_6&a_{11}&5a_{16}&2a_{17}&a_{18}&a_{19}
    \end{pmatrix},
\]
so it is parameterized, in $\{a_0=1\}$, by 
$$(a_0,a_{20},a_{19},a_{18},a_{17},a_{16},a_{11},a_{6},a_{1})$$
$$=(1,t,\frac{1}{40}t^2,\frac{1}{400} t^3,\frac{1}{32000} t^4,\frac{1}{800000} t^5,\frac{1}{6400000} t^6,\frac{3}{64000000} t^7,\frac{7}{5120000000} t^8).$$
Notice that the line $L_0$ is tangent to $C$ at the point where also $a_5$ vanishes, i.e. at the point $J(0)$.  
Clearly, $SO(q,3)$ acts on $S$, and the curve $C$ of singularities is an orbit. We conclude that the surface $\bar S$ is the tangent developable of the rational normal curve $C$ in $\PP^8$.

Since $\langle J(a)\rangle$ is a $5$-dimensional subspace of $\langle J\rangle$, we get a natural map 
$$\bar S\to \GG(5,\langle J\rangle).$$

%Since a tangent developable of a rational normal curve of degree $20$ is a degeneration of a K3-surface of degree $38$, we conclude that  the image of $S$ in $G(4,U)$ is the tangent developable of a rational normal curve.  
The line $L_0$ parameterize a linear pencil of $5$-spaces $\langle J(a)\rangle$ so the image of $L_0$ is a line in $\GG(5,\langle J\rangle)$.  Furthermore, the curve $C$ is a rational normal curve of degree $8$ in $\PP^8$.  The corresponding scroll of projective $4$-spaces in $\PP(\langle J\rangle)$ is parameterized by the matrices $D(a)$, with $a\in C$.  With the above parameterization of $C$ in $t$, we get 
\[
D(t)=\begin{pmatrix}
1& 0& 0& 0& 0& \frac{7}{5120000000}t^8&
       \frac{3}{64000000}t^7& \frac{1}{6400000}t^6& \frac{1}{800000}t^5\\ 
       0&1& 0& 0& 0& \frac{1}{32000000}t^7& \frac{7}{6400000}t^6&
       \frac{3}{800000}t^5& \frac{1}{32000}t^4\\
       0& 0& 1& 0& 0&
       \frac{7}{3200000}t^6& (\frac{63}{800000}t^5& \frac{9}{32000}t^4&
       \frac{1}{400}t^3\\
       0& 0& 0& 1& 0& \frac{7}{400000}t^5&
       \frac{21}{32000}t^4& \frac{1}{400}t^3& \frac{1}{40}t^2\\
       0& 0& 0& 0&
       1& \frac{7}{16000}t^4& \frac{7}{400}t^3& \frac{3}{40}t^2& t
\end{pmatrix}.
\]
%\[
%D(t)=\begin{pmatrix}
%1& 0& 0& 0& 0&\frac{1}{800000}t^5&\frac{1}{6400000} t^6& \frac{21}%{64000000}t^7& \frac{7}{5120000000}t^8\\
%0& 1& 0& 0& 0&\frac{1}{32000}t^4& \frac{3}{800000}t^5& \frac{7}{6400000}t^6& \frac{7}{32000000}t^7\\
%0& 0& 1& 0& 0&\frac{1}{400}t^3&\frac{9}{32000} t^4&\frac{63}{800000} t^5& \frac{7}{3200000}t^6\\
%0& 0& 0& 1& 0&\frac{1}{40}t^2&\frac{1}{400} t^3& \frac{21}{32000}t^4& \frac{7}{16000}t^5\\
%0& 0& 0& 0& 1&t&\frac{3}{40} t^2&\frac{7}{400} t^3&\frac{7}{16000} t^4
%\end{pmatrix}.
%\]
The vector of $5\times 5$-minors of $D(t)$ defines a parametrization by $t$ of a rational normal curve of degree $20$. So the image of $C$ in $\GG(5,\langle J\rangle)$ is a rational normal curve of degree $20$.
Therefore, the image of $\bar S$ in $\GG(5,\langle J\rangle)$ is the tangent developable of a rational normal curve of degree $20$.  
\end{proof}

\begin{remark}\label{rnc of degree 20}
The tangent developable of a rational normal curve of degree $20$ is  surface of degree $38$, it is a degeneration of a linearly normal K3-surface of degree $38$. This shows how $\VAPS(q^3,10)$ is a specialisation of $\VAPS(f,10)$ for a general ternary sextic $f$: 

The linear syzygies of $J(a)$ form a $4$-dimensional space $U(a)$ of columns in  the skew symmetric matrix $M_q$.  So seeing $M_q$ as a net of skew forms on a $9$-space $U$, the columns space $U(a)$ is identified as a  natural subspace of $U$ which is isotropic for $M_q$.  
Thus $S\to \GG(5,\langle J\rangle)$ correspond to a ``dual map" $S\to \GG(4,U)$, where the image is identified with the variety of isotropic subspaces of $M_q$ in the Grassmannian $\GG(4,9)$.
For a general skew symmetric net $M$, this variety is a K3-surface of degree $38$.  In particular the Buchsbaum-Eisenbud matrix $M_f$ of syzygies for the apolar ideal of a general ternary sextic form $f$, is a skew symmetric $(9\times 9)$-matrix.  Mukai showed that the variety $\VAPS(f,10)$ of apolar schemes of length $10$  to $f$ is isomorphic to the variety of isotropic $4$-spaces with respect to $M_f$ as a net of skew forms on a $9$-dimensional space $U$.\end{remark}

\section{Saturated and unsaturated  ideals of length ten schemes apolar to the square of a quaternary quadric} %and a quadratic polarity}
\label{main3}

%\section{Apolar schemes of length ten to the square of a quaternary %quadric}\label{main3}
In this section we first discuss unsaturated apolar ideals that are limits of ideals of apolar schemes, and then go on to consider the higher order polarity defined by the catalecticant map between quadratic differential operators and second order partial derivatives of a quartic form.  This is then applied to give first results on apolar schemes of length ten to a quaternary quartic that is the square of a quadric. 

\subsection{Nonsaturated apolar ideals and different models of Varieties of Apolar Schemes} A delicate issue for a variety $\mathrm{VAPS}(f,r)$ parameterizing apolar schemes of length $r$ to $f$ is that their compactification in the Hilbert scheme may contain schemes that are not apolar to $f$, i.e.~
%not be contained in the local Gorenstein locus.  Any minimal length apolar scheme is locally Gorenstein.  So this means, 
there may be schemes which are limits of apolar schemes but which themselves are not apolar.  We shall see that this happens for $\mathrm{VAPS}(q^2,10)$ where $q$ is a quaternary quadric, see Section \ref{affine deformation} below.   Following Buczy\'nska and Buczy\'nski we consider the compactification of an open subset set of apolar schemes of length $10$ in the multigraded Hilbert scheme $\Hilb^H$ instead of the ordinary Hilbert scheme $\Hilb^P$, cf.~\cite{BB}.  The former parameterizes ideals $I$ with a given Hilbert function $$H(n)=\dim_k(T/I)_n, n=0,1,2\dots$$ and not schemes $\Gamma$ with a given Hilbert polynomial $P$, and has the advantage that apolarity is a closed condition.  The cost is that unsaturated apolar ideals may appear as limits of saturated apolar ideals (ideals of apolar schemes).
%In our case we fix the Hilbert function $H=\{1,4,10,10...\}$ and 
We consider ideals $I\subset S$ with Hilbert function $H$ for $T/I$, and  say that an ideal $I$ is apolar to $f$ if $I\subset f^\perp\subset T$.
Thus we define
$$\mathrm{VAPS}(f,H)=\{[I]\in \Hilb^H| I\subset f^\perp\}.$$
For a fixed Hilbert function $H$, %Since we work in the multigraded Hilbert schemes, 
there may be both saturated and unsaturated ideals in $\Hilb^H$ that are apolar to $f$.  The saturated ideals form an open subscheme of $\mathrm{VAPS}(f,H)$, so the closure of the set of saturated ideals form components $\mathrm{VAPS}(f,H)$.  We are mostly interested in these components, so we denote by $\mathrm{VAPS}^{sbl}(f,H)$ the union of them, i.e. the subscheme of saturable ideals.
An {\em unsaturated limit ideal} in $\mathrm{VAPS}^{sbl}(f,H)$ is then an unsaturated apolar ideal that is a limit of saturated ideals.
  The residual components in $\mathrm{VAPS}(f,H)$, of unsaturated apolar ideals, are denoted $\mathrm{VAPS}^{uns}(f,H)$, cf. \cite{JRS}.  
%In lack of a better name we call schemes of length $r$ {\em bad limits}, if they are limits of apolar schemes of length $r$, but are not themselves apolar to $F$.  The reference to $F$ is supposedly unnderstood.
In our case there are unsaturated limit ideals.
When $H$ is the Hilbert function of $r$ general points, we simplify the notation and denote by $\mathrm{VAPS}(f,r)\subset \Hilb^H$, the subscheme $\mathrm{VAPS}^{sbl}(f,H)$ of saturable ideals.  It is a compactification of all the saturated ideals of length $r$ schemes apolar to $f$.

Before investigating global properties of $\mathrm{VAPS}(f,10)$ when $f=q^2$, we describe ideals that may appear on the boundary of $\mathrm{VAPS}(f,10)$, i.e. the unsaturated ideals in $\mathrm{VAPS}(f,H)$, where $H=\{1,4,10,10,...\}$.  In particular we give closed conditions on the schemes defined by these ideals.
\begin{lemma}\label{line or conic}
Let $f$ be a quaternary quartic that is not apolar to any quadric, and let  $[I]\in \mathrm{VAPS}(f,H), H=\{1,4,10,10,...\} $ be an unsaturated apolar ideal to $f$.  Then $\Gamma=V(I)$ is contained in a conic section or a subscheme of $\Gamma$ of length at least six that is contained in a line.
\end{lemma}
\begin{proof} We need to characterize schemes $\Gamma$ of length $10$ that impose dependant conditions on quartics.  Let us start by observing that $\Gamma$ must be contained in a quadric hypersurface. 

Let $X$ be a quadric that contains $\Gamma$, and consider the ideal $I_{X,\Gamma}$ of $\Gamma$ on $X$.  By restriction, $\Gamma$ imposes dependant conditions on quartics on $X$, and as above, also on cubics on $X$. Let $\Gamma'$ be any subscheme of length eight of $\Gamma$. Then $\Gamma'$ is contained in a quadric on $X$, and cubics that contain $\Gamma$ do not separate the residual points $\Gamma_r=\Gamma\setminus\Gamma'$.  So planes cannot separate $\Gamma_r$.  The latter has length two, so this is absurd, unless the quadric that contains $\Gamma'$ on $X$, contains all of $\Gamma$.

So $\Gamma$ is contained in a complete intersection $Y$ of two quadrics, or in the union $Y'$ of a plane $P$ and a line $L$, where $L$ is possibly embedded in $P$.

$Y$ cannot be irreducible: In case $Y$ is an irreducible complete intersection, i.e. a possibly nodal or cuspidal curve of arithmetic genus one, any scheme $\Gamma\subset Y$ of length ten impose independent conditions on quartics  $\Gamma$.

So the complete intersection $Y$ must be reducible.  $\Gamma$ imposes dependant conditions on quartics restricted to $Y$ if and only if the restriction of $\Gamma$ to one of its components imposes dependant conditions on quartics.  But the components are rational, so this means a subscheme of $\Gamma$ of length at least six lies in a line, or $\Gamma$ is contained in a possibly reducible conic.

In case $Y'$ contains a plane and line, we may again argue that the restriction to the plane or to the line must impose dependant conditions on quartics.  The line case is treated above.  In the plane case, we first note, again, that $\Gamma$ is contained in a plane cubic curve. Furthermore that this curve must be reducible, which again reduces to the two cases of a line and a conic.  The lemma follows.
\end{proof}

We shall consider a compactification of the set of apolar schemes in a Grassmannian and need the following lemma.
\begin{lemma}\label{ten cubics} Let $\Gamma$ be a scheme of length ten that is apolar to a quaternary quartic $f$, and $f$ is not apolar to any quadric, $\Gamma$ imposes independent conditions on cubics, i.e. $I_{\Gamma,3}$ is $10$-dimensional.
\end{lemma}
\begin{proof}
Assume that the scheme $\Gamma$ imposes dependant conditions on cubics.  Take any length nine subscheme of $\Gamma$.  It is contained in a quadric.  So any cubic that is a multiple of this quadric must contain $\Gamma$.  Therefore $\Gamma$ is contained in a quadric, against the apolarity assumption.
\end{proof}

For a quaternary quartic $f$ that is not apolar to any quadric, there is a natural forgetful map
\[
\pi_{\rm G}:\mathrm{VAPS}(f,10)\to \GG(10,16)=\GG(10,(f)^{\bot}_3).
\]
By Lemma \ref{ten cubics}, when $f$ is a quaternary quartic that has no quadric form in its apolar ideal, $\mathrm{VAPS}(f,10)$ contains all apolar schemes of length ten.  The ideal of any apolar scheme of length $10$ contains a $10$-dimensional space of cubic forms, so the image 
$\pi_{\rm G}(\mathrm{VAPS}(f,10))$ is a compactification of the variety of apolar schemes of length ten.

From here on we shall be concerned with this compactification. We shall denote it 
 $$\VAPS(f,10):=\pi_{\rm G}(\mathrm{VAPS}(f,10))\subset \GG(10,(f)^{\bot}_3).$$
In $\VAPS(f,10)$ we identify unsaturated limit ideals $[I]\in\mathrm{VAPS}(f,10))$ by their image $\pi_{\rm G}([i])$, i.e by their degree three part.

\subsection{Higher order polarity}\label{higher order apolarity}
Although being apolar is not a closed condition in the ordinary Hilbert scheme, there is a closed condition that any length two subscheme of an apolar scheme in $\Hilb^{10}(\PP(S_1))$  must satisfy, therefore also the 
 scheme $V(I)$ of an unsaturated limit ideal $I\in \VAPS(f,10)$.
  This condition may be seen using {\em higher order polarity}.

Let $f\in S_4$ be a quartic form such that $f^\perp$ contains no quadric.
Then $f$ defines  an isomorphism
$$\Omega_f: T_2\to S_2; \quad g\mapsto g(f),$$
and a bilinear symmetric form on $T_2\times T_2$:
$$(g_1,g_2)\mapsto\Omega_f(g_1,g_2)=g_1(g_2(f))=(g_1\cdot g_2)(f),$$
whose matrix with respect to the monomial basis of $T_2$ and the dual basis of $S_2$ is the catalecticant matrix $Cat_{2,2}(f)$.
 Since the adjoint of a catalecticant matrix, in general, is not a catalecticant matrix (cf. \cite {Dol}), the inverse $$\Omega^{-1}_f: S_2\to T_2$$ is not $\Omega_g$ for some $g\in T_4$. Still it is symmetric and defines a bilinear symmetric form on $S_2\times S_2$
$$(q_1,q_2)\mapsto\Omega^{-1}_f(q_1)(q_2).$$
%We denote the corresponding quadratic form on $\PP(S_2)$ by $q^{-1}.$  Thus, the symmetric form $\Omega^{-1}_f$ defines polarity with respect to $Q^{-1}:=\{q^{-1}=0\}$ in $\PP(S_2)$.

The restriction of $\Omega^{-1}_f$ to the rank one quadrics in $S_2$ defines a biquadratic form 
$$q_f(l_1, l_2)=\Omega^{-1}_f(l_1^2)(l_2^2)$$
on $S_1\times S_1$.
The vanishing locus of this form on $\PP(S_1)\times \PP(S_1)$ defines a
{\em quadratic polarity}, a $(2,2)$-correspondence on $\PP(S_1)\times \PP(S_1):$
$$CP_f=\{q_f(l,l')=0\}\subset \PP(S_1)\times \PP(S_1).$$ 
%$$q_f(l_1, l_2)=\Omega^{-1}_f(l_1^2)(l_2^2)$$
%on $S_1)\times \PP(S_1)$.  
%We denote by $q_l:=\Omega^{-1}_f(l^2)\in S_2$.
%Thus $q_l$ is defined by the property: $$\Omega_f(q_l)=q_l(f)=l^2.$$ 
%Consider the preimage by this map of the rank $1$ quadrics:
%$$q_f=\{(q,l)| q(f)=l^2\}\subset P(T_2)\times P(S_1),$$
%and denote by $q_l\in T_2$ the quadric form such that $q_l(f)=l^2$.  Note that $q_l$ is only defined up to scalar, as so is $f$.

\begin{remark}
When $f=(x_0x_1+x_2x_3)^2$, then 
%$$q_f=3y_3^2x_0^2-6y_2y_3x_0x_1+3y_2^2x_1^2-6y_1y_3x_0x_2+4y_1y_2x_1x_2+2y_0y_3x_1x_2
%$$
%$$+3y_1^2x_2^2+2y_1y_2x_0x_3+4y_0y_3x_0x_3-6y_0y_2x_1x_3-6y_0y_1x_2x_3+3y_0^2x_3^2).$$
 \begin{align*}
  q_f= &3y_1^2z_0^2+4y_0y_1z_0z_1-2y_2y_3z_0z_1+3y_0^2z_1^2+6y_1y_3z_0z_2+6y_0y_3z_1z_2+\\&3y_3^2z_2^2+6y_1y_2z_0z_3+6y_0y_2z_1z_3-2y_0y_1z_2z_3+4y_2
        y_3z_2z_3+3y_2^2z_3^2.
        \end{align*}

\end{remark}

%$$CP_f=\{q_f(l,l')=0\}\subset \PP(S_1)\times \PP(S_1).$$ 
 A pair $([l],[l'])\in \PP(S_1)\times \PP(S_1)$ is called {\em conjugate} if it is in $CP_f$, and we call $CP_f$ the {\em conjugate pair variety} (cf. \cite[Section 1.4.2]{Dol}).

 The fiber over $[l]$ under each projection $CP_f\to \PP(S_1)$ is a quadric surface.  It is defined by $q_l:=\Omega^{-1}_f(l^2)\in S_2$. 
In particular, the form $q_l$ is defined, up to nonzero scalar, by the property: $$\Omega_f(q_l)=q_l(f)=l^2.$$ 

Notice that 
$$([l],[l'])\in CP_f\Leftrightarrow q_l((l')^2)=q_{l'}(l^2)=0.$$
%Notice that $q_f$ is symmetric: 
%$$q_{l'}(l^2)=0 \quad {\rm iff}\quad q_{l'}q_l(f)=0 \quad {\rm iff}\quad  q_l((l')^2)=0$$ 
%Any length $2$ subscheme in $Q_f$
%The Hilbert scheme $H_2(\PP(S_1))$ of subschemes of length $2$, is obtained by first symmetrizing $\PP(S_1)\times \PP(S_1)$, and then blowing up the diagonal.

The intersection with the diagonal $CP_f\cap\Delta\subset\PP(S_1)\times \PP(S_1)$ is a quartic surface.

 \begin{definition}\label{inverse} The restriction of $q_f$ to the diagonal defines a quartic form $$f^{-1}(l^4):=q_f(l^2,l^2),$$
 $f^{-1}\in T_4$, that we call the inverse of $f$.
 \end{definition}

 The quartic surface $\{f^{-1}=0\}\subset \PP(S_1)$ may be characterized as
 $$\{f^{-1}=0\}=\{[l]\in \PP(S_1)|q_l(l^2)=0\}$$

%The open set of distinct  conjugate pairs $([l],[l'])\in CP_f$ maps $2:1$ to the Hilbert scheme of length two subschemes $z_2$.  
%We denote by 
%$$H_{2,f}=\overline{\{\{[l],[l']\}|[l]\not=[l'],q_l((l')^2)=q_{l'}%(l^2)=0\}}\subset Hilb_2(\PP(S_1)$$ the closure in the Hilbert scheme of %conjugate pairs.
% Notice that if $z_2\in H_{2,f}$ and $[l]\in z_2$, then $q_l((l')^2)=0$, when $[l']$ is residual to $[l]$ in $z_2$.
%, denoted by $H_{2,f}$, 
%consists of the length two subschemes $z_2\subset \PP(S_1)$, 
% such that whenever $[l]\in z_2$, then $z_2\subset\{q_l=0\}$. 
% {\color{red} This does not sound correct to me.  For $Z$ nonreduced this just depends on the support and in fact means that $[l]$ is in the inverse quartic. But there is a 2 dimensional family of schemes supported on $[l]$ and only a curve can be in the closure of $q_{[l]}$}  AGREE.  

\begin{lemma}\label{residual in ql}
Let $\Gamma$ be a scheme of length ten apolar to $f$. Assume that $[l]\in \Gamma$ is a point. Let $\Gamma'$ be the scheme residual to $[l]$ in $\Gamma$. Then $\Gamma'\subset \{q_l=0\}$.
In particular, $\Gamma$ is nonreduced at $[l]$, if and only if $[l]\in \{f^{-1}=0\}$.
\end{lemma}
\begin{proof}
%The ideal of $\Gamma'$ is the colon ideal $I_{\Gamma'}=I_\Gamma:l^{\bot}$, and contains a unique quadric $q'$ since $\Gamma'$ has length nine, and $\Gamma$ is contained in no quadrics.  In particular, $q'$ is the unique quadric that contains the scheme residual to $[l]$ in any apolar scheme $\Gamma$ of length ten containing $[l]$. 
%Therefore $l^\bot\cdot q'\subset I_\Gamma$ for any such scheme $\Gamma$.

Consider the quadric $q_l$, defined by $q_l(f)=l^2$, and assume that $q_l(l^2)\not= 0$. Consider the pencil of quartics $f-\lambda l^4$ for $\lambda\in \mathbb C$. The  catalecticant matrices $Cat_{2,2}(f-\lambda l^4)$ of quartics in this pencil then form a pencil of $(10\times 10)$ matrices of which the general is of rank ten whereas a special, namely the catalecticant of $l^4$,  is of rank one.  It follows that in the pencil there is a unique quartic with catalecticant rank nine. We conclude that in the pencil there is a unique quartic apolar to some quadric and that the quadric is then also unique. Let us denote the corresponding constant by $\lambda_c$ and the apolar quadric by $q_c$. 

Now, on one hand we know that $q_l(f-\lambda_l l^4)=0$ for some scalar $\lambda_l$. In fact, $\lambda_l$ is then given by the formula $\lambda_l q_l(l^4)=l^2$. It follows that $\lambda_c=\lambda_l$ and $q_c=q_l$. 

On the other hand $\Gamma'$, as a scheme of length 9 is contained in a quadric $q'$ such that $q'(f)$ is anihilated by $l^{\bot}$ and hence $q'(f)=\lambda_1 l^2$ for some $\lambda_1\in \mathbb C$. We deduce that $q'$ is apolar to $f- \lambda_2 l^4$ for some $\lambda_2\in \mathbb C$ and hence by the uniqueness above that $\lambda_2=\lambda_c=\lambda_l$ and $q'=q_c=q_l$. We conclude that $\Gamma'\subset \{q_l=0\} $
%Therefore $q_l$ must coincide with $q'$.
  %The correspondance $[l],\{q_l=0\}$ defined by $q_f$ above therefore has the property of the lemma 
whenever $q_l(l^2)\not=0$. 

Consider now the incidence 
$$\{([l],\Gamma,q')|[l]\in \Gamma, \Gamma\setminus [l]\subset \{q'=0\}\}\subset \PP(S_1)\times {\rm VAPS}(f,10)\times \PP(T_2).$$
When $q_l(l^2)\not=0$, the pair $([l],q_l)$ is in the image of the projection onto the first and third factor. But the set of pairs $([l],q_l)$ is irreducible, isomorphic to $\PP(S_1)$.   Hence $([l],q_l)$ and $([l],q')$ coincides also when $q_l(l^2)=0$.

Finally, $\Gamma$ is nonreduced at $[l]$ if and only if $q_{l'}(l^2)=0$ for every point $[l']$ in the support of $\Gamma$, which reduces to the condition that $q_l(l^2)=0$. 
\end{proof}
%\begin{lemma}\label{localresidual in ql}
%Let $Z$ be a  scheme of length $10$ apolar to $f$. Assume that %$z_2\subset Z$ 
%Assume that $z_2\in H_{2,f}$ 
%is a local scheme of length at least $2$ supported at $[l]$. Then %$q_l(l^2)=0$.
%\end{lemma}
%\begin{proof}
%Since $z_2\subset Z$ is local and supported at $[l]$, the point $[l]$ %is contained in the scheme $Z'$ residual to $[l]$ in $Z$.   

%\end{proof}
\begin{lemma}\label{z_2 in gamma} A pair of points $\{[l],[l']\}\subset \PP(S_1)$ is contained in an apolar subscheme of length ten to $f$ if and only if it is a conjugate pair.
\end{lemma}
\begin{proof} %On the other hand, since $q_l(f)=l^2, q_l(f-\lambda_l l^4)=0$ for a unique scalar $\lambda_l$. In fact $\lambda_l q_l(l^4)=l^2$.

%Notice that if $q_l(f)=l^2$, then $q_l(f-\lambda_l l^4)=0$ for a unique scalar $\lambda_l$. In fact $\lambda_l q_l(l^4)=l^2$. 
Assume first that $\{[l],[l']\}$ is contained in the apolar scheme $Z$ of length ten.  By Lemma \ref{residual in ql}, $Z\setminus [l]\subset \{q_l=0\}$ and $Z\setminus [l']\subset \{q_{l'}=0\}$, so $q_l((l')^2)=q_{l'}(l^2)=0$ and $([l],[l'])$ is a conjugate pair.

On the other hand, if $([l],[l'])$ is a conjugate pair, and both $q_l(l^2)\not=0$ and $q_{l'}((l')^2)\not=0$, then there are nonzero scalars $\lambda$ and $\lambda'$, such that 
$$q_l(f-\lambda_l l^4)=q_l(f-\lambda'_{l'} (l')^4)=0.$$
But then both $q_l$ and $q_{l'}$ are apolar to $f_{l,l'}=f-\lambda_l l^4-\lambda_{l'} (l')^4$, and so $f_{l,l'} $ is apolar to a subscheme $\Gamma'$ of length eight (\cite[Proposition 5.1.5]{KKRSSY}).  This means that $f$ is apolar to $\{[l],[l']\}\cup \Gamma'$, a scheme of length $10$.
When $q_l(l^2)=0$ or $q_{l'}((l')^2)=0$ we argue by taking limits as in the proof of Lemma \ref{residual in ql}.
\end{proof}

\subsection{First results on saturable ideals apolar to $q^2$} We apply the quadratic polarity from Section \ref{higher order apolarity} to describe saturable ideals apolar to $q^2$.

Let now $$f=q^2=(x_0x_1+x_2x_3)^2,$$
then direct computation shows 
$$f^{-1}=(y_0y_1+y_2y_3)^2.$$ 
We set $q^{-1}=y_0y_1+y_2y_3$ and $Q^{-1}=\{q^{-1}=0\}\subset\PP(S_1)$, similarly $$Q=\{q=0\}\subset\PP(T_1).$$

 Our aim in this section is to prove the following caharacterisation of ideals corresponding to points of $\VAPS(q^2,10)$:
 
\begin{theorem} Let I be an ideal corresponding to a point $[I]\in \VAPS(q^2,10)$. Then one of the following is true:
\begin{itemize}
    \item The ideal $I$ is saturated and defines a scheme of length 10 that is the union of a local  scheme of length at least eight supported on a point $p\in Q^{-1}$ and a scheme of length at most 2 supported on the tangent plane to $Q^{-1}$in $p$ but outside $Q^{-1}$.
    \item The ideal $I$ is not saturated and $V(I)$ contains a line in  $Q^{-1}$.
\end{itemize}
\end{theorem}

Notice that $q^{-1}$ is the inverse quadric of $q=x_0x_1+x_2x_2$; the polar map $$q^{-1}: S_1\to T_1; l\mapsto l(q^{-1})$$ is inverse to 
$q: T_1\to S_1$.  For $l\in S$ we let $$p_l=l(q^{-1})\quad \text{or equivalently,}\quad p_l(q)=l.$$
In particular, $\{p_l=0\}$ is the polar plane to $[l]$ w.r.t $Q^{-1}$.
\begin{lemma}\label{polar quadrics}
Let $l\in S_1$ and $q_l\in T_2$ such that $q_l(q^2)=l^2$, then 
$q_l\in \langle q^{-1}, p_l^2\rangle$.  In particular,
$\{q_l=0\}$ is tangent along its intersection with $Q^{-1}$, i.e. along the conic section $\{p_l=q^{-1}=0\}$ where $\{p_l=0\}$ is the polar plane to $[l]$ w.r.t $Q^{-1}$.
\end{lemma}
\begin{proof}
First, notice that $q^{-1}(q^2)=6q$ and that $p_l=l(q^{-1})\in T_1$.
Then 
$$p_l^2(q^2)=2(p_l(q))^2+2(p_l^2(q))q$$ 
But $p_l(q)=l$, so 
$p_l^2(q^2)=2l^2+2(p_l^2(q))q=2l^2+\frac{2}{6}(p_l^2(q))q^{-1}(q^2)$, and therefore
$$(p_l^2-\frac{1}{3}(p_l^2(q))q^{-1})q^2=2l^2=2q_l(q^2).$$
In particular 
$$q_l=\frac{1}{2}p_l^2-\frac{1}{6}(p_l^2(q))q^{-1}.$$
\end{proof}
\begin{lemma}\label{q_l irred}
Let $l\in S_1$ and $q_l\in T_2$ such that $q_l(q^2)=l^2$, then $q_l$ has rank at least three, unless $[l]\in Q^{-1}$, i.e. $q^{-1}(l^2)=0$, in which case $q_l=\frac{1}{6}p_l^2$.
\end{lemma}
\begin{proof}  The quadratic form $q^{-1}$ is nonsingular, so any quadric in the pencil $\langle q^{-1},p_l^2 \rangle$ except $p_l^2$ is irreducible.
\end{proof}

We apply the Lemma \ref{polar quadrics} to the intersection of the two quadrics $q_{l_1}$ and $q_{l_2}.$
\begin{lemma}\label{double polarity}
Let $[l_1],[l_2]\in \PP(S_1)\setminus Q^{-1}$ and set $q_{l_i}\in T_2$ and $p_{l_i}\in T_1$ by $q_{l_i}(q^2)=l_i^2$ and $p_{l_i}(q)=l_i$, for $i=1,2$. 
Then the complete intersection $\{q_{l_1}=q_{l_2}=0\}$ is the union of two plane conics that intersect in $p_{l_1}=p_{l_2}=q^{-1}=0$.  The planes of the two conics are distinct from $\{p_{l_1}=0\},\{p_{l_2}=0\}$.

When $[l_1]\in Q^{-1}$, then $q_{l_1}=p_{l_1}^2$, and so the intersection $\{q_{l_1}=q_{l_2}=0\}$ is a double conic.
\end{lemma}
\begin{proof} The complete intersection $\{q_{l_1}=q_{l_2}=0\}$ is singular at the points $p_{l_1}=p_{l_2}=q^{-1}=0$, so it is reducible and consists of two conic  sections. The degeneration when $[l_1]\in Q^{-1}$ follows immediately.
Finally, if  $p_{l_1}=0$ contains a component of 
$\{q_{l_1}=q_{l_2}=0\}$, then $\{p_{l_1}^2,q_{l_1},q_{l_2}\}$ generate a net of quadrics that contain the quadric $q^{-1}$ and $p_{l_2}^2$. And all quadrics of the net vanish  on the component $\{p_{l_1}=q_{l_1}=q_{l_2}=0\}$. But then  $p_{l_1}=p_{l_2}$, a contradiction.
\end{proof}
\begin{lemma}
%In the notation of lemma \ref{double polarity}, 
Let $[l_1],[l_2]\in \PP(S_1)$ and assume that $q_{l_1}(l_2^2)=0$ and equivalently $q_{l_2}(l_1^2)=0$.  
Then $p_{l_2}(l_1)=p_{l_1}(l_2)=0$ if and only if $[l_1]\in Q^{-1}$ or $[l_2]\in Q^{-1}$ .  
\end{lemma}
\begin{proof}We have, $$0=q_{l_1}(l_2^2)=\frac{1}{2}(p_{l_1}^2(l_2^2))-\frac{1}{6}(p_{l_1}^2(q))(q^{-1}(l_2^2))=\frac{1}{2}(p_{l_1}(l_2))^2-\frac{1}{6}(p_{l_1}(l_1))(p_{l_2}(l_2)),$$
so $q_{l_1}(l_2^2)=0$ implies $p_{l_1}(l_2)=0$ if and only if $p_{l_1}(l_1)=0$ or $p_{l_2}(l_2)=0$. Therefore, the lemma follows by polarity w.r.t $Q^{-1}$.
\end{proof}
Any two points $[l_1]$ and $[l_2]$ lie in some tangent plane to $Q^{-1}$, namely the ones containing the polar line $p_{l_1}=p_{l_2}=0$.
\begin{lemma}\label{restriction to Tp}
 The restriction of the $(2,2)$-correspondence to a tangent plane is a $(2,2)$-correspondence on the pencil of tangent lines.
 Explicitly, in the tangent plane $\{y_3=0\}$ to $Q^{-1}$ at $[x_2]$, the correspondence $q_l\mapsto l^2\in \{y_3=0\}$ is given by
  %      $$(3a_1^2y_0^2+4a_0a_1y_0y_1-%2a_2a_3y_0y_1+3a_0^2y_1^2+6a_1a_3y_0y_2+6a_0a_3y_1y_2
 %  +3a_3^2y_2^2)(q^2)$$
%   $$=6(a_0y_0+a_1y_1+a_2y_2)^2$$
 %      $$(3a_1^2y_0^2+4a_0a_1y_0y_1+3a_0^2y_1^2
 %  )(q^2)$$
%   $$=6(a_0y_0+a_1y_1+a_2y_2)^2$$
%$$(3a_2^2y_1^2+4a_1a_2y_1y_2+3a_1^2y_2^2+3a_3^2y_0^2-6a_2a_3y_0y_1-%6a_1a_3y_0y_2+2a_1a_2y_0y_3)(q^2)=$$
%$$6(a_0y_0+a_1y_1+a_2y_2)^2.$$
\begin{align*}
&(3a_1^2y_0^2+4a_0a_1y_0y_1+3a_0^2y_1^2+6a_1a_2y_0y_3+6a_0a_2y_1y_3-2a_0a_1y_2y_3+3a_2^2y_3^2)(q^2)=\\
&6(a_0x_0+a_1x_1+a_2x_2)^2.\end{align*}

% $$((3a_2^2y_1^2+8a_1a_2y_1y_2+3a_1^2y_2^2+3a_3^2y_0^2-12a_2a_3y_0y_1-%12a_1a_3y_0y_2+4a_1a_2y_0y_3)(q^2)=$$
% 12(a_1y_1+a_2y_2+a_3y_3)^2.$$
 The inverse correspondence composed with the restriction to $\{y_3=0\}$ define up to scalar, the quadratic map
 $$r_{y_0}: S_1\to T_2;  (a_0x_0+a_1x_1+a_2x_2)\mapsto ((3a_1^2y_0^2+4a_0a_1y_0y_1+3a_0^2y_1^2).$$

\end{lemma}
\begin{proof}
Use the Lemma \ref{polar quadrics}, or a direct computation.
\end{proof}

\begin{lemma}\label{two points in tangent plane} If two points $l_1,l_2$ lie in the tangent plane $T_p$ to $Q^{-1}$, then $$\{q_{l_1}=q_{l_2}=0\}\cap T_p$$ is supported at the point of tangency $p$, unless $l_1$ and $l_2$ are collinear with the $p$. If furthermore $q_{l_1}(l_2^2)=0$, then $l_1$ and $l_2$ are collinear with $p$ only if both lie on $Q^{-1}$.
    \end{lemma}
    \begin{proof} We may assume $p=[x_2]$ and $T_p=\{y_3=0\}$ and $$l_1=a_0x_0+a_1x_1+a_2x_2, l_2=b_0x_0+b_1x_1+b_2x_2.$$
        Notice that the restriction $q_{l_1}=y_3=0$ is a pair of lines through the tangency point:  $$\{q_{l_1}=3a_1^2y_0^2+4a_0a_1y_0y_1+3a_0^2y_1^2=0\}.$$  Also, this pair of lines depends only on the line $\{a_1y_0-a_0y_1=0\}$ between $l_1$ and $[x_2]$. Furthermore 
        $3a^2y^2+4ay+3=3b^2y^2+4by+3=0 $ have a common zero only if $a=b$, so the first part of the lemma follows.  For the second part it suffices to notice that if $l_1$ and $l_2$ are collinear with $[x_2]$, then $(b_0:b_1)=(a_0:a_1)$ and $$q_{l_1}(l_2^2)=
        3a_1^2(a_0^2)+4a_0a_1(a_0a_1)+3a_0^2(a_1^2)=10a_0^2a_1^2$$
        %3a_2^2(a_1^2)+4a_1a_2(a_1a_2)+3a_1^2(a_2^2)=10a_1^4a_2^4=0$$ 
        if and only if $a_0=0$ or $a_1=0$, which means $l_1$ and $l_2$ lie on $Q^{-1}$.
    \end{proof}
\begin{proposition}\label{2+8} Let $[l_1]$ and $[l_2]$ be a pair of distinct conjugate points outside $Q^{-1}$, and let $L_{12}=\{p_{l_1}=p_{l_2}=0\}$
be the line polar with respect to $Q^{-1}$ to both $[l_1]$ and $[l_2]$. 
 
 Let $\Gamma$ be a scheme of length ten apolar  to $q^2$ that contains $[l_1]$ and $[l_2]$. Then the residual subscheme $\Gamma\setminus \{[l_1],[l_2]\}$ is a scheme of length eight supported at one of the two points on $L\cap Q^{-1}$.
 %on the line polar with respect to $Q^{-1}$ to both $[l_1]$ and $[l_2]$.
% Explicitely, in the tangent plane $\{x_0=0\}$ to $Q^{-1}$ at $[y_3]$, the correspondence $q_l\mapsto l^2\in \{x_0=0\}$ is given by
% $$(3a_2^2x_1^2+4a_1a_2x_1x_2+3a_1^2x_2^2+3a_3^2x_0^2-6a_2a_3x_0x_1-6a_1a_3x_0x_2+2a_1a_2x_0x_3)(q^2)=$$
% $$6(a_1y_1+a_2y_2+a_3y_3)^2.$$
% %$$((3a_2^2x_1^2+8a_1a_2x_1x_2+3a_1^2x_2^2+3a_3^2x_0^2-12a_2a_3x_0x_1-%12a_1a_3x_0x_2+4a_1a_2x_0x_3)(q^2)=$$
% $$12(a_1y_1+a_2y_2+a_3y_3)^2.$$
% The inverse correspondence composed with the restriction to $\{x_0=0\}$ define up to scalar, the quadratic map
% $$r_{x_0}: S_1\to T_2;  (a_1y_1+a_2y_2+a_3y_3)\mapsto (3a_2^2x_1^2+4a_1a_2x_1x_2+3a_1^2x_2^2)$$
\end{proposition}
\begin{proof} 
Assume first that $\Gamma$ intersects $Q^{-1}$.
By Lemma \ref{z_2 in gamma}, the residual scheme $\Gamma\setminus \{[l_1],[l_2]\}$ is contained in 
the intersection
$\{q_{l_1}=q_{l_2}=0\}$, which, by Lemma \ref{polar quadrics}, intersects $Q^{-1}$ in $$L_{12}\cap Q^{-1}=\{p(l_1)=p(l_2)=q^{-1}=0\}.$$
Let $[l],[l']=L_{12}\cap Q^{-1}$
be these two points.  Then $q_l$ and $q_{l'}$ define the double tangent planes to $Q^{-1}$ at $[l]$ and $[l']$.
If the line spanned by $[l_1]$ and $[l_2]$ is not tangent to $Q^{-1}$, then $q_l((l')^2)\not=0$ and so $\Gamma$ can only have support at one of the points $[l]$ and $[l']$, say $[l]$.  In this case $\{q_l=0\}$ is the tangent plane to $Q^{-1}$ at $l$ and $[l_1]$ and $[l_2]$ lies in this tangent plane.  By Lemma \ref{two points in tangent plane}, $\Gamma\setminus [l_1],[l_2]$ is supported only at $[l]$.

If the line spanned by $[l_1]$ and $[l_2]$ is tangent to $Q^{-1}$, then $[l]=[l']$ is the tangency point.  But $[l_1]$ and $[l_2]$ are conjugate, so according to Lemma \ref{two points in tangent plane},  the two points $[l_1],[l_2]$ and the tangency point $[l]$ are collinear only if $[l_1],[l_2]\in Q^{-1}$, against our assumption.

    If $\Gamma$ does not intersect $Q^{-1}$, it consists of ten distinct points, i.e. $\Gamma$ is nonsingular. This is excluded by the computations performed in Section 5, however below we present an independent argument using the tools introduced in this section.

\begin{lemma}\label{noreducedZ}
There are no nonsingular  schemes of length ten apolar to $q^2$.
\end{lemma}
\begin{proof}
%Let $\Gamma$ be an apolar scheme of length ten to $q^2$. First, if $\Gamma$ intersects $Q^{-1}$, say in $[l]$, then $q_l=p_l^2$ and $\Gamma$ is nonreduced at $[l]$ by Lemma \ref{residual in ql}. In fact, the scheme $\Gamma'$ residual to $[l]$ in $\Gamma$ is contained in $\{p_l^2=0\}$, so is nonreduced unless it is contained in the plane $\{p_l=0\}$.
Assume by contradiction that $\Gamma=\{[l_1],...,[l_{10}]\}$ is a scheme apolar to $q^2$.
By Lemma \ref{residual in ql} the scheme $\Gamma$ does not intersect $Q^{-1}$. In particular, by Lemma \ref{q_l irred}, the quadric $q_{l_i}$ is irreducible for every $[l_i]\in \Gamma$. 
Let $p_{l_{i}}:=l_i(q^{-1}),\ i=1,..,10$ be the polars of the $l_i$ with respect to $Q^{-1}$, and let $$L_{i,j}=\{p_{l_{i}}=p_{l_{j}}=0\}$$ be the polar line to $\langle[l_i],[l_{j}]\rangle,$ $i<j$. By Lemma \ref{double polarity}, $\{q_{l_i}=q_{l_j}=0\}$ is a pair of conic sections that lie in a pair of planes $\pi_{i,j},\pi_{i,j}'$ that intersect along $L_{i,j}$.  Furthermore, by Lemma \ref{z_2 in gamma}, the eight points $\Gamma\setminus \{[l_i],[l_{j}] \}$ are contained in these two conics away from the line $L_{i,j}$.  
%In fact $4$ points in each plane.
%Let $[l_k]$ be one of these eight points. 

\textbf{Claim 1:} No three points in $\Gamma$ are colinear.

Indeed, assume by contradiction that the points $[l_i],[l_{j}],[l_k]$ are collinear.  Then the polar lines $L_{i,j}, L_{i,k}, L_{j,k}$ coincide and the remaining seven points $\Gamma\setminus\{[l_i],[l_{j}],[l_k]\}$ lie in $\{q_{l_i}=q_{l_j}=q_{l_k}=0\}$, the intersection of three irreducible quadrics.
They also have to lie in three pairs of planes all through the polar line $L_{i,j}$, so either the three pairs of planes have a common plane and the seven points lie in this plane, or all seven points lie in $L_{i,j}$. In either case, the ten points $\Gamma$ lie in a quadric, which means it cannot be apolar to $q^2$. 

\textbf{Claim 2:} No five points in $\Gamma$ are coplanar.

Indeed, assume by contradiction that the set of five points 
$$K=\{[l_i],[l_{j}],[l_k],[l_m],[l_n]\}\subset \Gamma$$ lie in a plane.
Then the polars $p_{l_{i}}, p_{l_{j}}, p_{l_{k}}, p_{l_{m}}, p_{l_{n}}$ of points in $K$ with respect to $Q^{-1}$ intersect in a point $v_{ijkmn}$, and any three, say $\{[l_k],[l_m],[l_n]\}$ of the points together with the  five points in $\Gamma\setminus K$ lie in a rank two quadric $q_{k,m,n}$ singular at $v_{ijkmn}$, away from the polar line $L_{ij}$.  Note that $\Gamma\setminus K$ cannot lie in a plane, since then $\Gamma$ would lie in a quadric. By Claim 1 the minimal number of lines through $v_{ijkmn}$ whose union contains $\Gamma\setminus K$ is three.  So there is at most a net of quadric surfaces singular at $v_{ijkmn}$ that contains $\Gamma\setminus K$. In particular the rank two quadrics $q_{k,m,n}$ above, all lie in this net.
We infer that for any three points of $K$, there is a quadric in the net that contains these three and $\Gamma\setminus K$. It follows that the number of conditions $K$ imposes on the net of quadrics is at most 2 and hence $\Gamma$ has to lie in a quadric in the net against the apolarity assumption.

%Since no three of the points $[l_i],..,l_[m]$ are collinear, the ten polar lines $L_{i,j},..,L_{m,n}$ are distinct. 

We may therefore assume that no three of the ten points in $\Gamma$ are collinear and no five are coplanar. 
%In particular $[l_i],[l_{j}],[l_k]$ are not collinear.
% The remaining $7$ lie in $\{q_{l_i}=q_{l_j}=q_{l_k}=0\}$, the intersection of three irreducible quadrics.
%By Lemma \ref{double polarity}, they lie in three pairs of distinct planes, each pair meet in a line, and all six pass through the point 
%$$v_{i,j,k}=\{p_{l_{i}}=p_{l_{j}}=p_{l_{k}}=0\}.$$
Let $\pi_{ij},\pi_{ij}'$ be the pair of planes in the pencil defined by $q_{l_i},q_{l_j}$, then we may assume that the  eight points $\Gamma\setminus \{[l_i],[l_j]\}$ lie four in $\pi_{ij}$ and four in $\pi_{ij}'$. 
Consider the four points in $\pi_{ij}\cap \Gamma$, and assume they are $[l_1],...,[l_4]$.  
Then the polars of these four points with respect to $Q^{-1}$ intersect in a point $v_{1234}$, and any two of these points together with the remaining six in $\Gamma \setminus \pi_{ij} $ lie in a rank two quadric singular at $v_{1234}$. 
Consider the minimal set of lines through $v_{1234}$ whose union contains $\Gamma \setminus \pi_{ij} $, and the system of quadrics singular at $v_{1234}$ that contain them. Then we have three cases to consider:

\begin{enumerate}
    \item The minimal number of lines whose union contains $\Gamma \setminus \pi_{ij} $ is five or more, then the system of quadrics singular at $v_{1234}$ and containing these lines consists of just one quadric and any pair of points $\Gamma \cap \pi_{ij} $ must lie on that quadric. Hence $\Gamma$ has to lie on a quadric, a contradiction.

\item The minimal number of lines whose union contains $\Gamma \setminus \pi_{ij} $ is four, then the system of quadrics containing these lines is a pencil. Any pair of points in $\Gamma \cap \pi_{ij} $ lies in a quadric in this pencil, but that means that any two points in $\Gamma \cap \pi_{ij} $ give dependent conditions on the pencil.  That means that all $\Gamma \cap \pi_{ij} $ give only one condition on the pencil of quadrics that contain $\Gamma \setminus \pi_{ij} $ and implies that $\Gamma$ is contained in a quadric, a contradiction.
 
\item The minimal number of lines whose union contains $\Gamma \setminus \pi_{ij} $ is four. Let $L$ be the union of three lines, each containing two points of $\Gamma \setminus \pi_{ij} $.
The pair of planes that contains all of $\Gamma\setminus\{[l_1],[l_2]\}$, then contains $L$ and the two points $[l_3],[l_4]$.  Since no five points are coplanar, one of the lines in $L$ and $[l_3],[l_4]$ are coplanar.  Similarly any two points $\pi_{ij}\cap \Gamma$ are coplanar with some line of $L$.  This is possible, only if the three lines form a cone over the intersections of diagonals of a complete quadrangle formed by the projections of $\pi_{ij}\cap \Gamma$ from $v_{1234}$.
\end{enumerate}

So we are left with the following situation: for any pair of points in $\Gamma$, the remaining eight lie four and four in two planes. And for each such four-tuple, the pole of the plane is the intersection point of three lines through the remaining six points, and these lines are the intersections of planes containing two of the four points.

Altogether, there are $90$ poles, and three lines through each that contain two points among the ten. So, there is at least one line $L_{12}$ containing two points, say $[l_1], [l_2]$, of the ten points in $\Gamma$, and this line contains at least six such poles. For each pole, there are four points, such that the remaining six lie in pairs on three lines through the pole. One such pair is $[l_1], [l_2]$, and there are six pairs of lines that meet on $L_{12}$, with each pair containing four of the remaining eight points. Thus, there are twelve lines, grouped into six pairs, and each pair spans a plane. Therefore, at least one of the remaining eight points lies on three of the twelve lines. 

However, in this case, this point either lies on $L_{12}$ or two of the three lines span a plane that contains $L_{12}$. In the first case, $L_{12}$ would contain three points of $\Gamma$, and in the second case, the plane would contain five points of $\Gamma$. Both cases are excluded by the Claims.
\end{proof}
The proof of Proposition \ref{2+8} is then complete.
\end{proof}
\begin{proposition}\label{apolar schemes} A scheme $\Gamma$ of length ten apolar  to $q^2=(x_0x_1+x_2x_3)^2$ is the union of a local scheme of length at least eight supported at a point $[l]$ on the inverse quadric $Q^{-1}$ and a scheme of length one or two in the tangent plane at $[l]$ but outside the inverse quadric.
%, or it is a scheme supported in at most two points on a line in the inverse quadric. 
\end{proposition}
\begin{proof} If $\Gamma$ contains exactly one point $[l_1]$ outside $Q^{-1}$, and $[l]\in Q^{-1}$ is in the support of $\Gamma$, then $\{q_l=q_{l_1}=q^{-1}=0\}$ is a local scheme supported at $l$. 
On the other hand, if $\Gamma$ is supported in more than one point and only on $Q^{-1}$, then each point lies in the tangent planes to $Q^{-1}$ at the others.  This holds only if all points lie in a line in $Q^{-1}$. It remains to show that there are at most one such point.   This is verified in computations, see Lemma \ref{supportonline} and Proposition \ref{local at p}. \end{proof}

\begin{corollary}\label{unsaturatedinline} An unsaturated apolar limit ideal of a scheme of length ten, is contained in the ideal of a line in $Q^{-1}$.
\end{corollary}
\begin{proof}  By Lemma \ref{line or conic}, an unsaturated apolar ideal of a scheme of length ten contains the ideal of a line or a conic in $Q^{-1}$.  But by Proposition \ref{apolar schemes} any apolar scheme is supported in a tangent plane to $Q^{-1}$.  This tangent plane intersects $Q^{-1}$ in two lines.
\end{proof}
\subsection{Ranks of powers of quadrics}
Any apolar scheme of length ten to $q^2$, when $q$ is a rank four quadric, or to $q^3$ when $q$ is of rank three,  is singular by Proposition \ref{apolar schemes} (resp. Proposition \ref{main2}), so in either case any smooth apolar scheme has length at least eleven.  
\begin{proposition}\label{rankq2}
If $q^2$ is the square of a nondegenerate quadric polynomial in four variables it has rank $11$.      
\end{proposition}

\begin{proof} The rank of $q^2$ is the minimal length of a smooth apolar scheme.   According to the Apolarity lemma \cite[Lemma 1.15]{IK}, a smooth apolar scheme of finite length $r$ defines a decomposition of $f$ as a sum of $r$ powers of linear forms.

From Proposition \ref{apolar schemes} %and the apolarity Lemma 
we know that the rank of $q^2$ is at least eleven. To prove that it is equal to eleven we construct a smooth apolar scheme of length eleven. 
We proceed as follows. We start with a general point $[l_1]\in \mathbb P^3$. Note that it belongs to the open orbit under the automorphism group preserving the quadric.  The quadric $\{q_{l_1}=0\}$ is a smooth quadric totally tangent to $\{q^{-1}=0\}$ and $q_{l_1}$ is apolar to $q^2-\lambda_1 l_1^4$ for some $\lambda_1\in \mathbb C$ . If 
we then take a general point $[l_2]\in q_{l_1}$. Then $\{q_{l_1}= q_{l_2}=0\}$ is the union of two conics intersecting in the points of intersection of $\{q^{-1}=0\}$ with the line apolar to $[l_1]$, $[l_2]$. The intersection $\{q_{l_1}= q_{l_2}=0\}$ is then apolar to $q^2-\lambda_1 l_1^4-\lambda_2 l_2^4$ for some $\lambda_2\in \mathbb C$. For a general point  $[l_3]$ in one of the two conics, the intersection $\{q_{l_1}= q_{l_2}= q_{l_3}=0\}$ is a scheme of length 8 which is apolar to $q^2-\lambda_1 l_1^4-\lambda_2 l_2^4-\lambda_3 l_3^4$
for some $\lambda_3\in \mathbb C$. In this way we obtained a scheme of length eleven apolar to $q$. It remains to observe that the scheme is smooth.  

For this it is enough to give an explicit example, we find one defined over the finite field $\ZZ/101$:
Let $q=x_0x_1+x_2x_3$
and let $l_1=x_0+x_1-x_2-x_3$, then $$q_{l_1}=50y_0^2-34y_0y_1+50y_1^2+y_0y_2+y_1y_2+50y_2^2+y_0y_3+y_1y_3-34y_2y_3+50y_3^2.$$

Since $q_{l_1}(2:-3:-34:1)=0$, we 
 let $l_2=2x_0-3x_1-34x_2+x_3$, and find 
  $$q_{l_2}=46y_0^2-41y_0y_1-2y_1^2+3y_0y_2-2y_1y_2+50y_2^2-y_0y_3-33y_1y_3-13y_2y_3+28y_3^2$$
  Finally, since 
  $$q_{l_1}(-25:23:-43:1)=q_{l_2}(-25:23:-43:1)=0,$$
  we let $l_3=-25x_0+23x_1-43x_2+x_3$ and find
  $$q_{l_3}=-12y_0^2-35y_0y_1+41y_1^2-23y_0y_2+25y_1y_2+50y_2^2-21y_0y_3+36y_1y_3+39y_2y_3+35y_3^2.$$
 It is now straightforward to check that the complete intersection $$\{q_{l_1}=q_{l_2}= q_{l_3}=0\}$$ is smooth, and together with the three points $$(1:1:-1:-1),(2:-3:-34:1)\quad {\rm and}\quad(-25:23:-43:1)$$ form a smooth scheme of length eleven that is apolar to $q^2$.

 This example is defined over $\ZZ/101$.  The property for a subscheme of length eleven to be smooth is open, so the example means that a general apolar scheme over $\QQ$ of length eleven, restricts to a smooth one over $\ZZ/101$, so must therefore be smooth also over $\QQ$.
    \end{proof}

    A similar argument shows that 
    \begin{corollary}\label{rankq3}
If $q^3$ is the third power of a nondegenerate quadric polynomial in three variables, it has rank $11$.      
\end{corollary}
\begin{proof} By Proposition \ref{main2} the rank is at least eleven.
We proceed as above and start with a general point $[l_1]\in \mathbb P^2$. Note that it belongs to the open orbit under the automorphism group preserving the quadric.  The cubic form $q_{l_1}$ is a smooth cubic and $q_{l_1}$ is apolar to $q^3-\lambda_1 l_1^6$ for some $\lambda_1\in \mathbb C$ . If 
we then take a general point $[l_2]\in \{q_{l_1}=0\}$, then $\{q_{l_1}= q_{l_2}=0\}$ is a complete intersection, nine points of intersection that are apolar to $q^3-\lambda_1 l_1^6-\lambda_2 l_2^6$ for some $\lambda_2\in \mathbb C$. Together with $[l_1],[l_2]$ they form a scheme of length eleven apolar to $q^3$.
  
  It remains to observe that this scheme is smooth.  For this it is enough to give an explicit example: 
  Let $l_1=x_0$, then $q_{l_1}=y_0^3-3y_0y_1y_2$.
  Next, let $l_2=3x_0+3x_1+x_2$, then
  $$q_{l_2}=y_0^2y_1+y_0y_1^2+\frac{2}{9}y_1^3+3y_0^2y_2+3y_0y_1y_2+9y_0y_2^2+6y_2^3.$$
  It is then straightforward to check that the complete intersection $\{q_{l_1}=q_{l_2}=0\}$ is smooth, and together with $(1:0:0)$ and $(3:3:1)$ form a smooth scheme of length eleven that is apolar to $q^3$.
\end{proof}

\section{An affine deformation; elements, fibers and components %of $\VAPS(q^2,10)$
}\label{affine deformation}
\subsection{An affine unfolding}
Recall from  Proposition \ref{apolar schemes} that any apolar scheme of length ten to $q^2$ has a component of length at least eight at a point on $Q^{-1}$, while the residual scheme is disjoint from $Q^{-1}$.  
Let us, therefore, consider the set of apolar schemes of length ten
 with support containing a point $p\in Q^{-1}$. 
 %A local apolar scheme is Gorenstein so it is defined by the apolar ideal of a polynomial. 
 For the computations in this section we set  
 $$q=x_0^2-x_1^2+x_2x_3,$$
 then 
 $$q^{-1}=\frac{1}{2}y_0^2-\frac{1}{2}y_1^2+y_2y_3,$$
 and choose the point $p=(0:0:0:1)\in Q^{-1}= \{q^{-1}=0\}$
 \begin{remark}
    In the next subsections we change to the quadric
    $$q=x_0x_1+x_2x_3.$$ and the point $p=(0:0:1:0)$.
    The transformation in $y-$coordinates that preserves apolarity and  maps the old $p$ to the new $p$ is
    $$(y_0,y_1,y_2, y_3)\mapsto (y_1-y_0,y_1+y_0,y_3,y_2).$$
\end{remark}

We consider the variety $V^{aff}_{p}\subset \VAPS(q^2,10)$ of apolar schemes of length ten that contains the point $p=(0:0:0:1)$ and does not intersect the plane $\{y_3=0\}$.
The main theorem in this section is the following.
\begin{theorem}\label{main4} 
The reduced structure of the variety $V^{aff}_{p}$ has three $3$-dimensional components. One is a cone over a complete intersection $(2,2,2)$ surface, and the two others are $3$-dimensional affine spaces.
\end{theorem}
 \begin{proof}
 This proof depends on a number of computations in Macaulay 2, \cite{Macaulay_2} that are presented in the Appendix.
 We consider a so-called tautological scheme, a local apolar scheme as in Section \ref{sec1}.  
 We twist the quartic $q^2$ after dehomogenizing by setting $x_3=1$ and get  
 $$x_0^4-2x_0^2x_1^2+x_1^4+2x_0^2x_2-2x_1^2x_2+2x_2^2.$$
 Then the apolar ideal of this inhomogeneous polynomial is generated by 
 \begin{align*}y_0^3-6y_0y_2,&y_0^2y_1-2y_1y_2,y_0y_1^2+2y_0y_2,y_1^3+6y_1y_2,\\ &y_0^2y_2-y_2^2, y_0y_1y_2, 
 y_1^2y_2+y_2^2,y_0y_2^2, y_1y_2^2,y_2^3.
 \end{align*}
 Its homogenization is an ideal of a scheme $\Gamma_0$ supported in the point $p=(0:0:0:1)$ on $Q^{-1}$.  Clearly, $\Gamma_0\in V^{aff}_{p}$.  We shall find the latter as an affine deformation of $\Gamma_0$.
%We consider an affine deformation of $\Gamma_0$:

The apolar ideal $(q^2)^\bot$ has the following minimal set of generators:
\begin{align*}J=(&y_0^3-6y_0y_2y_3,y_0^2y_1-2y_1y_2y_3,y_0y_1^2+2y_0y_2y_3,\\ &y_1^3+6y_1y_2y_3, y_0^2y_2-y_2^2y_3, 
 y_0y_1y_2, y_1^2y_2+y_2^2y_3,y_0y_2^2,\\
& y_1y_2^2,y_2^3,
y_0^2y_3-y_2y_3^2, y_0y_1y_3,y_1^2y_3+y_2y_3^2,y_0y_3^2,y_1y_3^2,y_3^3).
\end{align*}  
The Betti table for its minimal free resolution is
$$\begin{matrix}&&&&\cr
    1 & - & - & - & - \cr
    - & - & - & - & - \cr
                                               - & 16 & 30 & 16 & - \cr
                                                 - & - & - & - & - \cr
                                                  - & - & - & - & 1 \cr
                                               \cr
                                             \end{matrix} $$   
The ideal of $\Gamma_0$ is generated by 
\begin{align*}J(0)=(&y_0^3-6y_0y_2y_3,y_0^2y_1-2y_1y_2y_3,y_0y_1^2+2y_0y_2y_3,y_1^3+6y_1y_2y_3,  \\
& y_0^2y_2-y_2^2y_3, y_0y_1y_2,y_1^2y_2+y_2^2y_3,y_0y_2^2, y_1y_2^2,y_2^3),
\end{align*}
i.e. the first $10$ of the $16$ generators of the apolar ideal. The Betti table for its resolution is 
$$\begin{matrix}&&&&\cr
    1 & - & - & - & - \cr
    - & - & - & - & - \cr
                                               - & 10 & 15 & 6 & - \cr
                                               \cr
                                             \end{matrix} $$   
The minimal resolution of the initial ideal of $\Gamma_0$ with respect to the reverse lexicographical order has, in fact, the same Betti table.

Consider any apolar scheme of length ten with support at $p$ that does not intersect $\{y_3=0\}$.  
Since the ideal of the scheme contains no quadric, the initial ideal with respect to the reverse lexicographic order coincides with that of $\Gamma_0$.  In particular, it belongs to 
the affine deformation $V_p^{aff}$ of $I_{\Gamma_0}$ obtained by adding scalar multiples of the last five generators of $(q^2)^\bot_3$ that vanish at $p$, to the ten generators of the ideal of $\Gamma_0$ and ask that they have $15$ linear syzygies.   
By adding only generators of $(q^2)^\bot_3$, we preserve apolarity. 
The syzygy condition suffices to get  a flat family of ideals \cite[Proposition 3.1]{artin_deform_of_sings}. So the deformation $V_p^{aff}$ of $I_{\Gamma_0}$ parameterizes apolar ideals $I$ such that $V(I)$ has length ten.  
Now, if $I\in V_p^{aff}$ is not saturated, then by Lemma \ref{line or conic},  $V(I)$ is contained in a conic or has a subscheme of length at least six contained in a line, so the cubics in $I$ vanishes on a line or a conic, a contradiction.  We conclude that $I\in V_p^{aff}$ if and only if $I$ is the ideal of an apolar scheme $\Gamma$ of length ten that contains $p$ and does not intersect $\{y_3=0\}$.

Explicitly, the ideal of $\Gamma$ is generated by the ten forms
%\begin{tiny}
$$J(a)=\begin{pmatrix}f_{000}\\
f_{001}\\
f_{011}\\
f_{111}\\
f_{002}\\
f_{012}\\
f_{112}\\
f_{022}\\
f_{122}\\
f_{222}\\
\end{pmatrix}=
\begin{pmatrix}y_{0}^3 -6y_0y_2y_3\\
y_{0}^2y_1 -2y_1y_2y_3\\
y_{0}y_1^2 +2y_0y_2y_3\\
y_1^3 +6y_1y_2y_3\\
y_0^2y_2 -y_2^2y_3\\
y_0y_1y_2\\
y_1^2y_2+y_2^2y_3\\
y_0y_2^2\\
y_1y_2^2\\
y_2^3\end{pmatrix}
+\begin{pmatrix}a_{00}\ldots a_{40}\\
a_{01}\ldots a_{41}\\
a_{02}\ldots a_{42}\\
\vdots \quad\quad\quad  \vdots\\
a_{09}\ldots a_{49}\end{pmatrix}\cdot 
\begin{pmatrix}y_0^2y_3-y_2y_3^2\\ y_0y_1y_3\\y_1^2y_3+y_2y_3^2\\y_0y_3^2\\y_1y_3^2%\\y_3^3
\end{pmatrix}$$
%   \end{tiny}
with the algebraic condition that the ideal has $15$ linear syzygies.

Here the generators $J(a)$ of the ideal of $\Gamma$ are obtained from the $(15\times 1)$-column matrix $J$ with entries the $15$ first generators of 
$(q^2)^\bot_3$ in the above order and the $(10\times 1)$-column matrix $J(0)$ with entries the generators of 
$I_{\Gamma_0}.$  Let $D(a)=(I_{10}|A_{5})$ be the concatenation of the identity $(10\times 10)$-matrix and  the $(10\times 5)$-matrix
$$A_5=(a_{ij}), \quad 0\leq i\leq 4,0\leq j\leq 9 .$$
Then the generators $J(a)$ form the above $(10\times 1)$-column matrix  $$J(a)=D(a)\cdot J.$$ 

For the syzygy condition we consider an unfolding;  an extension of the syzygies of the generators of $I_{\Gamma_0}$ to those of $I_\Gamma$.
%is a $(10\times 1)$-column matrix whose entries are the cubic forms $f_{ijk}$ listed above, with coefficients in the $a_{ij}$.  
Let $H(0)$ be the $(15\times 10)$-matrix of linear syzygies for $J(0)$, i.e. $H(0)\cdot J(0)=0$.  
Like in Section \ref{ternaryquadric}, we find an extension of $H(0)$ to syzygies for $J(a)$ 
by adding multiples of $y_3$:
%$H(0)$ is the syzygy matrix for $G(0)$, the initial terms in $G(a)$. In particular 
The initial monomials in $J(a)$ are the initial monomials in $J(0)$. They are all monomials in $y_0,y_1,y_2$, while the added monomials are all multiples of $y_3$.  Therefore, whenever $y_0,y_1$ or $y_2$ appear in a syzygy for $J(a)$, they appear as part of a syzygy in $H(0)$. In particular, each syzygy for $J(a)$ is a syzygy for $J(0)$ possibly with some terms added that are multiples of $y_3$. 

To see which multiples to add we consider the product $H(b)\cdot J(a) $, where $H(b)$ is the sum of $H(0)$ and a general $(15\times 10)$-matrix with entries $b_{i,j}y_3$. %Here $C_0$ is the matrix:
%\begin{tiny}
%\[
%\begin{pmatrix}
%-y_1&y_0&0&0&0&-4y_{3}&0&0&0&0\\
%0&-y_1&y_0&0&2y_3&0&-2y_3&0&0&0\\
%0&0&-y_1&y_0&0&4y_3&0&0&0&0\\
%-y_2&0&0&0&y_0&0&0&-5y_3&0&0\\
%0&-y_2&0&0&y_1&0&0&0&-y_3&0\\
%0&0&0&0&-y_1&y_0&0&0&-y_3&0\\
%0&0&-y_2&0&0&y_1&0&-2y_3&0&0\\
%0&0&0&0&0&-y_1&y_0&y_3&0&0\\
%0&0&0&-y_2&0&0&y_1&0&-5y_3&0\\
%0&0&0&0&-y_2&0&0&y_0&0&-y_3\\
%0&0&0&0&0&-y_2&0&y_1&0&0\\
%0&0&0&0&0&0&0&-y_1&y_0&0\\
%0&0&0&0&0&0&-y_2&0&y_1&-y_3\\
%0&0&0&0&0&0&0&-y_2&0&y_0\\
%0&0&0&0&0&0&0&0&-y_2&y_1\\
%\end{pmatrix}
%\]\end{tiny}
The fifteen entries of the product $H(b)\cdot J(a)$ are polynomials that are multiples of $y_3$. 
The monomials in these polynomials that are linear in $y_3$ have coefficients of the form $b_{ij}-l_{ij}$, where $l_{ij}$ is linear in $a_{kl}$.  Therefore, $H(b)$ is a matrix of syzygies only if each $b_{ij}=l_{ij}$.  See the Appendix \ref{unfolding32} for details.

So by adding these appropriate multiples of $y_3$ in the corresponding syzygy, we get a syzygy matrix  
$H(a)$ from $H(0)$ %by adding multiples of $y_3$ to the entries 
such that the entries of the product $H(a)\cdot J(a)$ are polynomials that are multiples of $y_3^2$, see Appendix \ref{unfolding32}.  The matrix $H(a)$ is an unfolding of $H(0)$, and the coefficients in $a_{ij}$ of all entries in $H(a)\cdot J(a)$ generate the ideal $I_{unfold}$ 
of an affine scheme in $\AAA^{50}$.
%that parameterizes ideals of apolar schemes of length $10$.
The generators of the ideal $I_{unfold}$ are linear and quadratic in the $a_{ij}$.
In Macaulay2 \cite{Macaulay_2} we find that the ideal $I_{unfold}$ is generated by $39$ linear and $15$ quadratic polynomials, see Appendix \ref{unfolding32}. 
Furthermore, its radical $I^{rad}_{unfold}$ is generated by $42$ linear and $12$ quadratic polynomials.  In particular, $V_p^{aff}:=V(I^{rad}_{unfold})\subset \AAA^{8}$, with coordinates 
$$a_{12},a_{13},a_{20},a_{21},a_{22},a_{23},a_{42},a_{43}$$
inherited from $\AAA^{50}$.
After substitution, i.e. in these coordinates, the above matrix $A_5$ is
%{\tiny
\[
A_5=
%\begin{pmatrix}a_{00}\ldots a_{40}\\
%a_{01}\ldots a_{41}\\
%a_{02}\ldots a_{42}\\
%\vdots \ldots \vdots\\
%a_{09}\ldots a_{49}\end{pmatrix}=
\begin{pmatrix}-10a_{13}+5a_{20}-24a_{22}&-5a_{12}-12a_{21}&a_{20}&-a_{43}& a_{42}\\
2a_{12}+5a_{21}&-5a_{13}-12a_{22}&a_{21}&-a_{42}&a_{43}\\
2a_{13}+5a_{22}&a_{12}&a_{22}&-a_{43}&a_{42}\\
(-2a_{12}+a_{23})/5&a_{13}&a_{23}&-a_{42}&a_{43}  \\
-a_{43}/4&a_{42}/2&-a_{43}/4&0&0\\
-a_{42}/4&a_{43}/2&-a_{42}/4&0&0\\
-a_{43}/4&a_{42}/2&-a_{43}/4&0&0\\
0&0&0&0&0\\
0&0&0&0&0\\
0&0&0&0&0\end{pmatrix}
\]
%}

By a Macaulay 2 \cite{Macaulay_2} computation described in the appendix, the affine set $V_{p}^{aff}$ has three irreducible components, $V_{p}^{aff}(0), V_{p}^{aff}(1), V_{p}^{aff}(2)$, the first one is a $3$-dimensional cone over a complete intersection $(2,2,2)$ and the other two, $V_{p}(1), V_{p}(2)$ are $3$-dimensional affine spaces. The equations for the three components are given below in Proposition \ref{componentequations}.
\end{proof}

Our Macaulay 2 \cite{Macaulay_2} computations reveal explicit equations for the components of $V_{p}^{aff}$.

%In the above coordinates the described computation yields the following 
\begin{proposition}\label{componentequations} In the coordinates $$a_{12},a_{13},a_{20},a_{21},a_{22},a_{23},a_{42},a_{43}$$
\begin{enumerate}
    \item $V_p^{aff}(0)$ is the intersection of the linear space $V(a_{42},a_{43})$ and the complete intersection  of three quadrics 
%  {\tiny  
  \begin{align*}
V(&
2a_{13}a_{20}-a_{12}a_{21}-5a_{21}^2+5a_{13}a_{22}+5a_{20}a_{22}+13a_{22}^2-a_{21}a_{23},\\
&2a_{12}a_{20}+5a_{13}a_{21}+5a_{12}a_{22}+25a_{21}a_{22}-a_{20}a_{23},\\
   &5a_{12}^2+25a_{13}^2+24a_{12}a_{21}+180a_{13}a_{22}+300a_{22}^2-12a_{21}a_{23}).
\end{align*}
%}
Furthermore, $V_p^{aff}(0)$ is singular along the cone $C_0$ over a rational normal curve of degree five, defined by the $(2\times 2)$-minors of

%{\tiny 
$$\begin{pmatrix}
a_{12}&10a_{13}+25a_{22}&a_{20}&a_{21}&a_{22}\\
a_{13}&-2a_{12}+a_{23}&a_{21}&a_{22}&a_{23}
\end{pmatrix}.$$
%}
\item $V_{p}^{aff}(1)$ is the linear space 
%{\tiny 
\begin{align*}
 V(&a_{42}-a_{43},a_{21}-2a_{22}+a_{23},a_{20}-3a_{22}+2a_{23},\\
 &3a_{13}+4a_{22}+2a_{23},3a_{12}+10a_{22}-4a_{23})
  \end{align*}
  %}
  \item $V_{p}^{aff}(2)$ is the linear space
 %{\tiny 
 \begin{align*}
  V(&a_{42}+a_{43},a_{21}+2a_{22}+a_{23},a_{20}-3a_{22}-2a_{23},\\
  &3a_{13}+4a_{22}-2a_{23},3a_{12}+10a_{22}-4a_{23})
  \end{align*}
  %}
\end{enumerate}
\end{proposition}

%\begin{remark}
%$V_p^{aff}(0)$ is singular along the cone $C_0$ over a rational normal curve of %degree $5$, defined by the $2\times 2$ minors of

%{\tiny 
%$$\begin{pmatrix}
%a12&10*a13+25*a22&a20&a21&a22\\
%a13&-2*a12+a23&a21&a22&a23
%\end{pmatrix}$$
%}
%\end{remark}
%On $V_{p}^{aff}(1)$, the matrix $A_5$ reduces to 
%\[
%A_5=
%\begin{pmatrix}-10a_{13}+5a_{20}-24a_{22}&-5a_{12}-12a_{21}&a_{20}&-a_{43}& a_{43}\\
%2a_{12}+5a_{21}&-5a_{13}-12a_{22}&a_{21}&-a_{43}&a_{43}\\
%2a_{13}+5a_{22}&a_{12}&a_{22}&-a_{43}&a_{43}\\
%(-2a_{12}+a_{23})/5&a_{13}&a_{23}&-a_{43}&a_{43}  \\
%-a_{43}/4&a_{43}/2&-a_{43}/4&0&0\\
%-a_{43}/4&a_{43}/2&-a_{43}/4&0&0\\
%-a_{43}/4&a_{43}/2&-a_{43}/4&0&0\\
%0&0&0&0&0\\
%0&0&0&0&0\\
%0&0&0&0&0\end{pmatrix}
%\]
%so that 
%\begin{tiny}
%$$G(a)=
%\begin{pmatrix}y_{0}^3 -6y_0y_2y_3\\
%y_{0}^2y_1 -2y_1y_2y_3\\
%y_{0}y_1^2 +2y_0y_2y_3\\
%y_1^3 +6y_1y_2y_3\\
%y_0^2y_2 -y_2^2y_3\\
%y_0y_1y_2\\
%y_1^2y_2+y_2^2y_3\\
%y_0y_2^2\\
%y_1y_2^2\\
%y_2^3\end{pmatrix}
%+
%\begin{pmatrix}-10a_{13}+5a_{20}-24a_{22}&-5a_{12}-12a_{21}&a_{20%}&-a_{43}\\
%2a_{12}+5a_{21}&-5a_{13}-12a_{22}&a_{21}&-a_{43}\\
%2a_{13}+5a_{22}&a_{12}&a_{22}&-a_{43}\\
%(-2a_{12}+a_{23})/5&a_{13}&a_{23}&-a_{43}  \\
%-a_{43}/4&a_{43}/2&-a_{43}/4&0\\
%-a_{43}/4&a_{43}/2&-a_{43}/4&0\\
%-a_{43}/4&a_{43}/2&-a_{43}/4&0\\
%0&0&0&0\\
%0&0&0&0\\
%0&0&0&0\end{pmatrix}\cdot 
%\begin{pmatrix}y_0^2y_3-y_2y_3^2\\ %y_0y_1y_3\\y_1^2y_3+y_2y_3^2\\y_0y_3^2-y_1y_3^2%\\y_3^3
%\end{pmatrix}$$
%   \end{tiny}

Applying only the linear forms in the ideal of $V_p^{aff}(0)$, we get that the cubic forms $J(a)$ that generate the apolar ideals in this component  are 
\begin{tiny}
$$J(a)=
\begin{pmatrix}y_{0}^3 -6y_0y_2y_3\\
y_{0}^2y_1 -2y_1y_2y_3\\
y_{0}y_1^2 +2y_0y_2y_3\\
y_1^3 +6y_1y_2y_3\\
y_0^2y_2 -y_2^2y_3\\
y_0y_1y_2\\
y_1^2y_2+y_2^2y_3\\
y_0y_2^2\\
y_1y_2^2\\
y_2^3\end{pmatrix}
+
\begin{pmatrix}-(10a_{13}-5a_{20}+24a_{22})&-(5a_{12}+12a_{21})&a_{20}\\
2a_{12}+5a_{21}&-(5a_{13}+12a_{22})&a_{21}\\
2a_{13}+5a_{22}&a_{12}&a_{22}\\
-(2a_{12}-a_{23})/5&a_{13}&a_{23}  \\
0&0&0\\
0&0&0\\
0&0&0\\
0&0&0\\
0&0&0\\
0&0&0\end{pmatrix}\cdot 
\begin{pmatrix}y_0^2y_3-y_2y_3^2\\ y_0y_1y_3\\y_1^2y_3+y_2y_3^2%\\y_0y_3^2\\y_1y_3^2%\\y_3^3
\end{pmatrix}.$$
   \end{tiny}
These equations show the following.
\begin{lemma}\label{supportonline}
An apolar scheme $\Gamma_a$ of length ten, with $a\in V_p^{aff}(0)$ intersects $Q^{-1}$ only at $p$. 
\end{lemma}
\begin{proof}
We may assume that the apolar scheme has support at $p=(0:0:0:1).$ 
Any other point in the support must lie in the tangent plane to $Q^{-1}$ at $p$.  So if the other point lies on $Q^{-1}$, it lies in one of the two lines on this quadric through $p$.
On $$Q^{-1}=\left\{\frac{1}{2}y_0^2-\frac{1}{2}y_1^2+y_2y_3=0\right\}$$ these lines are  $\{y_2=y_0-y_1=0\}$ and $\{y_2=y_0+y_1=0\}$.
Restricted to the line $\{y_2=y_0-y_1=0\}$ in $Q^{-1}$, the generators $J(a)$ of $I_{\Gamma_a}$ reduce to 

%\begin{tiny}
$$J(a)=
\begin{pmatrix}y_{0}^3 \\
y_{0}^3\\
y_{0}^3\\
y_0^3 \\
\end{pmatrix}
+
\begin{pmatrix}-(10a_{13}-5a_{20}+24a_{22})&-(5a_{12}+12a_{21})&a_{20}\\
2a_{12}+5a_{21}&-(5a_{13}+12a_{22})&a_{21}\\
2a_{13}+5a_{22}&a_{12}&a_{22}\\
-(2a_{12}-a_{23})/5&a_{13}&a_{23}  \\
\end{pmatrix}\cdot 
\begin{pmatrix}y_0^2y_3\\ y_0^2y_3\\y_0^2y_3
%\\y_0y_3^2\\y_1y_3^2%\\y_3^3
\end{pmatrix}$$
%   \end{tiny}
So the apolar scheme $\Gamma_a$ has support at a point on this line outside $p=(0:0:0:1)$ only if the sum of the rows in the right hand matrix have a constant nonzero value.  The ideal generated by the differences of the row sums in the right hand matrix is generated by the three linear forms:
$$a_{20}-3a_{21}+3a_{22}-a_{23},9a_{13}-7a_{21}+26a_{22}-a_{23},9a_{12}+5a_{21}+20a_{22}-7a_{23}.$$
Substituting for $a_{12}, a_{13}, a_{20}$ into the complete intersection that defines $V_p^{aff}(0)$, we get that this complete intersection modulo the ideal generated by these linear forms is generated by the single quadratic form
$$(a_{21}-2a_{22}+a_{23})^2.$$
So on the complete intersection $V_p^{aff}(0)$, the three linear forms vanish on the linear space defined  by the four linear forms
$$a_{21}-2a_{22}+a_{23},a_{20}-3a_{22}+2a_{23},
3a_{13}+4a_{22}+2a_{23},3a_{12}+10a_{22}-4a_{23}.$$
But any of the above row sums, say the third row sum, $2a_{13}+a_{12}+6a_{22}$, lies in the span of these four linear forms:
$$
a_{12}+2a_{13}+6a_{22}=\frac{1}{3}(3a_{12}+10a_{22}-4a_{23})
+\frac{2}{3}(3a_{13}+4a_{22}+2a_{23}),
$$
 so when there is a common row sum, it is always zero.  This means that any apolar scheme with support at $p$ contained in the ideal of the line $\{y_2=y_0-y_1=0\}$, is local with support only at $p$.

Similarly, restricted to the line $\{y_2=y_0+y_1=0\}$ in $Q^{-1}$, the equations $J(a)$ reduce to 
%\begin{tiny}
$$J(a)=
\begin{pmatrix}y_{0}^3 \\
-y_{0}^3\\
y_{0}^3\\
-y_0^3 \\
\end{pmatrix}
+
\begin{pmatrix}-(10a_{13}-5a_{20}+24a_{22})&-(5a_{12}+12a_{21})&a_{20}\\
2a_{12}+5a_{21}&-(5a_{13}+12a_{22})&a_{21}\\
2a_{13}+5a_{22}&a_{12}&a_{22}\\
-(2a_{12}-a_{23})/5&a_{13}&a_{23}  \\
\end{pmatrix}\cdot 
\begin{pmatrix}y_0^2y_3\\ -y_0^2y_3\\y_0^2y_3
%\\y_0y_3^2\\y_1y_3^2%\\y_3^3
\end{pmatrix}$$
%   \end{tiny}
   The argument as above applies here also, with the same conclusion.
%So the apolar scheme has at most one point on this line outside $p$. 
\end{proof}

In Proposition \ref{local at p} we show that any apolar scheme in the two families $V_p^{aff}(1)$ and $V_p^{aff}(2)$ is local with support at $p$.
%A more detailed computational analysis, that we give in Appendix \ref{V0schemes} proves the following 

\begin{remark}Computations in Macaulay 2 \cite{Macaulay_2} also show that the ideal \( I_{\text{unfold}} \) generated by the entries in the matrix product $H(a){\cdot} J(a)$, whose radical is the ideal of \( V_{p}^{\text{aff}} \), defines a non-reduced structure with multiplicity 4 along \( V_{p}^{\text{aff}}(0) \), and multiplicity 5 along each of \( V_{p}^{\text{aff}}(1) \) and \( V_{p}^{\text{aff}}(2) \).
Moreover, the quadratic equations in \( I^{\text{red}}_{\text{unfold}} \) are not in the ideal of 2-minors of \( A_5 \), so even the reduced structure on \( V_{p}^{\text{aff}} \) is not a linear section of the Grassmannian \( \GG(10,((q^2)^\perp)_3) \).
The component \( V_{p}^{\text{aff}}(0) \) is a cone over a surface that is the union of a quadratic pencil of secant lines to \( C_0 \).
In the pencil, exactly two secant lines are tangent lines to \( C_0 \). These tangent lines are contained in \( V_{p}^{\text{aff}}(1) \) and \( V_{p}^{\text{aff}}(2) \), one in each.

\end{remark}

%\vskip 1cm
%WHAT DOES A GENERAL SCHEME %IN EACH COMPONENT LOOK LIKE?
%\vskip 1cm

Since we are interested in the global structure of $\VAPS(q^2)$ we need to consider an appropiate compactification of $V_p^{aff}$.
\begin{definition}
    For $p\in Q^{-1}$ let us denote by $V_p\subset \VAPS(q^2,10)$ the closure of the variety parameterizing schemes of length ten apolar to $q^2$ and containing $p$. 
\end{definition}

\begin{proposition}\label{three components local}
    The variety $V_p$ has exactly three components $V_p(0)$,$V_p(1)$  and $V_p(2)$, which are compactifications of $V_p^{aff}(0)$,$V_p^{aff}(1)$  and $V_p^{aff}(2)$ respectively.
\end{proposition}
\begin{proof}
 Observe that since  $V_p$ is closed and contains $V_p^{aff}(0)$,$V_p^{aff}(1)$  and $V_p^{aff}(2)$ it has at least three components, namely the respective closures $V_p(0)$,$V_p(1)$  and $V_p(2)$ of $V_p^{aff}(0)$,$V_p^{aff}(1)$  and $V_p^{aff}(2)$. It remains to prove that $V_p$ has no components outside $V_p^{aff}$. Note that $V_p^{aff}$ is defined by the additional condition of not meeting the plane $y_3=0$. However one can apply on $\mathbb P^3$ an automorphism $\varphi$ preserving $Q^{-1}$ and $p$, but changing the plane $y_3=0$ to any other general plane $y'_3=0$ tangent to $Q^{-1}$. In that case image of $V_p^{aff}$ under the induced automorphism of the Grassmannian ${\rm G}$ will be isomorphic to $V_p^{aff}$, i.e. will still have three components that represent schemes apolar to $q$ that do not intersect $y'_3=0$. Since for every scheme $S$ of length ten there exists an open and dense set of planes that are tangent to $Q^{-1}$ and do not meet $S$ we conclude that $S$ is an element of one of the three components $V_p(0)$, $V_p(1)$  or $V_p(2)$.
\end{proof}
Next, for $p\in Q^{-1}$ we consider  the global structure of the compactification $V_p(0),V_p(1)$  and $V_p(2)$  of the components of $V_{p}^{aff}$ in $\GG(10,((q^2)^{\bot})_3).$

\subsection{The compactification of the main component}\label{compactificationofmain}

In this and the next subsections the arguments are still quite computational. To simplify computations we change coordinates. 
From here on we set
$q=x_0x_1+x_2x_3$, and so to get apolarity w.r.t $q^2$ with apolar schemes containing $p=(0:0:1:0)$  we use the transformation 
$$(y_0,y_1,y_2,y_3)\mapsto (y_1+y_0,y_1-y_0,y_3,y_2)$$
with the same parameters $a_{ij}$ as in the previous subsection.  The transformation maps forms apolar to $(x_0^2-x_1^2+x_2x_3)^2$ to forms apolar to $(x_0x_1+x_2x_3)^2$ and the point $\{y_0=y_1=y_2=0\}$ to the point $\{y_0=y_1=y_3=0\}$.

We start with the main component $V_p(0)$, and 
summarize our findings from the previous sections on apolar ideals in the compactification $V_p(0)$ of $V_p^{aff}(0)$:    
\begin{proposition}\label{grassmannian compactification of Vp0} Let $p\in Q^{-1}$,  then any apolar ideal $I$ in $V_p(0)$ is either unsaturated and is contained in the ideal of one of the two lines in $Q^{-1}$ through $p$, or $I$ is the saturated ideal of a scheme supported at $p$ and at most two points outside $Q^{-1}$ but in the tangent space to $Q^{-1}$ at $p$.
\end{proposition}
\begin{proof}
The case when the apolar ideal is saturated follows from Proposition \ref{apolar schemes} and Lemma \ref{supportonline}.  If the ideal $I$ is not saturated, then, by Lemma \ref{line or conic} the support $V(I)$ is contained in a conic section or has length at least $6$ in a line in $Q^{-1}$.  Since all apolar schemes in $V_p^{aff}(0)$  are supported in the tangent plane to  $Q^{-1}$ at $p$, the same holds for ideals in the compactification, which means that $V(I)$ has support in one of the two lines in $Q^{-1}$ through $p$.
\end{proof}

 %Further computations are simplified by changing coordinates. %From here on we set
With $q=x_0x_1+x_2x_3$, the inverse quadric is $Q^{-1}=\{y_0y_1+y_2y_3=0\}$ and the space of apolar cubic forms of $$q^2=(x_0x_1+x_2x_3)^2$$ is
 \begin{align*}(q^2)^\bot_3=\langle &y_3^3, y_1y_3^2, y_0y_3^2, y_1^2y_3, 2y_0y_1y_3-y_2y_3^2, y_0^2y_3, y_2^3, y_1y_2^2, y_0y_2^2,y_1^2y_2,\\ & 2y_0y_1y_2-y_2^2y_3, y_0^2y_2, y_1^3, y_0y_1^2-2y_1y_2y_3, y_0^2y_1-2y_0y_2y_3, y_0^3\rangle .
        \end{align*}
For each $p\in Q^{-1}$ we consider the following subspaces of apolar cubics:
\begin{itemize}\label{cubics at p}
\item $B_{T_p}$ is the space of apolar cubics that vanish on the tangent space $T_p$ to $Q^{-1}$ at $p$.
\item $A_p$ is the space of apolar cubics whose restriction to $T_p$ is singular at $p$.
\item $C_p$ is the space of apolar cubics whose restriction to $T_p$ has multiplicity at least three at $p$.
\end{itemize}

When $p=(0:0:1:0)$, then the tangent space $T_p=\{y_3=0\}.$  We fix this choice of point. 

\begin{remark}\label{Cp} The space $C_p$ is $10$-dimensional and generate the ideal $I_{\Gamma_0}$ of the tautological scheme $\Gamma_0$.
\end{remark}
On $V_{p}^{aff}(0)$, the complete intersection, the cubic forms generating ideals of apolar scheme of length ten are  parameterized by $(a_{12},a_{13},a_{20},a_{21},a_{22},a_{23})$.

In the new variables, 
  the $10$-space of cubic forms in an ideal in $V_{p}^{aff}(0)$ is contained in the $13$-dimensional space $A_p$ and contains the $6$-dimensional space 
 \[
B:=B_{T_p}=\langle y_{3}^3, y_{3}^2y_{0},y_{3}^2y_{1}, y_{3}y_{0}^2,y_{3}y_{1}^2, 2y_{3}y_{0}y_{1}-y_{2}y_{3}^2\rangle
\]
of apolar cubic multiples of $y_3$. So the quotient in  $ A:=A_p/B$ of each ideal in $V_{p}^{aff}(0)$ is a $4$-space of cubic forms modulo $B$. The
%that are all singular at $p$ after restriction to the plane $\{y3=0\}$,  for details.
 $7$-dimensional space $A$ is generated, modulo $B$, by
 \begin{align}\label{Alocal}
 y_{0}^3, y_{0}^2y_{2},y_{0}^2y_{1}-2y_{0}y_{2}y_{3}, 2y_{0}y_{1}y_{2}-y_{3}y_{2}^2,y_{0}y_{1}^2-2y_{1}y_{2}y_{3}, y_{1}^2y_{2},y_{1}^3,
\end{align}
 see Appendix \ref{Appendix} for details.
 
%The cubics in $a$ span a  $7$-dimensional space $A$ of apolar cubics.

\begin{lemma}\label{fibration of main component}
   The Grassmannian compactifications $V_p(0),V_{p'}(0)$, for distinct points $p,p'\in Q^{-1}$ are disjoint.
\end{lemma}
\begin{proof} Any apolar ideal $I$ in the Grassmannian compactifications of both $V_p(0)$ and $V_{p'}(0)$ contains the $5$-space of apolar cubics that contain the tangent plane $T_p$ and have multiplicity $3$ at $p$.  Furthermore, the $10$-dimensional space of cubic generators of $I$ are all singular at $p$ when restricted to the tangent plane $T_p$ and similarly for $p'$. In particular, every cubic in $I$ vanishes at both $p$ and $p'$.   By homogeneity, we may assume $p=(0:0:1:0)$.  Consider the pencil of apolar cubics  $$\langle y_0^2y_3,y_1^2y_3\rangle$$ that contain $T_p=\{y_3=0\}$ and have multiplicity $3$ at $p$ and therefore lie in $I$.   Then all cubics in this pencil are singular at the point $p'\not= p$ when restricted to $T_{p'}$, only if $y_0=y_1=0$ at $p'$, i.e. when $p'=(0:0:0:1)$. But then $T_{p'}=\{y_2=0\}$ and so $y_2^3\in I$ which does vanish at $p$, so the lemma follows.
\end{proof}
%\subsection{The compactification of $V_p(0)$ as $V_{22}$ and a $7\times 7$ skew matrix} 
%Consider the double quadric $q=(x_0x_1+x_2x_3)^2$ in this case
%\[
%a=\{y_{0}^3, y_{0}^2y_{2},y_{0}^2y_{1}-2y_{0}y_{2}y_{3}, 2y_{0}y_{1}y_{2}-%y_{3}y_{2}^2,y_{0}y_{1}^2-2y_{1}y_{2}y_{3}, y_{1}^2y_{2},y_{1}^3\}
%\]

%generate a 7-dimensional space $A$ of apolar cubics.  Additionally there is %a 6-dimensional space $B$ of cubics apolar to $q$ and which are in the %ideal generated by $y_3$.

%\[
%B=<y_{3}^3, y_{3}^2y_{0},y_{3}^2y_{1}, y_{3}y_{0}^2,y_{3}y_{1}^2, %2y_{3}y_{0}y_{1}-y_{2}y_{3}^2>
%\]

To further investigate the component $V_p(0)$, we identify the $4$-dimensional subspaces of $A$ that together with $B$ generate ideals of apolar schemes.
We do this by considering the space of linear syzygies of these ideals.

We start by noting that the following skew symmetric matrix 
\[
M=\begin{pmatrix}\label{Mp-matrix} 0& 0& 0&0& 0& 3y_{1}&-3y_{2}\cr 0& 0&0& 0& -3y_{1}&-6y_{3}&3y_{0}\cr
0& 0& 0& 2y_{1}&-y_{2}&-3y_{0}&0\cr
 0& 0&-2y_{1}& 0& 2y_{0}&0&0\cr
0& 3y_{1}&y_{2}&-2y_{0}&0&0&0\cr
-3y_{1}&6y_{3}&3y_{0}& 0& 0&0&0\cr
3y_{2}&-3y_{0}& 0&0& 0& 0&0\cr

\end{pmatrix}.
\]
 is the matrix of linear syzygies for the set of cubics in $A$ In particular, with the vector $a$ of generators of $A$ as displayed in (\ref{Alocal}), $a\cdot M=M\cdot a^t=0$.
%Note that $M$ is a syzygy matrix of $a$.
Furthermore, we have the following:
\begin{lemma}\label{syzygies a B} The set $Ay_3$ of products of cubics from $A$ with the linear form $y_3$ is contained in the ideal generated by $B$. 
\end{lemma}
\begin{proof} It is enough to observe that $B$ generates the ideal of all forms apolar to $q$ which are in the ideal generated by $y_3$.
\end{proof}
Now we introduce the following notation. If $f$ is a polynomial form in $\mathbb{C}[y_1\dots y_4]$ we define by $f'$ its restriction to the plane $\{y_3=0\}$, i.e.~the polynomial obtained from $f$ by substituting $y_3=0$. We extend this notation to matrices and sets or spaces of polynomials.
\begin{lemma} \label{lifting from tangent} The map $r: A\ni f\mapsto f'\in A'$ is an isomorphism. 
\end{lemma}
\begin{proof} It is enough to see that each of the monomials in $A'$ appears only once in $A$.
\end{proof}
Let $P\subset A$ be a $3$-dimensional subspace such that $P'\subset A'$ is an isotropic subspace with respect to $M'$.  Then $M_P=P\cdot M$ is $3$-dimensional space of syzygies, whose restriction to $A'$ are syzygies of a $4$-dimensional subspace $A_P'\subset A'$.  In fact $A_P'$ is spanned by the maximal minors of the $(3\times 7)$-matrix $P'\cdot M'$.
\begin{proposition} \label{V22mapstoVSP} Let $P$ be a $3$-dimensional subspace of $A$ such that $P'$ is isotropic with respect to $M'$, and let $A_P'$ be the space of maximal minors of $M'_P=P'\cdot M'$. 
%Let $M_P$ be the corresponding submatrix of syzygies of $A$, and let $A_P'$ be the space of maximal minors of $M'_P$.   
Set $A_P=r^{-1}(A_P')$.
Then $A_P$ together with $B$ generates a $10$-dimensional space of apolar cubic forms with a $15$-dimensional space of linear syzygies.
\end{proposition}
\begin{proof} Since $P$ is isotropic, $M'_P$ has a $(3\times 3)$-block of zeros and the maximal minors $A_P'$ and hence also $A_P$ is a $4$-dimensional subspace of $A$, which with $B$ generates a  
%by Lemma \ref{lifting from tangent}, we have $r^{-1}((P')^\perp) +B$ is a 
$10$-dimensional space of cubics.

We have the following type of syzygies:
\begin{enumerate}
\item an $8$-dimensional space of syzygies between generators of $B$
\item a $4$-dimensional space of syzygies obtained from the relation 
$$y_3 A_P\subset y_3 A \subset (y_0,y_1,y_2)B.$$

%$y_3 r^{-1}((P')^\perp)\subset y_3 A \subset <B>$
\item $3$-dimensional space of syzygies obtained from $y_3P\subset (y_0,y_1,y_2)B$ combined with the fact that $P'$ is isotropic with respect to $M'$ and $M$ is a syzygy matrix for $a$. 
\end{enumerate}
Note that these syzygies are linearly independent and span a $15$-dimensional space. Indeed, the first type are syzygies which involve only elements of B.
In the second type the coefficients at generators of $A_P$
%$(P'^{\perp})$ 
are 0 when restricted to $y_3=0$.  The three syzygies of $A_P'$ lift to syzygies of $A_P$ and $B$ since $y_3P\subset (y_0,y_1,y_2)B$. So for the third type of syzygies the coefficients at generators of $A_P$ when restricted to $y_3=0$ are always nontrivial. To show that there is no linear dependancy between the three types, assume we have a linear relation between syzygies of the three types:  $$a_1s_1+a_2s_2+a_3s_3=0$$ where $s_1$ and  $s_2$ or $s_3$ are nontrivial syzygies of the respective types, then $a_3=0$ if $s_3$ is nontrivial, because $s_3$ is the only syzygy 
being nonzero after restricting it to $y_3=0$. Furthermore, $a_2=0$ when $s_2$ is nontrivial, because $s_2$ involves some nonzero coefficient corresponding to generators of $A_P$. Finally, $a_1=0$ as $s_1$ was assumed to be nontrivial.
\end{proof}
Let $\GG(B,10, B+A)$ be the variety of $10$-dimensional subspaces of $A+B$ that contain $B$.  It is a subvariety of $\GG(10,((q^2)^{\bot})_3)$ isomorphic to  $\GG(4,7)$.
\begin{corollary}\label{matrixM'} The variety $V_p(0)$ is the intersection of $$\VAPS(q^2,10)\subset \GG(10,((q^2)^{\bot})_3)$$ with $\GG(B,10, B+A)\subset \GG(10,((q^2)^{\bot})_3)$ and is isomorphic to $V_{22}(M')\subset \GG(3,7)$ the variety of subspaces isotropic with respect to the net of skew-symmetric 2-forms  given by the matrix $M'$.
\end{corollary}
\begin{proof} Let $Syz_{10,15}\subset \GG(B,10, B+A)$ be the subset of $10$-dimensional spaces of cubics with a $15$-dimensional space of linear syzygies.

First, the space of cubics in the ideal of any apolar scheme of length ten in $V_p(0)$ belong to $Syz_{10,15}$. %$\GG(B,10, B+A)\subset \GG(10,((q^2)^{\bot})_3)$ and has a $15$-dimensional space of linear syzygies. 
In fact $V_p(0)$ is a compactification of the affine variety $V^{aff}_p(0)$ that parameterizes all subspaces of cubics in
$Syz_{10,15}$ that have the same initial monomials with respect to fixed monomial order as a subspace $I_{\Gamma_0}\in Syz_{10,15}.$  Therefore, $V_p(0)$ is a component of $Syz_{10,15}$.

From Proposition \ref{V22mapstoVSP} we know that there is an injective morphism $$\nu: V_{22}(M') \to Syz_{10,15}$$
$$P'\mapsto A'_P\mapsto A_P\mapsto [A_P+B].$$
Since $[I_{\Gamma_0}]$ belongs to the image,  $\nu (V_{22}(M'))\subset V_p(0)$.
%since any $10$-space of apolar cubics that contain $B$ and has a $15$-dimensional space of linear syzygies belongs to the $3$-fold $V_p(0)\subset \VAPS(q^2,10)$.  Since $V_{22}(M')$ is irreducible and $3$-dimensional the image must coincide with $V_p(0).$
But both are irreducible $3$-folds, so $\nu (V_{22}(M'))= V_p(0)$.
\end{proof}

\begin{remark}
To identify an unsaturated ideal in $V_p(0)$ we note first that 
%With $B$ as above 
%\[
%B=<y_{3}^3, y_{3}^2y_{0},y_{3}^2y_{1}, y_{3}y_{0}^2,y_{3}y_{1}^2, %2y_{3}y_{0}y_{1}-y_{2}y_{3}^2>
%\]
the syzygy matrix $M$ for $A$ is equivalent to 
\[
\begin{pmatrix} 0& 0& 0&0& 6y_{0}& 3y_{1}&-3y_{2}\cr 0& 0&-12y_{3}& 0& -3y_{1}&-6y_{3}&3y_{0}\cr
0& 12y_{3}& 0& 2y_{1}&-y_{2}&-3y_{0}&0\cr
 0& 0&-2y_{1}& 0& 2y_{0}&0&0\cr
-6y_{0}& 3y_{1}&y_{2}&-2y_{0}&0&0&0\cr
-3y_{1}&6y_{3}&3y_{0}& 0& 0&0&0\cr
3y_{2}&-3y_{0}& 0&0& 0& 0&0\cr

\end{pmatrix}.
\]
The isotropic $3$-space of rows $1,2,4$ defines a $4$-dimensional  subspace of cubic pfaffians, the maximal minors of 
\[
\begin{pmatrix} 0& 6y_{0}& 3y_{1}&-3y_{2}\cr -12y_{3}&  -3y_{1}&-6y_{3}&3y_{0}\cr
 -2y_{1}& 2y_{0}&0&0\cr
\end{pmatrix}
\]
that together with the space $B$ of apolar cubic forms that are multiples of $y_3$, defines an unsaturated ideal. It defines a line and an embedded scheme of length three.

The isotropic $3$-space of rows $5,6,7$ defines a similar unsaturated ideal.   
\end{remark}

\begin{remark}  The isotropic $3$-spaces for $M$ is a curve of degree 10 with three components; two double skew lines and one rational normal curve $C$ of degree 6.  Each line meet $C$ in a point.  These two points are the two apparent isotropic $3$-spaces of the previous remark. The lines are the $3$-spaces with rows one and two and a linear combination of three and four, and similarly the rows six and seven and linear combinations of four and five. The ideals on lines are unsaturated and are contained in $(y_0,y_3)$ or $(y_1,y_3)$.
An ideal on the curve $C$, except in the points of intersection with the two lines, is saturated and defines a scheme of length eight at $V(y_0,y_1,y_3)$ and a scheme of length two on the line $V(y_2,y_3)$.

The isotropic $3$-spaces to $M'$, i.e. restricting $M$ to $y_3=0$, form a (singular) degeneration of a Fano variety  $V_{22}$ (cf. \cite{Muk92}).
\end{remark}

%(HERE $S_1$
\vskip .5cm
\subsection{The compactification of the special components}\label{compactificationofspecial}
For each $p\in Q^{-1}$ we found in Proposition \ref{three components local} two components $V_{p}(1)$ and $V_{p}(2)$ besides $V_{p}(0)$ in $V_{p}$.  They are compactifications in $\GG(10,((q^2)^\perp)_3)$ of the affine $3$-folds $V_{p}^{aff}(1)$ and $V_{p}^{aff}(2)$ of local apolar schemes of length ten, cf. Proposition \ref{componentequations}. 
Here, we describe these compactifications. The two $3$-folds and their compactifications, are isomorphic under the involution exchanging the two factors in $Q^{-1}=\PP^1\times \PP^1.$ So it suffices to describe one of them, e.g. the compactification $V_{p}(1)$ of $V_{p}^{aff}(1)$.

As in our analysis of the component  $V_{p}^{aff}(0)$, we define, for each $p\in Q^{-1}$ and each line $L\subset Q^{-1}$, some subspaces of apolar cubics:
\begin{itemize}
\item $K_L$ is the space of apolar cubics that are triple along $L.$
\item $W_L$ is the space of apolar cubics that are singular along $L$.
\item $U_L$ is the space of apolar cubics that vanish along $L.$
\item $W_{L,p}$, with $p\in L$, is the space of apolar cubics in $W_L$ whose restriction to the tangent plane $T_p$ has multiplicity $3$ at $p$.
\item $U_{L,p}$ is the space of apolar cubics that are in $U_L$ or has multiplicity $3$ at $p$.
\end{itemize}

We let $p=(0:0:1:0)$ and $L=\{y_0=y_3=0\}$.
By Proposition \ref{componentequations} we may parameterize the linear space $V_{p}^{aff}(1)$ with parameters $a_{22},a_{23},a_{42}$. So  
$V_{p}^{aff}(1)$ is a $3$-dimensional affine family of ideals $I(a_{22},a_{23},a_{42})$.
Each ideal $I(a_{22},a_{23},a_{42})$ contains the $7$-space of cubic forms        
\begin{align*}
W_{L,p}=&\langle y_3^3,y_1y_3^2,y_0y_3^2,2y_0y_1y_3-y_3^2y_2,y_0^2y_3,y_0^2y_1-2y_0y_2y_3,y_0^3\rangle 
\end{align*}
%the apolar cubics that are singular along the line $L=\{y0=y3=0\}$ and whose restriction to the tangent plane $\{y3=0\}$ has multiplicity $3$  at $p=(0:0:1:0).$
and is % $I(a_{22},a_{23},a_{42})$ is 
contained in the $13$-space $U_{L,p}$. 

Observe, as a short digression, the following.
%of all apolar cubics that vanish on the line $L$ or has multiplicity $3$ at $p$. 
\begin{proposition}\label{local at p}
    Any apolar scheme $\Gamma\in V_p^{aff}(1)\cup V_{p}^{aff}(2)$ is local, supported at $p$.
\end{proposition}
\begin{proof}  The ideal $I(a_{22},a_{23},a_{42})$ of $\Gamma\in V_p^{aff}(1)$ contains $K_{L}$, so the support of $\Gamma$ is contained in $L$. Furthermore, $I(a_{22},a_{23},a_{42})$ is contained in $U_{L,p}$, so, since $\Gamma$ is finite, it is supported only at $p$.  
For $\Gamma\in V_p^{aff}(2)$ the argument is similar.
\end{proof}
Returning to the matrices representing $I(a_{22},a_{23},a_{42})$, we first notice that modulo $W_{L,p}$, the space $U_{L,p}$ is generated by
$$ y_0^2y_2,
y_1^2y_3,
y_0y_1^2-2y_1y_2y_3,
y_2^2y_3-2y_0y_1y_2,
y_0y_2^2,
y_1^3.$$
% that all vanish on the line $L$, or vanish with multiplicity $3$ at $p$.
 %when $I(a_{22},a_{23},a_{42})\in V_{p}^{aff}(1)$ $I(a_{22},a_{23},a_{42})_3/W_{L,p}$ is a $3$-space in $U_{L,p}/W_{L,p}$.
In the parameters $(a_{22},a_{23},a_{42},b)$ the $3$-space that generate $I(a_{22},a_{23},a_{42})$ is generated by the entries of the matrix product
%{\tiny
$$\begin{pmatrix}
-a_{42}&b&0&0&0&0\\
26a_{22}-14a_{23}&0&0&-4(a_{22}-a_{23})&-6a_{42}&3b\\
2a_{22}-2a_{23}&0&b&0&0&0\\
%0&2(a_{22}-a_{23})&a_{42}&0&0&0\\
\end{pmatrix}\cdot \begin{pmatrix}y_0^2y_2\\
y_1^2y_3\\
y_0y_1^2-2y_1y_2y_3\\
y_2^2y_3-2y_0y_1y_2\\
y_0y_2^2\\
y_1^3\\
\end{pmatrix}
$$
%}
when $b=1$.% i.e. when $I(a_{22},a_{23},a_{42},b)\in V_{p}^{aff}(1)$.

%In the Grassmannian of $3$-spaces in $U_{L,p}/W_{L,p}$ the $3$-spaces $I(a_{22},a_{23},a_{42},b)_3/W_{L,p}$ are computed to form the cone over a Veronese surface with vertex at the vertex plane. In fact,
The nonzero $3$-minors of the above $(3\times 7)$ coefficient matrix presenting $I(a_{22},a_{23},a_{42},b)_3/W_{L,p}$, are
%{\tiny
\begin{align*}
&p_{123}=-b^2(26a_{22}-14a_{23}), p_{124}=-8b(a_{22}-a_{23})^2,p_{125}=-12ba_{42}(a_{22}-a_{23}),\\
&p_{126}=6b^2(a_{22}-a_{23}), p_{134}=-4ba_{42}(a_{22}-a_{23}), p_{135}=-12ba_{42}^2,\\
&p_{136}=6b^2(a_{22}-a_{23}), p_{234}=4b^2(a_{22}-a_{23}),
p_{235}=6b^2a_{42},p_{236}=-3b^3.
\end{align*}
%}
They satisfy the linear relations
%{\tiny
$$p_{125}=3p_{134}, 2p_{126}=3p_{234}=2p_{136},$$
%}
and the quadratic relations expressing that the symmetric $(3\times 3)$-matrix
%{\tiny
\[
\begin{pmatrix}
3p_{124}&2p_{125}&2p_{126}\\
2p_{125}&2p_{135}&2p_{235}\\
2p_{126}&2p_{235}&2p_{236}
\end{pmatrix}
\]
%}
has rank one. These relations define a cone over a Veronese surface inside the space defined by the linear relations.  The vertex of the Veronese surface is the point where only $p_{123}$ is nonzero, which corresponds to the ideal 
\begin{align*}
       I_{ver,p}=(&y_3^3,y_1y_3^2,y_0y_3^2,2y_0y_1y_3-y_3^2y_2,y_0^2y_3,\\
       &y_0^2y_1-2y_0y_2y_3,y_0^3,y_0^2y_2,y_1^2y_3,y_0y_1^2-2y_1y_2y_3).
       \end{align*}
%$I_{\Gamma_{ver}}$.
In this cone, the variety $V_{p}^{aff}(1)$ is the locus of points with $b=1$.  The locus of points in the cone that correspond to $3$-spaces with $b=0$ are found by first dividing the minors $p_{ijk}$ by the common factor $b$.   The locus of $3$-spaces with $b=0$ are then ideals that vanish on $L$, and are parameterized by
%{\tiny
$$p_{126}=p_{235}=p_{236}=3p_{124}p_{135}-2p_{125}^2=0.$$
%}
This is a cone over a conic.  Since these ideals are contained in the ideal of $L$, they are unsaturated ideals in $V_{p}(1)$.  We summarize these findings in the following proposition.

\begin{proposition} \label{grassmannian compactification of Vp1}
       The Grassmannian compactification $V_{p}(1)$ is a cone over a Veronese surface with vertex the unsaturated ideal $I_{ver,p}$.
 The unsaturated ideals in $V_p(1)$ are contained in the ideal of $L=\{y_0=y_3=0\}$ and they form 
 the cone over a conic in the Veronese surface.
\end{proposition}
\begin{proof} It remains to prove the last part, i.e., that all unsaturated ideals in \( V_p(1) \) are contained in the ideal of \( L \).
For this, notice that the cubics in an ideal in \( V_p(1) \) are contained in the 13-dimensional space \( U_{L,p} \), in which only the last basis element \( y_1^3 \) vanishes at \( p \) and not on \( L \).
However, any unsaturated limit ideal in \( V_p(1) \) is contained in the ideal of a line (see Corollary \ref{unsaturatedinline}), so for ideals in \( V_p(1) \), this line must be \( L \).
\end{proof}
%{\tiny
%\begin{align*}%{cc}
% I(a_{22},a_{23},a_{42},b)=&
%(  y3^3,y1*y3^2,y0*y3^2,2*y0*y1*y3-y3^2*y2,y0^2*y3,y0^2*y1-%2*y0*y2*y3,y0^3,\hfill \notag\\
%&b*y1^2*y3-a42*y0^2*y2,\hfill \notag\\
%&3*b*y1^3+(26*a22-14*a23)*y0^2*y2+(8*a22-8*a23)*y0*y1*y2-\\
%&6*a42*y0*y2^2-(4*a22-4*a23)*y3*y2^2,\hfill \notag\\
%&b*(y0*y1^2-2*y1*y2*y3)+(2*a22-2*a23)*y0^2*y2,\hfill \notag\\
%&a42*(y0*y1^2-2*y1*y2*y3)+(2*a22-2*a23)*y1^2*y3)\hfill \notag\\
% \end{align*}
% }
%When the restriction $b=1$ is removed, the family  $I(a_{22},a_{23},a_{42},b)$ is homogeneous. When $b=1$ it is the affine family $V_{p}^{aff}(1)$, and the last cubic form lies in the span of the others.  When $b=0$, the first ten forms are dependant, but together with the last one they define ten independent equations.  
%So the homogeneous family $I(a_{22},a_{23},a_{42},b)$ defines a compactification $V_p(1)$ of $V_{p}^{aff}(1)$ in the Grassmannian $\GG(10,((q^2)^\perp)_3)$.
%In the compactification we find, in particular, the limit, when $a_{22}-a_{23}=a_{42}=0, b=1$ and $a_{22}$ tends to infinity in
%$I(a_{22},a_{23},a_{42})$.  This limit is the ideal
%%{\tiny
 %      \begin{align*}
  %     I_{ver,p}=&(y_3^3,y_1y_3^2,y_0y_3^2,2y_0y_1y_3-y_3^2y_2,y_0^2y_3,\\
   %    &y_0^2y_1-2y_0y_2y_3,y_0^3,y_0^2y_2,y_1^2y_3,y_0y_1^2-2y_1y_2y_3).
    %   \end{align*}
  %  }

We shall see that $V_p(1)$ may be identified with a variety in a Lagrangian Grassmannian inside $\GG(10,((q^2)^\perp)_3)$.
For this, consider again the $16$-space of apolar cubics,
\begin{align*}
(q^2)^\bot_3=\langle &y_3^3,y_1y_3^2,y_0y_3^2,y_1^2y_3,2y_0y_1y_3-y_2y_3^2,y_0^2y_3,y_2^3,y_1y_2^2,y_0y_2^2,\\
&
y_1^2y_2,2y_0y_1y_2-y_2^2y_3,y_0^2y_2,y_1^3,y_0y_1^2-2y_1y_2y_3,y_0^2y_1-2y_0y_2y_3,y_0^3\rangle.
\end{align*}
%and the $3$-dimensional family of apolar schemes of length $10$ that are supported at $(0:0:1:0)$ and are at least triple at the point when restricted to the line spanned by $(0:0:1:0)$ and $(0:1:0:0)$.
The ideals $I(a_{22},a_{23},a_{42},b)$ in $V_p(1)$ all contain the space 
$$K_L=\langle y_3^3,y_0y_3^2,y_0^2y_3,y_0^3\rangle.
$$
of cubics that are triple along $L$.  In fact $K_L\subset (q^2)^\bot_3$ and so modulo $K_L$ the $6$-spaces of cubic forms $I(a_{22},a_{23},a_{42},b)_3/K_L$ are subspaces of 
 the $12$-space
\begin{align*}
E_L:=(q^2)^\bot_3/K_L=\langle & y_1y_3^2,2y_0y_1y_3-y_2y_3^2,y_0^2y_1-2y_0y_2y_3,y_0^2y_2,
    y_0y_2^2,
    \\
&2y_0y_1y_2-y_2^2y_3,y_0y_1^2-2y_1y_2y_3,y_1^2y_3,
    y_1^3,y_1^2y_2,y_1y_2^2,y_2^3\rangle.
\end{align*}

  In the given basis of $E_L$
  the $6$-spaces $I(a_{22},a_{23},a_{42},b)_3/K_L$ are represented by the $(6\times 12)$-matrix

\[
U_1= \begin{pmatrix}{1\;0\;0}  & 0 & 0 & 0& 0 & 0 & 0 & 0 & 0 & 0\\ {0\;1\;0}&0&0&0&0&0&0&0&0&0\\ {0\;0\;1}&0&0&0&0&0&0&0&0&0\\
 {0\;0\;0}&a_{42}&0&0&0&-b&0&0&0&0\\ {0\;0\;0}& (26a_{22}-14a_{23})&-6a_{42}&4(a_{22}-a_{23})&0&0&3b&0&0&0\\ {0\;0\;0}&(2a_{22}-2a_{23})&0&0&b&0&0&0&0&0\\
 \end{pmatrix}
\]

%$(26*a22-14*a23),4*(a22-a23),(2*a22-2*a23)$
There is a skew symmetric nondegenerate $(12\times 12)$-matrix $M$ for which this family of  $6$-spaces are isotropic:
Let $M$ be the skew matrix whose antidiagonal has the entries $(3,3,3,-4,-2,-6)$ above the diagonal, starting with the entry in the first row, and $0$'s everywhere else. So 
$$M={\rm antidiagonal}\;(3,3,3,-4,-2,-6,6,2,4,-3,-3,-3).$$
Then the matrix product $$U_1\cdot M\cdot U_1^t=0.$$
So the the $6$-spaces $I(a_{22},a_{23},a_{42},b)_3/K_L$ are isotropic for $M$.

%The form is a natural map from $M:V_{12}^*\to V_{12}$. 
%Let us denote by $D_{ver}$ the quotient $I_{{ver,p},3}/K_L.$
Let $LG_M(6,E_{L})$ be the Lagrangian Grassmannian of $6$-spaces in 
$E_{L}$ that are isotropic for $M$.
 \begin{proposition} \label{LG compactification of Vp1}  
 The ideals $[I]\in V_p(1)$ are generated by $10$-spaces $I_3$ of apolar cubics that contain $K_L$, and whose quotient $I_3/K_L$ in $E_{L}$ are isotropic for the skew form $M$ and satisfy the following two Schubert conditions:
 \begin{itemize}
 \item $I_3\subset U_{L,p}$
 \item $I_3\cap I_{{ver,p},3}$ is at least $9$-dimensional.
 \end{itemize}
% The locus of ideals in $V_p(1)$ that are contained in the ideal of $L=\{y_0=y_3=0\}$
%    is the cone over a conic in the Veronese surface.
 %   Any saturated ideal in $V_p(1)$ defines a local apolar scheme supported at $p$.
       \end{proposition}
%{\color{red} I am confused by the name $\Gamma_0$, wasn't it reserved to the local %special scheme on the quadric?  I changed to $\Gamma_{ver}$}
\begin{proof}

%So the vertex of $V_{p}(1)$ is spanned by the cubics $y0^2*y2,
%y1^2*y3,
%y0*y1^2-2*y1*y2*y3$.  
%Each $3$-space in $V_{p}(1)$, contains a $2$-space $U$ in the space %spanned by $c1,c2,c3$, and is spanned by $U$ and a cubic in the span %of $c1$ and a cubic $c_U$ in the $3$-space $c4,c5,c6$.
%The map $U\mapsto c_U$ is linear:  
%$$U\mapsto c_U:  (a_1:a_2:a_3)\mapsto (3a_1c6-6a_2c5+4a_3c4),  %\lambda\in \CC.$$
%So the $3$-spaces in $V_p(1)$ are the $3$-spaces $\{\langle %U,c_U\rangle|$ in $V_p(1)$ generated by graph  of the linear map %$U\mapsto c_U$. 

The $3$-spaces $I(a_{22},a_{23},a_{42},b)_3/W_{L,p}$ in the Grassmannian $\GG(3,U_{L,p}/W_{L,p})$  are shown in Proposition \ref{grassmannian compactification of Vp1} to form the cone over a Veronese surface with vertex at $[I_{(ver,p),3}/W_{L,p}]$. 

In $LG_M(6,E_{L})$ the variety $V_p(1)$ lies in the subvariety defined by the listed Schubert conditions.
The first Schubert condition defines the subvariety $$\GG(W_{L,p}/K_L,6,U_{L,p}/K_L, E_{L})\cap LG_M(6,E_{L}).$$
Here the $3$-space $W_{L,p}/K_L$ is isotropic, so its orthogonal $(W_{L,p}/K_L)^\bot$ w.r.t. $M$ is a $9$-space that coincides with $U_{L,p}/K_L$. Thus the first condition defines a Lagrangian Grassmannian ${\rm LG}(3,6)$ in $LG_M(6,E_{L})$.  The second condition defines the tangent cone in this ${\rm LG}(3,6)$ at the $3$-space $I_{ver,p,3}/K_L$ of the vertex ideal. This cone is the cone over a Veronese surface, so coincides with $V_p(1)$.
\end{proof}

The family $V_p(2)$ is, of course, entirely similar to $V_p(1)$, with the line $$L'=\{y_3=y_1=0\}$$ playing the role of $L$.
In the analysis of $V_p(1)$, many of the ideals in the family depend on $L$ while the dependency on $p$ is less clear.  We will clarify this issue.

For this we consider the apolar cubics that vanish on $L$.
There is an $8$-space $W_L$ of apolar cubic forms that are singular along $L$
\begin{align*}
 W_L=&\langle y_3^3,y_1y_3^2,y_0y_3^2,2y_0y_1y_3-y_3^2y_2,y_0^2y_3,y_0^2y_1-2y_0y_2y_3,y_0^3,y_0^2y_2,\rangle 
\end{align*}
and a $12$-space $U_L$ of apolar cubics that vanish on $L$. 
Modulo $W_L$ the latter is 
\[
 U_L/W_L=\langle y_1^2y_3,
y_0y_1^2-2y_1y_2y_3,
y_2^2y_3-2y_0y_1y_2,
y_0y_2^2 \rangle .
\]
Of these, the pencils of forms that together with the first eight generate an apolar ideal $I\in V_p(1)$ are parameterized by
%{\tiny
$$\begin{pmatrix}
2(a_{22}-a_{23})&a_{42}&0&0\\
0&0&-4(a_{22}-a_{23})&-6a_{42}\\
\end{pmatrix}\cdot
\begin{pmatrix}y_1^2y_3\\
y_0y_1^2-2y_1y_2y_3\\
y_2^2y_3-2y_0y_1y_2\\
y_0y_2^2\\
\end{pmatrix}.$$
%}
In Pl\"ucker coordinates, $p_{01}=p_{23}=p_{03}-3p_{12}=0$ in $\GG(2,U_L/W_L)$.

Next, we move $p$ on $L$ and compare.
The linear transformation 
%{\tiny 
$$(y_0,y_1,y_2,y_3)\mapsto (y_0-(1/a)y_3,y_1-ay_2,y_1+ay_2,y_0+(1/a)y_3)$$
%}
leaves the quadric $Q^{-1}$, the line $L=\{y_0=y_3=0\}$ and the apolar ideal $(q^2)^{\bot}$ invariant, and moves $$(0:0:1:0)\mapsto (0:a:1:0).$$  With this transformation we analyze $V_{p}^{aff}(1)$ as $p=(0:a:1:0)$ moves with $a$ along the line $y_0=y_3=0$.  The vertex ideal in $V_{p}(1)$ is generated by
the $8$-space $W_L$ of cubics singular along the line $L$
%{\tiny
%$$ 
%y3^3,y1y3^2,y0y3^2,2y0y1*y3-y3^2*y2,y0^2*y3,y0^2*y1-%2*y0*y2*y3,y0^3,y0^2*y2,y2^2*y3-2*y0*y1*y2.
%$$
%}
and in addition the two cubics
%{\tiny
$$\begin{pmatrix}
1&0&-a^2&-2a^3\\
2&a&0&-a^3\\
\end{pmatrix}\cdot
\begin{pmatrix}y_1^2y_3\\
y_0y_1^2-2y_1y_2y_3\\
y_2^2y_3-2y_0y_1y_2\\
y_0y_2^2\\
\end{pmatrix}$$
%}
in $U_L/W_L$.

 Notice that these ten cubic generators of the vertex ideal are 
all apolar cubics that contain the line $\{y_0=y_3=0\}$ and are singular at the point $p=(0:a:1:0)$.

The Pl\"ucker coordinates in $\GG(2,U_L/W_L)$ for the line generated by the last two generators are
%{\tiny
$$(1,2a,3a^2,a^2,2a^3,a^4),$$
%}
so the vertex ideals form a rational normal quartic curve in the Grassmannian.  The tangents of this curve are easily seen to also be contained in the Grassmannian, so the vertex ideals form the tangent developable of a twisted cubic curve in the space of cubics. This curve of cubics is 
%{\tiny
$$\begin{pmatrix}
1&a&a^2&a^3\\
\end{pmatrix}\cdot
\begin{pmatrix}y_1^2y_3\\
y_0y_1^2-2y_1y_2y_3\\
y_2^2y_3-2y_0y_1y_2\\
y_0y_2^2\\
\end{pmatrix}.$$
%}
The pencils of unsaturated apolar ideals $I\in V_{p}(1)$ 
are parameterized by ($c=a_{22}-a_{23},d=a_{42}$)
%{\tiny
$$\begin{pmatrix}
2bc-3bd&(2c-d)&2ac+ad&2a^2c+3a^2d\\
-12bc+6bd&4c-6d&4ac+6ad&-12a^2c-6a^2d\\
\end{pmatrix}\cdot
\begin{pmatrix}y_1^2y_3\\
y_0y_1^2-2y_1y_2y_3\\
y_2^2y_3-2y_0y_1y_2\\
y_0y_2^2\\
\end{pmatrix}$$
%}
Also these pencils are contained in the Pl\"ucker hyperplane $$\{p_{03}-3p_{12}=0\}\subset \GG(2,U_L/W_L).$$

We summarize our findings 
and define $V_L(1)$ as the Grassmannian compactification of the variety of apolar ideals of schemes of length ten with multiplicity $3$ at some point along the line $L=\{y_0=y_3=0\}\subset Q^{-1}$.

\begin{proposition}\label{apolar_L} The $4$-fold variety 
$$V_L(1)=\bigcup_{p\in L}V_{p}(1)\subset \GG(10,(q^2)^{\bot}_3)$$ contains a quadric threefold of unsaturated ideals that all are contained in the ideal of the line $L$.  
%These ideals all contain all $8$ apolar cubics that are singular along the %line, and are generated in addition by a pencil of cubics in the $4$-space
%{\tiny
%$$\langle y1^2*y3,
%y0*y1^2-2*y1*y2*y3,
%y2^2*y3-2*y0*y1*y2,
%y0*y2^2\rangle.$$
%}
\begin{itemize}

\item For each point $p\in L$, the variety $V_p(1)$, contains unsaturated ideals that form a cone over a conic.  The vertex of this cone is the vertex ideal $I_{ver,p}$ generated by the cubics that contain $L$ and are singular at $p$.  Modulo the cubics singular along $L$, the vertex ideals are generated by the pencils of cubics tangent to a twisted cubic $C_L$ of apolar cubics.
\item 
For any two distinct points $p,p'\in L$, the varieties  $V_p(1)$, and $V_{p'}(1)$ intersect in a conic section of unsaturated apolar ideals. 
\end{itemize}

%{\tiny
%$$s^3y1^2*y3+ s^2*t*
%(y0*y1^2-2*y1*y2*y3)+s*t^2*
%(y2^2*y3-2*y0*y1*y2)+t^3*
%y0*y2^2$$
%}
%The threefold of apolar unsaturated ideals in the Grassmannian compactification is a quadric threefold;  
The quadric threefold of unsaturated ideals in $V_L(1)$ is the intersection of the linear span of the curve of vertex ideals with the Grassmannian $\GG(10,((q^2)^\perp)_3)$.
% of ten-dimensional subspaces of apolar cubics.  

\end{proposition}
\begin{proof}
First, recall that the ideal of any apolar scheme with multiplicity $3$ at $p$ when restricted to $L$ is generated by apolar cubics in $U_{L,p}$, so it lies in $V_p(1)$. Therefore, $V_L(1)$ is the union of the $V_p(1)$.  Next, the set of ideals in $V_L(1)$ that are contained in the ideal of $L$, are the unsaturated ones, their parameterization all satisfy the Pl\"ucker relation $\{p_{03}-3p_{12}=0\}$ in the above coordinates for $\GG(2,U_L/W_L)$ seen as the subvariety 
$$\GG(2,U_L/W_L)\cong \GG(W_L,10,U_L,(q^2)^{\bot}_3)\subset \GG(10,(q^2)^{\bot}_3) .$$ 
Here, $\GG(W_L,10,U_L,(q^2)^{\bot}_3)$ is the variety of $10$-dimensional subspaces of $U_L$ that contain $W_L$.
The Pl\"ucker relation defines a smooth quadric threefold in $\GG(2,U_L/W_L)$ that for each $p\in L$ contains the cone over a conic with vertex (ideal) on the curve of tangents to the twisted cubic curve $C\subset \PP(U_L/W_L)$.
The cone must be the intersection of the quadric threefold with the tangent hyperplane at the vertex.  As $p$ varies these cones fill the quadric threefold, and any two cones intersect in a plane conic, so the proposition follows. 
\end{proof}
%An ideal $I\in V_L(1)$ is unsaturated if the cubic generators of $I$ all vanish on $L$.
%This is the case when all $8$ apolar cubics 
%{\tiny
%$$ 
%y3^3,y1*y3^2,y0*y3^2,2*y0*y1*y3-y3^2*y2,y0^2*y3,y0^2*y1-%2*y0*y2*y3,y0^3,y0^2*y2,y2^2*y3-2*y0*y1*y2.
%$$
%}
%that are singular along $L$ are in $I$, and 
%in addition the two cubics that depend on $p=(0:a:1:0)$, so that $I\in %V_p(1)$, for some $p$.
%The rational normal quartic curve of vertex ideals satisfy the %$p_{03}=3*p_{12}$ in the Pl\"ucker coordinates of the  $4$-space of apolar %cubics with basis 
%{\tiny
%$$\{ y1^2*y3,
%y0*y1^2-2*y1*y2*y3,
%y2^2*y3-2*y0*y1*y2,
%y0*y2^2\}.$$
%}

%??????????
%The pencil of cubics of all unsaturated apolar ideals in the %compactification satisfy the relation $p_{03}=3*p_{12}$ in the Pl\"ucker %coordinates of the  $4$-space of apolar cubics with basis {\tiny
%$$\{ y1^2*y3,
%y0*y1^2-2*y1*y2*y3,
%y2^2*y3-2*y0*y1*y2,
%y0*y2^2\}.$$
%}

%??????????

%{\tiny
%$$\begin{pmatrix}
%b&0&0\\
%0&-4*(a22-a23)&-6*a42\\
%2*a22-2*a23&a42&0&0\\

%\end{pmatrix}\cdot \begin{pmatrix}
%y1^2*y3\\
%y0*y1^2-2*y1*y2*y3\\
%y2^2*y3-2*y0*y1*y2\\
%y0*y2^2\\
%\end{pmatrix}
%$$
%}

%This follows by a computation; the case of $p=(0:0:1:0)$ follows by the above remark.  For any other point $q\in L$, a group action that preserves apolarity with respect to $f=q^2$ and maps $p$ to $q$ is applied to the ideals of $V_p(1)$ and its compactification, to conclude.

Finally, we find the intersection $V_p(0)\cap V_p(1)$.
Recall the following subspaces of apolar cubics.  
\begin{itemize}
\item $B_{T_p}$ is the $6$-space of apolar cubics that vanish on the tangent space $T_p$ to $Q^{-1}$ at $p$.
\item $A_{p}$ is the $13$-space of apolar cubics that are singular at $p$ after restriction to the tangent space $T_p$.
\item $W_{L,p}$ is the $7$-space of apolar cubics that are singular along $L\subset Q^{-1}$, and has multiplicity $3$ at $p$ after restriction to $T_p$.
\item $U_{L,p}$ is the $13$-space of apolar cubics that have multiplicity $3$ at $p$ when restricted to $L$.
\end{itemize}
We apply this, as above, when $p=(0:0:1:0)$, so $T_p=\{y_3=0\}$.

An ideal in $V_p(0)$ contains $B_{T_p}$, and an ideal in $V_p(1)$ contains $W_{L,p}$, so an ideal in $V_p(0)\cap V_p(1)$ contains
$B_{T_p}+W_{L,p}$.  Notice that $$B_{T_p}+W_{L,p}=W_{L,p}+\langle y_1^2y_3\rangle,$$ so it is $8$-dimensional.  Furthermore, the $10$-space of cubics in an ideal in $V_p(0)\cap V_p(1)$ is contained in $A_p\cap U_{L,p}$, the $12$-space of apolar cubics that has multiplicity $3$ when restricted to $T_p$.  Now an ideal in $V_p(1)$ contains $y_1^2y_3$ when the parameter $a_{42}$ vanishes, which also ensures that the ideal is contained in this $12$-space, i.e. belongs to $V_p(0).$   In the notation of the proof of Proposition \ref{grassmannian compactification of Vp1} above, the form $$p_{125},p_{134},p_{135},p_{235}, 3p_{124}p_{236}-2p_{126}^2$$ in the Pl\"ucker coordinates for $V_p(1)$ vanishes, so the intersection is a cone over a conic with vertex at $I_{ver,p}$.

We summarize:
\begin{proposition}\label{cone} The intersection $V_p(0)\cap V_p(1)$ in $\GG(10,(q^2)^\bot_3)$ is a cone over a conic with vertex the vertex ideal $I_{ver,p}$.
    \end{proposition}

\section{Global structure and degrees }\label{main6}

In Sections \ref{maincomponent} and \ref{specialcomponents}, we described subvarieties of $\VAPS(q^2,10)$ that compactify the set of apolar schemes of length ten that contain a fixed point $p$ in $Q^{-1}$. In this last section we investigate the global Grassmannian compactification $\VAPS(q^2,10)$ of apolar schemes of length ten; the compactification of the image of the set of apolar schemes $\Gamma$ under the forgetful map
$$\Gamma\mapsto [I_{\Gamma,3}]\in \GG(10,(q^2)^\bot_3).$$
 
By Proposition \ref{three components local}, the closure $V_p\subset \VAPS(q^2,10)$ of the set of apolar schemes $$\Gamma\in \VAPS(q^2,10)$$ contained in the ideal $(y_0,y_1,y_2)$ (a point on the inverse quadric), has three components $V_{p}(0),V_{p}(1)$ and $V_{p}(2)$.  Here  $V_{p}(0)$ is isomorphic to a singular Fano threefold $V_{22}$ of degree $22$, while each of $V_{p}(1)$ and $V_{p}(2)$ are cones over Veronese surfaces, cf. Corollary \ref{matrixM'} and Proposition \ref{grassmannian compactification of Vp1}.

Now, by moving $p$ on the quadric $Q^{-1}$ we sweep out the whole of $\VAPS(q^2,10)$. Denote  $$V:=\VAPS(q^2,10)=\bigcup_p{V_p}\subset \GG(10,(q^2)^\bot_3).$$

\begin{lemma}\label{three components global} The variety 
$V$ has three components: $V(i) :=\bigcup_p{V_p(i)}$ for $i=0,1,2$.
\end{lemma} 
\begin{proof}It is clear that $V(i)$ are components of $V$, and it remains to observe that $V(i)$ are distinct. For the latter, we recall that general elements of $V(0)$ are distinguished by the property that their corresponding schemes have three points in their support. Moreover, general elements of $V(1)$ and $V(2)$ represent schemes supported at a single point. Hence, $V(1) = V(2)$ would imply $V_p(1) = V_p(2)$ for a general $p$, but this is not the case by our local study.  
\end{proof}

For $V(0)$ we will identify $\GG(10,(q^2)^\bot_3)=\GG(6,((q^2)^\bot_3)^*)$, and consider $V(0)$ as a subvariety of the Grassmannian $\GG(6,((q^2)^\bot_3)^*)$.

\subsection{The main component}\label{maincomponent} %Over each point $p$ in the inverse quadric $Q^{-1}$ we have seen three affine components of apolar schemes of length 10: $V_p(0),V_p(1),V_p(2).$  Compactifying the first and taking the union over $Q^{-1}$, we get the main component. %treated in previous sections. 
By Lemma \ref{fibration of main component} we know that $V_p(0)$ is disjoint from $V_{p'}(0)$ for $p$, $p'$ distinct in $Q^{-1}$.  Therefore, the main component $$V(0)\subset \VAPS(q^2,10)$$ has the structure of a fibre bundle, i.e. $$V(0)=\bigcup_{p\in Q^{-1}}V_p(0)$$ is a set theoretic disjoint union. By globalizing the description of $V_p(0)$ we will obtain a global structure of $V(0)$.

The cubic generators of $(q^2)^\bot_3\subset T_3$ define a map
$$\phi:\PP(S_1)\to \PP(V_{16}),$$
where we set $V_{16}=((q^2)^\bot_3)^*\subset S_3$.
Let $Q^{-1}=\{q^{-1}=0\}\subset \PP(S_1) $, where $q^{-1}\in T_2$ is the quadric inverse to $q\in S_2$, then the image $S_{Q^{-1}}=\phi(Q^{-1})\subset \PP(V_{16})$  of $Q^{-1}$ is a smooth surface, the embedding of the quadric surface $Q^{-1}$ by the complete system of $(3,3)$-forms. 

For a saturated ideal $$I_\Gamma \in V(0)\subset \VAPS(q^2,10),$$ i.e. in the main component, we have that the apolar scheme $\Gamma$ contains a subscheme of length at least eight supported in some point $p\in Q^{-1}$,  see Proposition \ref{grassmannian compactification of Vp0}.
The span of $\phi(\Gamma)$ in $\PP(V_{16})$ is $\PP(I_{\Gamma,3}^\bot)\subset\PP(((q^2)^\bot_3)^*)$.
Recall, from \ref{compactificationofmain} the set of subspaces of apolar cubics
$B_{T_p}\subset I_{\Gamma,3}\subset A_p\subset (q^2)^\bot_3$.   The inclusions give reverse inclusions
$$\PP(A_p^\bot)\subset \PP(I_{\Gamma,3}^\bot)\subset \PP(B_{T_p}^\bot).$$
Here $\PP(A_p^\bot)$ is the projective tangent plane to $S_{Q^{-1}}$ at $\phi(p)$. In fact every cubic in $A_p$ is singular at $p$ after restriction to the tangent plane $T_p$ to $Q^{-1}$ at $p$. Moreover, $\PP(B_{T_p}^\bot)$ is the third order osculating space to 
$S_{Q^{-1}}$ at $\phi(p)$, %span of $\phi(T_p)\subset \PP^{15}$, 
since every cubic in $B_{T_p}$ vanishes on $T_p$.   Similarly, $\PP(C_{p}^\bot)$ is the second order osculating space to $S_{Q^{-1}}$, since every cubic in
$C_{p}$ has multiplicity three at $p$ when restricted to $T_p$, cf. Section \ref{grassmannian compactification of Vp0}.
\begin{remark}\label{osculatingspace in VAPS}
Recall from Section \ref{affine deformation}, that $V_p$ contains a special apolar scheme $\Gamma_0$, the tautological scheme whose ideal $I_{\Gamma_0}$ is generated by the cubic forms in $C_p$, cf. Remark \ref{Cp}. 
Therefore, the second order osculating spaces to $S_{Q^{-1}}$ all lie in $\VAPS(q^2,10)\subset \GG(6,V_{16})$.    
\end{remark}

%where $T_P$ is the projective tangent space to $Q^{-1}\subset \PP^3$ at $P$. 
%The space $B_{T_p}$ is $6$-dimensional, so $\PP(B_{T_p}^\bot)$ is a $\PP^9$ and $\phi|_{T_p}:T_p\to \PP^{9}$ is the third Veronese morphism.  
%Let $H_p\subset \langle x_0,x_1,x_2,x_3\rangle$ be the $3$-space such that $T_p=\PP(H_p)$.
%Then there is a natural identification $$B_{T_p}^\bot=Sym_3(H_p).$$

Now, we note that  $I_{\Gamma,3}^\bot$ is a $6$-dimensional subspace of $((q^2)^\bot_3)^*$, while $I_{\Gamma,3}^\bot/A_p^\bot$ is a $3$-dimensional subspace of 
the $7$-dimensional vector space $$E_p=B_{T_p}^\bot/A_p^\bot.
$$
So $[I/A_p^\bot]\mapsto [I]$
defines a natural map $$\psi_p:\GG(3,E_p)\to \GG(6,((q^2)^\bot_3)^*).$$

%the $3$-space  $I_{\Gamma,3}^\bot/A_p^\bot$ 
% is a point in $G(3,E_p)$ where $E_p$ is the $7$-dimensional vector space $$E_p=B_{T_p}^\bot/A_p^\bot,$$ and 
We will globalize this map as $p$ moves on  $Q^{-1}$, by first identifying $E_p$ as the fiber at $p$ of a rank $7$ vector bundle $\mathcal E$ on $Q^{-1}$.  

For this, we use the relation between $A_p^\bot$ and $B_{T_p}^\bot$ with the tangent space and the third-order osculating space to $S_{Q^{-1}}$ at $p$. The projective tangent bundle to $S_{Q^{-1}}$ is dual to the bundle $P^1_{Q^{-1}}(3,3)$ of principal parts. Similarly, the third-order osculating bundle is dual to the bundle $P^3_{Q^{-1}}(3,3)$ of third-order principal parts.

As $p$ moves on $Q^{-1}$, the vector space $E_p$ is therefore the fiber of a rank-7 vector bundle $\mathcal{E}$ on $Q^{-1}$, which is the quotient of $P^3_{Q^{-1}}(3,3)^*$ by $P^1_{Q^{-1}}(3,3)^*$. The dual bundle $\mathcal{E}^*$ therefore fits into the short exact sequence
$$
0 \to \mathcal{E}^* \to P^3_{Q^{-1}}(3,3) \to P^1_{Q^{-1}}(3,3) \to 0.
$$

So there is a surjective vector bundle map
$$P^3_{Q^{-1}}(3,3)^*\to \mathcal E$$
and an injective vector bundle map
$$ P^3_{Q^{-1}}(3,3)^*\to ((q^2)^\bot_3)^*\times Q^{-1}$$
that induces the global map factorizing as follows:
$$\psi: \GG(3,\mathcal E)\to \GG(6,P^3_{Q^{-1}}(3,3)^*)\to  Q^{-1}\times \GG(6,((q^2)^\bot_3)^*)\to \GG(6,((q^2)^\bot_3)^*),$$
where the left morphism is given by pullback via the surjective map $P^3_{Q^{-1}}(3,3)^*\to \mathcal E$ and the second embedding is induced by the inclusion $ P^3_{Q^{-1}}(3,3)^*\to ((q^2)^\bot_3)^*\times Q^{-1}$ the last is the projection from the trivial bundle to its fiber. 

 %=Sym_3(H_p)/A_p^\bot.$$
 %obtained from the quotient $\PP_p$ by $\PP^p$.

%We first globalize the $7$-dimensional vector space $E_p=B_{T_p}^\bot/A_p^\bot$ to a rank $7$
%bundle on $Q^{-1}$.  For this we use the relation of $A_p^\bot$ and $B_{T_p}^\bot$ to the tangent and the third order osculating space to $S_{Q^{-1}}$ at $p$.  The projective tangent bundle to $S_{Q^{-1}}$ is dual to the bundle $P^1_{Q^{-1}}(3,3)$ of principal parts.  Similarly the third order osculating bundle is dual to the bundle $P^3_{Q^{-1}}(3,3)$ of third order principal parts.
%As $p$ moves on $Q^{-1}$ the vector space $E_p$ is therefore the fiber of a rank $7$ vector bundle $\mathcal{E}$ on $Q^{-1}$, the quotient of $P^3_{Q^{-1}}=P^3_{Q^{-1}}(3,3)^*$ by $P^1_{Q^{-1}}(3,3)^*$.
%So the dual bundle $\mathcal{E}^*$ fits in a short exact sequence
%is the kernel of the natural surjective map $P^3_{Q^{-1}}(3,3)\to P^1_{Q^{-1}}(3,3)$.

The bundle of principal parts on ${Q^{-1}}$ w.r.t. the third Veronese embedding fits
 in the exact sequence % obtained from the cotangent bundle $\Omega_{Q^{-1}}$:%following extension of the cotangent bundle %$T_{Q^{-1}}$: 
$$0\to \Omega_{Q^{-1}}(3,3)\to P^1_{Q^{-1}}(3,3)\to \oo(3,3)\to 0,
$$
where $\Omega_{Q^{-1}}$ is the cotangent bundle on $Q^{-1}.$
Similarly,
$$0\to {\rm Sym}^2\Omega_{Q^{-1}}(3,3)\to P^2_{Q^{-1}}(3,3)\to P^1_{Q^{-1}}(3,3)\to 0,
$$
and 
$$0\to {\rm Sym}^3\Omega_{Q^{-1}}(3,3)\to P^3_{Q^{-1}}(3,3)\to P^2_{Q^{-1}}(3,3)\to 0.
$$

The surjective map $P^3_{Q^{-1}}(3,3)\to P^1_{Q^{-1}}(3,3)$ clearly factors through $P^2_{Q^{-1}}(3,3)$, so its kernel $\mathcal E^*$ maps onto the kernel of  $P^2_{Q^{-1}}(3,3)\to P^1_{Q^{-1}}(3,3)$.  Hence there is 
%By the above, the bundle $\mathcal E^*$  fits in the exact sequence
%$$
% 0\to \mathcal E^*\to P^3_{Q^{-1}}(3,3)\to P^1_{Q^{-1}}(3,3)\to 0.  
%$$
%Since the last map factors through $P^2_{Q^{-1}}(3,3)$, and hence %onto the kernel of  $P^2_{Q^{-1}}(3,3)\to P^1_{Q^{-1}}(3,3)$, we get 
a surjective map $\mathcal E^*\to {\rm Sym}^2\Omega_{Q^{-1}}(3,3)$.  The kernel coincides with the kernel of $$P^3_{Q^{-1}}(3,3)\to P^2_{Q^{-1}}(3,3),$$  so $\mathcal E^*$ fits in the exact sequence
\begin{align}\label{E^*}
0\to{\rm Sym}^3\Omega_{Q^{-1}}(3,3)\to \mathcal E^*\to {\rm Sym}^2\Omega_{Q^{-1}}(3,3)\to 0. 
\end{align}

The bundle of principal parts on ${Q^{-1}}$ w.r.t. the trivial line bundle fits
 in the exact sequence % obtained from the cotangent bundle $\Omega_{Q^{-1}}$:%following extension of the cotangent bundle %$T_{Q^{-1}}$: 
$$0\to \Omega_{Q^{-1}}\to P^1_{Q^{-1}}\to \oo_{Q^{-1}}\to 0,
$$
where $\Omega_{Q^{-1}}$ is the cotangent bundle on $Q^{-1}.$
Consider the rank $3$-bundle $$R:=(P^1_{Q^{-1}})^*,$$ the dual of the bundle of principal parts. It is a general extension 
$$0\to\mathcal{O}_{Q^{-1}}\to R\to \mathcal{O}_{Q^{-1}}(2,0)\oplus \mathcal{O}_{Q^{-1}}(0,2)\to 0.$$

Recall that the matrix $M'$, 
\[
M'=\begin{pmatrix}\label{Mp'-matrix} 0& 0& 0&0& 0& 3y_{1}&-3y_{2}\cr 0& 0&0& 0& -3y_{1}&0&3y_{0}\cr
0& 0& 0& 2y_{1}&-y_{2}&-3y_{0}&0\cr
 0& 0&-2y_{1}& 0& 2y_{0}&0&0\cr
0& 3y_{1}&y_{2}&-2y_{0}&0&0&0\cr
-3y_{1}&0&3y_{0}& 0& 0&0&0\cr
3y_{2}&-3y_{0}& 0&0& 0& 0&0\cr
\end{pmatrix}
\] of \ref{matrixM'}, defines a net of skew symmetric maps
$\rho_p: E_p\to E_p^*$, i.e. a net of $2$-vectors in $\wedge^2 E_p^*$. 
This net provides an embedding $R_p\to \wedge^2 E_p^*$.

\textbf{Claim} We claim that this embedding lifts to a vector bundle map 
$$R=(P^1_{Q^{-1}})^*\to \wedge^2\mathcal E^*.$$

In fact, the matrix $M'$
 defines a net of skew symmetric maps
$\rho_p: E_p\to E_p^*$, with respect to the basis 
 \ref{Alocal}, for $A_p/B_{T_p}$.  Since $B_{T_p}^\bot/A_p^\bot$ is natural dual to $A_p/B_{T_p}$, the cubic forms \ref{Alocal} in $y_i$ form a basis for $E_p^*$, i.e. a basis of coordinates on $E_p$ as follows \begin{align*}
     &(e_{3,0},e_{2,0},e_{2,1},e_{1,1},e_{1,2},e_{0,2},e_{0,3})=\\
     &(y_{0}^3, y_{0}^2y_{2},y_{0}^2y_{1}-2y_{0}y_{2}y_{3}, 2y_{0}y_{1}y_{2}-y_{3}y_{2}^2,y_{0}y_{1}^2-2y_{1}y_{2}y_{3}, y_{1}^2y_{2},y_{1}^3).
 \end{align*}  
We will now spread all this data to a map of bundles by means of a transitive group action on the quadric.
For that consider the group $\mathcal{G}=SL(2)\times SL(2)$  inducing automorphisms of $Q^{-1}$ preserving its two rulings. Note that $\mathcal{G}$ acts transitively on $Q^{-1}$ hence if $\mathcal{P}$ is the stabilizer of the point $p$ we can write 
$$Q^{-1}=\mathcal{G}/\mathcal{P}$$ and both $R$ and $\bigwedge^2 \mathcal E^*$ are homogeneous bundles, which in particular are described as:
$$R=(\mathcal{G}\times R_p)/\mathcal{P},$$
and 
$$\wedge^2\mathcal E^*=%(G\times \wedge^2\mathbb C^{7})/P =\quad 
(\mathcal{G}\times\wedge^2 E_p^*)/\mathcal{P},$$
for appropriate group actions of $\mathcal{P}$. Now $M'$ defines a map 
$${\rm id}\times M':   \mathcal{G}\times R_p\to \mathcal G\times \wedge^2 E_p^*$$ % =(G\times\wedge^2 E^*)/P ,$$
which we need to check that it induces a map of quotients.

Let us describe the above arguments explicitly. The action of $\mathcal{G}$  is the restriction to $Q^{-1}$ of an action on the projective space $\mathbb P(V_x)$ spanned by $Q^{-1}$. If now $Q^{-1}$ is the image of $\mathbb{P}^1\times \mathbb P^1$ via the map $$\mathbb P^1\times \mathbb P^1\ni  ((a:b),(c:d))\mapsto [-ad x_0 + bc x_1+ bd x_2+ac x_3] \in \mathbb P(V_x),$$
then the natural action of $\mathcal{G}=SL(2)\times SL(2)$ on $\mathbb P^1 \times \mathbb P^1 $ 
corresponds to an action on $\mathbb C^2 \times \mathbb C^2$
 \begin{equation}\label{g}
 g=\left(\left ( \begin{array}{cc}
     h_1 & h_2 \\
    h_3  & h_4
 \end{array}\right),  \left ( \begin{array}{cc}
     k_1 & k_2 \\
    k_3  & k_4
 \end{array}\right)\right), \text{ with } h_1h_4-h_2h_3=1 \text{ and }  k_1k_4-k_2k_3=1 \end{equation}
given by 
$$g* \left(\left(\begin{array}{c}
     a \\
    b
 \end{array}\right), \left(\begin{array}{c}
     c \\
    d
 \end{array}\right)\right)= \left(\left(\begin{array}{c}
     h_1a+h_2b \\
    h_3 a+h_4b
 \end{array}\right), \left(\begin{array}{c}
     k_1c+k_2d \\
    k_3c+k_4d
 \end{array}\right)\right),$$
which induces a representation on $V_x$ given by:
 \begin{equation}\label{action Vy}
    g*_x\left\{\begin{array}{c}
       x_0 \\
       x_1\\
       x_2\\
       x_3
 \end{array}\right\} =\left\{\begin{array}{c}
h_1k_4x_0-h_2k_3x_1-h_2k_4x_2-h_1k_3x_3\\-h_3k_2x_0+h_4k_1x_1+h_4k_2x_2+h_3k_1x_3\\-h_3k_4x_0+h_4k_3x_1+h_4k_4x_2+h_3k_3x_3\\-h_1k_2x_0+h_2k_1x_1+h_2k_2x_2+h_1k_1x_3 \end{array}\right\}.   
 \end{equation}

%This action clearly preserves the quadric $Q^{-1}$ given by the equation $$y_0y_1+y_2y_3=0$$

 In these coordinates the stabilizer of the  point $p=[x_2]\in Q^{-1}$ is the subgroup $\mathcal{P}\subset \mathcal{G}$ of pairs of triangular matrices:
 $$\mathcal{P}=\left\{\left(\left ( \begin{array}{cc}
     h_1 & h_2 \\
    0  & h_4
 \end{array}\right),  \left ( \begin{array}{cc}
     k_1 & k_2 \\
    0  & k_4
 \end{array}\right)\right)\in \mathcal{G}  : h_1h_4=k_1k_4=1\right\}.$$

 We hence have $$Q^{-1}\simeq \mathcal{G}/\mathcal{P},$$
 where the action of $\mathcal{P}$ is by left multiplication and the isomorphism $\mathcal{G}/\mathcal{P}\to Q^{-1}$ is given by:

 $$\left[ \left(\left ( \begin{array}{cc}
     h_1 & h_2 \\
    h_3  & h_4
 \end{array}\right),  \left ( \begin{array}{cc}
     k_1 & k_2 \\
    k_3  & k_4
 \end{array}\right)\right)\right]\to ((h_3:h_4), (k_3:k_4))\to (-h_3k_4:h_4k_3:h_4k_4:h_3k_3).$$

 Note that the 4-dimensional representation $V_x$ of $\mathcal{G}$ represented by \eqref{action Vy} induces a representation on $V_y=V_x^*$ given by the action 
 $$g*_y v:=\phi^{-1}((g^{-1})^T *_x \phi(v)),$$ where 
 $$\phi: V_y\to V_x$$ is the map
  $$(y_0,y_1,y_2,y_3)\mapsto (x_0,x_1,x_2,x_3).$$
  Explicitly, we have the following representation on $V_y$:
\begin{equation}\label{action Vx}
    g*_y\left\{\begin{array}{c}
       y_0 \\
       y_1\\
       y_2\\
       y_3
 \end{array}\right\} =\left\{\begin{array}{c}
h_4k_1y_0-h_3k_2y_1+h_3k_1y_2+h_4k_2y_3\\-h_2k_3y_0+h_1k_4y_1-h_1k_3y_2-h_2k_4y_3\\h_2k_1y_0-h_1k_2y_1+h_1k_1y_2+h_2k_2y_3\\h_4k_3y_0-h_3k_4y_1+h_3k_3y_2+h_4k_4y_3 \end{array}\right\}.   
 \end{equation}

 The latter induces a representation $S^3V_x^*$ which preserves the space $(q^2)^{\perp}_3\subset S^3V_x^*$ of cubics polynomials in $y$ apolar to $q^2$. In particular, it induces an action of $\mathcal{G}$ on $(q^2)^{\perp}_3$. Now, to see the bundle $\mathcal E^*$ as a homogeneous bundle we observe that the stabilizer subgroup $\mathcal{P}$ acts on $(q^2)^{\perp}_3$ and preserves the subspaces $A_p$ and $B_{T_p}$ hence acts also on $A_p/B_{T_p}$. The latter action describes the homogeneous bundle $\mathcal E^*$ as a quotient 
%We can hence write  $$\mathcal A= \bigcup_{g\in G} \{g*p\}\times g* A_p \subset Q\times (q^2)^{\perp}_3$$ as well as 
%$$\mathcal B_T= \bigcup_{g\in G} \{g*p\}\times g* B_{T_p} \subset Q\times (q^2)^{\perp}_3$$ and get  

$$\mathcal E ^*=(\mathcal{G}\times (A_p/B_{T_p}))/\mathcal{P}.$$
Let us write explicitly the action of the stabiliser $\mathcal{P}$ on $A_p/B_{T_p}$ in the chosen basis \begin{align*}
     &(e_{3,0},e_{2,0},e_{2,1},e_{1,1},e_{1,2},e_{0,2},e_{0,3})=\\
     &(y_{0}^3, y_{0}^2y_{2},y_{0}^2y_{1}-2y_{0}y_{2}y_{3}, 2y_{0}y_{1}y_{2}-y_{3}y_{2}^2,y_{0}y_{1}^2-2y_{1}y_{2}y_{3}, y_{1}^2y_{2},y_{1}^3).
 \end{align*}  

$$g*\left\{\begin{array}{c}
       e_{3,0}\\
       e_{2,0}\\
       e_{2,1}\\
       e_{1,1}\\
       e_{1,2}\\
       e_{0,2}\\
       e_{0,3}
 \end{array}\right\} = \left\{\begin{array}{c}  h_4^3k_1^3 e_{3,0}\\
 
      h_4^2k_1^2(h_2k_1 e_{3,0}+h_1k_1 e_{2,0}-h_1k_2 e_{2,1})\\
       h_1h_4^2k_1^2k_4e_{2,1}\\
      2h_1h_4k_1k_4( h_2k_1 e_{2,1}+ h_1k_1 e_{1,1}-h_1k_2e_{1,2})\\
       h_1^2h_4 k_1k_4^2 e_{1,2}\\
       h_1^2k_4^2( h_2 k_1 e_{1,2}+h_1k_1 e_{0,2}- h_1k_2 e_{0,3})\\
       h_1^3k_4^3 e_{0,3}
 
 \end{array}\right\}.$$

This permits us to understand the bundle $\wedge^2 \mathcal E^*$ by describing the corresponding  action of $\mathcal{P}$ on $\wedge^2 \mathcal E^*_p$. %Consider $E_{y_0}$ to be the orbit of the line 

 %$$(s,t)\times (y_0:y_1:y_2:y_3)\mapsto (sty_0:s^{-1}t^{-1}y_1:st^{-1}y_2:s^{-1}ty_3)$$
% or
%  $$(s,t)\times (y_0:y_1:y_2:y_3)\mapsto (sy_0:ty_1:y_2:sty_3)$$
 %induced from automorphisms  of $Q^{-1}$ preserving $p$. % and the factors of $\PP^1\times \PP^1$.
  %Under the induced action on  $E_p$ and $E_p^*$ the basis vectors have distinct weights, so the action is diagonal for the  coordinates \ref{Alocal} on $E_p$.  The distinct weights allow us to see each of these coordinates on $E_p$ as restrictions of global sections $e_{i,j}$ of an appropriate twist of $\mathcal E^*$. The index $(i,j)$ indicated the weight under the torus action. 

%These sections on each fiber $E_p$ form a coordinate system with a diagonal $(\CC^{\ast})^2$ action induced from automorphisms  of $Q^{-1}$ preserving $p$ and acting naturally on the factors of $\PP^1\times \PP^1$.
%These coordinates have weights $(3,2,2,1,1,0,0)$ with respect to the first action and $(0,0,1,1,2,2,3)$ with the other.
Let us now consider the net $K_p$ of skew symmetric forms in $\mathcal E_p$ corresponding to the matrix $M'$. It is generated by: 

%maps $\rho_i: \mathcal E(u_i)\to \mathcal E(v_i)^*,$ for some twists $u_i,v_i$. Since $y_0,y_1,y_2$ have weights $(1,0),(0,1)$ and $(0,0)$ respectively, we get $$\rho_0\in {\rm H}^0(Q^{-1},\wedge^2\mathcal E^*(1,0)),
%\rho_1\in {\rm H}^0(Q^{-1},\wedge^2\mathcal E^*(0,1)),\rho_2\in {\rm H}^0(Q^{-1},\wedge^2\mathcal E^*(0,0 ))$$(???).

%We saw that fiber-wise over $p\in Q^{-1}$ the sections are
%variety $\VAPS(q^2,10)$ is defined by the zero locus of sections
$$\rho_{y_0} :\qquad \frac{2}{3}e_{1,1}\wedge e_{1,2}-e_{2,1}\wedge e_{0,2}+ e_{2,0}\wedge e_{0,3},$$
$$\rho_{y_1} :\qquad  \frac{2}{3}e_{2,1}\wedge e_{1,1}-e_{2,0}\wedge e_{1,2}+ e_{3,0}\wedge e_{0,2},$$

$$\rho_{y_2} :\qquad  - \frac{1}{3}e_{2,1}\wedge e_{1,2}-e_{3,0}\wedge e_{0,3}.$$
These span a three-dimensional vector space space $K_p$.
We compute that  by acting with $\mathcal{G}$ on $K_p$ we obtain a rank 3 subbundle $K\subset \bigwedge^2 \mathcal E^*$. Indeed, it is enough to check that $g*K_p=K_p$ for $g$ such that $g\in \mathcal{P}$. 
In fact, using our explicit action we get 
$$K=(\mathcal{G}\times \mathbb C^3)/\mathcal{P}$$ with the action of $\mathcal{P}$ on $\mathbb C^3(\rho_{y_0},\rho_{y_1},\rho_{y_2})$ given  by 
$$g*_K\left\{\begin{array}{c}
       \rho_{y_0}\\
       \rho_{y_1}\\
       \rho_{y_2}
 \end{array}\right\} =\left\{\begin{array}{c}
  h_1^4h_4^2k_1^3k_4^3\rho_{y_0} - h_1^3h_2h_4^2k_1^3k_4^3 \rho_{y_2}\\
  h_1^3h_4^3k_1^4k_4^2\rho_{y_1} + h_1^3h_4^3k_1^3k_2k_4^2\rho_{y_2}\\
 h_1^3h_4^3k_1^3k_4^3 \rho_{y_2}
 \end{array}\right\}=$$
 $$\left\{\begin{array}{c}
h_1^2k_1k_4\rho_{y_0}-h_1h_2k_1k_4\rho_{y_2}\\
h_1h_4k_1^2\rho_{y_1}+h_1h_4k_1k_2\rho_{y_2}\\
h_1h_4k_1k_4 \rho_{y_2} \end{array}\right\},
$$
where $g\in \mathcal{P}$ is as in (\ref{g}).

On the other hand we can describe the bundle $R$ as follows. 
Over the point $p$ we have three vectors $(1,0,0,0)$, $(0,1,0,0)$, $(0,0,1,0)$, which span the projective tangent plane to $Q^{-1}$ and hence generate  $(P^1_{Q^{-1}})^*(-1,-1)$ through the  group action of $\mathcal{G}$ on $Q^{-1}$, i.e. the action of $\mathcal{P}$ on $\mathbb C^3$ with coordinates $y_0,y_1,y_2$ describing $(P^1_{Q^{-1}})^*(-1,-1)$ as $$(P^1_{Q^{-1}})^*(-1,-1)=(\mathcal{G}\times \mathbb C^3)/\mathcal{P}$$ is
$$g*_{R(-1,-1)}\left\{\begin{array}{c}
       x_0 \\
       x_1\\
       x_2
 \end{array}\right\} =\left\{\begin{array}{c}
h_1k_4x_0-h_2k_4x_2\\
h_4k_1x_1+h_4k_2x_2\\
h_4k_4x_2 \end{array}\right\}.$$

Furthermore, observe that in our setting $$\mathcal{O}(1,1)=(\mathcal{G}\times \mathbb C)/\mathcal{P},$$ where the action of $\mathcal{P}$ on $\mathbb C$ is:
$$g*_{(1,1)} c=k_1h_1c.$$
Indeed, the latter defines a line bundle with a section defined by
$$\sigma: \mathcal{G}\to \mathbb C$$
defined by 
$$\sigma\left(\left ( \begin{array}{cc}
     s_1 & s_2 \\
    s_3  & s_4
 \end{array}\right),  \left ( \begin{array}{cc}
     t_1 & t_2 \\
    t_3  & t_4
 \end{array}\right)\right) ={ s_3t_3},$$
 which is equivariant.

Hence after multiplication with $\mathcal O(1,1)$ which amounts to multiplying the representation matrix by $h_1k_1$ we  obtain a description of $R=(P^1_{Q^{-1}})^*$ as $$R=(\mathcal{G}\times \mathbb C^3)/\mathcal{P}$$ with the action of $\mathcal{P}$ on  $\mathbb C^3(y_0,y_1,y_2)$
$$g*_R\left\{\begin{array}{c}
       y_0 \\
       y_1\\
       y_2
 \end{array}\right\} =\left\{\begin{array}{c}
h_1^2k_1k_4y_0-h_1h_2k_1k_4y_2\\
h_1h_4k_1^2y_1+h_1h_4k_1k_2y_2\\
h_1h_4k_1k_4 y_2 \end{array}\right\}.$$

We are now ready to extend the map given by $M'$ to a map of bundles $$R\to K.$$ 
Indeed, $M'$ defines an isomorphism $R_p\to K_p$ given by 
$$(y_0,y_1,y_2)\mapsto (\rho_{y_0},\rho_{y_1},\rho_{y_2}).$$
This extends to an isomorphism:
$F:\mathcal{G}\times R_p \to \mathcal{G}\times K_p$

$$ ((G_1,G_2),(y_0, y_1, y_2))\mapsto ((G_1,G_2), (\rho_{y_0}, \rho_{y_1},  \rho_{y_2}) ). $$
The latter map $F$ is clearly an $\mathcal{G}$-equivariant map:
\begin{align*}
   &  F((M_1,M_2)*_{R}((G_1,G_2),({y}_0, y_1,y_2)))\\
  &=\;\;F((M_1G_1, M_2G_2),  (h_1^2k_1k_4^2 y_0 -h_1h_2k_1k_4y_2 , h_1h_4k_1^2 y_1+h_1h_4k_1k_2y_2, h_1h_4k_1k_4 y_2))\\
    &=\;\;((M_1G_1, M_2G_2),   (h_1^2k_1k_4^2 \rho_{y_0} -h_1h_2k_1k_4\rho_{y_2} , h_1h_4k_1^2 \rho_{y_1}+h_1h_4k_1k_2\rho_{y_2}, h_1h_4k_1k_4 \rho_{y_2}))\\
    &=\;\;(M_1,M_2)*_K ((G_1,G_2), (\rho_{y_0}, \rho_{y_1}, \rho_{y_2} ))\\&=\;\;(M_1,M_2)*_{K}  F ((G_1,G_2),(y_0,y_1,y_2)).
\end{align*}
The isomorphism $F$ hence induces an isomorphism of quotients $$
\bar F\colon R\to K.$$

We are now ready to prove the following:
\begin{proposition}\label{mainComp}
The main component (of saturated ideals) in $\VAPS(q^2,10)$ can be described in the Grassmann bundle $\GG(3, \mathcal E)$ over $Q^{-1}$  with relative universal bundle $\mathcal U$. It is isomorphic to the zero locus $$V(0)=\V_{22}\subset \GG(3, \mathcal E)$$ of a section of $\wedge^2  \mathcal U^*\otimes \pi^* R^* =\mathcal{H} om(\pi^* R,\wedge^2 \mathcal{U}^*)$ corresponding to the homomorphism $$R\simeq K\to \wedge^2 \mathcal E^* .$$

\end{proposition}
\begin{proof}
Recall that $V_p(0)$ is the locus of the rank $3$ subspaces of $E_p$ isotropic w.r.t. the subspace $K_p\subset \wedge^2 E_p^*$ of skew forms generated by $\rho_{y_0}, \rho_{y_1}, \rho_{y_2}$. To obtain $V(0)$ we then need only to observe 
$$V(0)=\bigcup_{g\in \mathcal{G}} g*_{\mathcal G}V_p(0)$$
where $*_{\mathcal G}$ is the action on $\GG(6,((q^2)_3^\bot)^*)$ induced from the action of $\mathcal{G}$ on
$((q^2)_3^\bot)^*$.
But $g*V_p(0)=V_{g*p}(0)$ and $V_{g*p}(0)$ is the locus of rank $3$ subspaces of $\mathcal E_{g*p}$ isotropic with respect to the subspace $g*K_p$ of skew forms generated by $g*\rho_{y_0}, g*\rho_{y_1}, g*\rho_{y_2}$.

    The variety $V(0)$ is hence realized in the relative Grassmann bundle $\GG (3, \mathcal E)$  over $Q^{-1}$ with structure map $\pi$ as the vanishing locus of the composition homomorphism  $K\to \pi^* \wedge^2 E^* \to \wedge^2 \mathcal U^*$, i.e. the vanishing locus of a section of the bundle $\wedge^2 \mathcal U^*\otimes K^* $.
Now since $K$ is isomorphic to $R$ by means of the map $\bar F$ induced by $M'$ we conclude the assertion.
\end{proof}
Recall the natural map
$$\psi: \GG(3,\mathcal{E})\to \GG(6,((q^2)_3^\bot)^*).$$
Let us compute the degree of the main component of the embedding $\VAPS(q^2,10)$ in the Pl\"ucker embedding of $\GG(6,16)$. 

First we need the following lemma.
\begin{lemma}\label{lem1}
The bundle $\mathcal E^*$ on $Q^{-1}$ has 
Chern polynomial
$$c_t(\mathcal E^*)=1+(3,3)t+34 t^2.$$
%The projective tangent bundle  to $Q^{-1}\subset \PP^3$ is the %projectivization of  $$\mathcal T=\oo(0,0)\oplus\oo(0,1)\oplus %\oo(1,0).$$
\end{lemma}
\begin{proof} By the exact sequence \ref{E^*}, the Chern polynomial of $\mathcal E^*$ is computed as the product,

\begin{align*}
c_t(\mathcal E^*)&= c_t({\rm Sym}^3\Omega_{Q^{-1}}(3,3))c_t({\rm Sym}^2\Omega_{Q^{-1}}(3,3))\\
&=c_t({\rm Sym}^3(\oo(-2,0)+\oo(0,-2))(3,3))\cdot\\
&\qquad\qquad\qquad\qquad\qquad c_t({\rm Sym}^2(\oo(-2,0)+\oo(0,-2))(3,3))\\
&=c_t(\oo(-3,3)+\oo(-1,1)+\oo(1,-1)+\oo(3,-3))\cdot \\
&\qquad\qquad\qquad\qquad\qquad
c_t(\oo(-1,3)+\oo(1,1)+\oo(3,-1))\\
&=(1+20t^2)(1+(3,3)t+14t^2)\\
&=1+(3,3)t+34t^2.\\
\end{align*}
\end{proof}

%Consider the map $\psi$ from $G(3,\mathcal E)$ to $\GG(6,16)\subset \PP(\bigwedge^3 V_{16})$
By Remark \ref{osculatingspace in VAPS} we infer that the apolar subschemes $\Gamma_0$ supported at $p$ as $p$ moves on $Q^{-1}$ spans the second order osculating spaces  to $S_{Q^{-1}}\subset \PP(V_{16})$.  Therefore, there is a section $\sigma_0: Q^{-1}\to \GG(3,\mathcal E)$ over $Q^{-1}$, whose image under the composition $\psi\circ \sigma_0:Q^{-1}\subset \GG(6,V_{16})$ is contained in $\VAPS(q^2,10)$.

The above map $\psi$ composed with the Pl\"ucker embedding $\GG(6,16)\subset \PP(\bigwedge^6 V_{16})$ 
%from $G(3,\mathcal E)$ to $\GG(6,16)\subset \PP(\bigwedge^3 V_{16})$ 
is given by a linear system on $\GG(3,\mathcal E)$ of the form $|\eta+a .h_1+b.h_2|$ where $h_1$ and $h_2$ are the pullbacks from $Q^{-1}$ of $\oo(0,1)$ and $\oo(1,0)$ respectively, and $\eta$ is a chosen tautological line bundle such that $\eta$ restricted to the section $\sigma_0: Q^{-1}\subset \GG(3,\mathcal E)$ is trivial. 
Let us compute $a,b$ by considering the restriction of $|\eta+a .h_1+b.h_2|$ to this section.
\begin{lemma} The image of the map $\psi\circ \sigma_0:Q^{-1}\to \GG(6,V_{16})$ given by the second order osculating spaces to the Veronese $S_{Q^{-1}}\subset \PP(V_{16})$ is mapped by the $10$th Veronese to the Pl\"ucker embedding.
\end{lemma}
\begin{proof}
We consider the osculating bundle %$OSC_{Q^{-1}}(3,3)$ 
on ${Q^{-1}}$ in its third Veronese embedding $S_{Q^{-1}}\subset \PP(V_{16})$, i.e. the embedding defined by the line bundle  ${\mathcal O}_{Q^{-1}}(3,3)$.  This osculating bundle is dual to the bundle $P^2_{Q^{-1}}(3,3)$ of second order principal parts, so the map
$\PP(P^2_{Q^{-1}}(3,3)^*)\to \PP(V_{16})$, is defined by the surjection 
$$V_{16}^*\times Q^{-1}\to P^2_{Q^{-1}}(3,3).$$
Therefore, the induced map $Q^{-1}\to \GG(6,V_{16})$ is defined by the line bundle
$\wedge^6(P^2_{Q^{-1}}(3,3))$.  By the exact sequences of bundles of principal parts above,
\begin{align*}
c_1(\wedge^6(P^2_{Q^{-1}}(3,3))&=c_1(P^2_{Q^{-1}}(3,3))\\
&=c_1({\rm Sym}^2\Omega_{Q^{-1}}(3,3)+\Omega_{Q^{-1}}(3,3)+\oo(3,3))=\oo(10,10).
\end{align*}
%It is the image of  $$H^0({\mathcal O}_{Q^{-1}}(3,3))_{Q^{-1}}\to P^2_{Q^{-1}}(3,3)$$ of second order principal parts on ${Q^{-1}}$ w.r.t the embedding.
%The bundles of principal parts fit in the exact sequences
%$$0\to \Omega_{Q^{-1}}(3,3)\to P^1_{Q^{-1}}(3,3)\to {\mathcal O}_{Q^{-1}}(3,3)\to 0$$
%and 
%$$0\to Sym^2(\Omega_{Q^{-1}})(3,3)\to P^2_{Q^{-1}}(3,3)\to P^1_{Q^{-1}}(3,3)\to %0$$
%So the total Chern polynomials of $P^1_{Q^{-1}}(3,3)$ and $P^2_{Q^{-1}}(3,3)$ are
%$$c_t(P^1_{Q^{-1}}(3,3))=1+(7,7)t+64t^2$$
%and 
%$$c_t(P^2_{Q^{-1}}(3,3))=1+(10,10)t+120t^2.$$
\end{proof}

It follows that $\psi$ is given by the system $|\eta+10h_1+10h_2|$ on $\GG(3,\mathcal E)$.
The restriction of $\psi$ to the fivefold $\V_{22}$ where the fibers are embedded as disjoint $3$-folds $V_{22}\subset \GG(6,V_{16})$, 
%and a $10$-tuple of cubics from $V(0)\subset \VAPS(q^2,10)$ defines a point on $Q^{-1}$ (as a fat point) 
is birational, so to compute the degree of $\V_{22}\subset \GG(6,V_{16})$ we compute the intersection $${\rm deg}\;\V_{22}=\V_{22}\cdot (\eta+10h_1+10h_2)^5.$$
Recall that the Chow ring of $\GG(3,\mathcal E)$ is described in \cite[Proposition 14.6.5]{Fulton}.
Let $\mathcal{U}$ be the universal rank three subbundle on $\GG(3,\mathcal E)$. The variety $\V_{22}$ is then the zero locus on $\GG(3,\mathcal E)$ of the composition of the vector bundle maps $$R\times \GG(3,\mathcal E)\to \wedge^2 \mathcal E^*\times \GG(3,\mathcal E)\to \wedge^2\mathcal{U}^* .$$
The class of the zero locus is the top Chern class of the rank nine vector bundle $\wedge^2\mathcal{U}^*\otimes R^*,$ so since
$R$ fits into the exact sequence
$$0\to \mathcal O\to R\to \mathcal O(0,2)\oplus \mathcal O(2,0)\to 0$$
the class of $\V_{22}$ is computed by the Chern class $$c_{9}(\bigwedge^2 \mathcal{U}^*\otimes (P^1_{Q^{-1}}))=c_9(\bigwedge^2\mathcal{U}^*\oplus \bigwedge^2\mathcal{U}^*(-2h_1) \oplus \bigwedge^2\mathcal{U}^*(-2h_2)).$$
%since we have an exact sequence $$0\to \mathcal O\to K\to \mathcal O(0,2)\oplus \mathcal O(2,0)\to 0.$$

\begin{corollary}

    The degree of the main component $\V_{22}$ of $\VAPS(q^2,10)$ in $\GG(6,16)$ is 2560.
  %  The canonical class is ???
\end{corollary}
\begin{proof}
The degree is given by a computation with the Macaulay2 package Schubert2 following the discussion above. 
%By the adjunction formula we have $K_{\V_{22}}=K_{\GG(3,\mathcal E)}\otimes \det ()$.
\end{proof}

\subsection{The special components}\label{specialcomponents} We can conduct a similar global study of the two special components by globalizing the descriptions of \( V_p(1) \) and \( V_p(2) \). We focus on the first component that depends on one line, say \( L \), in the quadric \( Q^{-1} \) through \( p \). The union \( V_L(1) \) of the compactifications \( V_p(1) \), as \( p \) moves along \( L \), is not a fibration over \( L \) since the fibers intersect (cf. Proposition \ref{apolar_L}). However, the varieties \( V_L(1) \), as \( L \) moves in its family on \( Q^{-1} \), form a fibration. We describe the total space \( V(1) \) of this fibration.

Recall that $V_{L,p}(1)$ is identified with a variety of isotropic $6$-spaces to the $12$-space $E_{L}$ with respect to a 
 a skew symmetric nondegenerate $(12\times 12)$-matrix  
$$M={\rm antidiagonal}\;(3,3,3,-4,-2,-6,6,2,4,-3,-3,-3).$$
This matrix defines a natural map $M:E_{L}^*\to E_{L}$.  

We next explain how this map extends to a skew symmetric bundle map $$\phi: E^*\to E$$ on $Q^{-1}$ where $E$ has rank $12$ the restriction of $\phi$ the fiber over $(0:0:1:0)$ is given by $M$.

For this we use an isomorphism between the space of apolar cubics and the space of $(3,3)$-forms on $Q^{-1}=\PP^1\times\PP^1$.

This isomorphism is induced by the simple substitution
$$(y_0,y_1,y_2,y_3)\mapsto (s_1t_1,s_2t_2,s_2t_1,s_1t_2).$$
The basis of the $12$-space $V_{12}:=(q^2)_3^{\perp}/K_L=E_L$ above is then mapped to the basis
$$s_1^2s_2(t_2^3,t_2^2t_1,t_2t_1^2,t_1^3),s_1s_2^2(t_2^3,t_2^2t_1,t_2t_1^2,t_1^3),s_2^2(t_2^3,t_2^2t_1,t_2t_1^2,t_1^3).$$
The rank $12$-bundle $E$ has a decomposition, when restricted to a line with coordinate $(t_1:t_2)$, i.e. where a linear form $s$ in $(s_1,s_2)$ vanishes, into three free bundles of rank four.  These three blocks correspond to the three blocks of base vectors in $V_{12}$.

To make this precise we consider first the basis with four blocks of rank four of the space $V_{3,3}$ of $(3,3)$ forms, where each block is canonically isomorphic to $V_{t,3}=\langle t_2^3,t_2^2t_1,t_2t_1^2,t_1^3\rangle$.
They define a filtration
$${\mathcal O}(-3,0)\otimes V_{t,3}\to{\mathcal O}(-2,0)\otimes 2V_{t,3}\to{\mathcal O}(-1,0)\otimes 3V_{t,3}\to{\mathcal O}(0,0)\otimes 4V_{t,3},$$
where the maps are multiplication by $s$ on the line where $s=0$.
The rank $12$  bundle $E$ is the sum of successive quotients:
$$0\to{\mathcal O}(-3,0)\otimes V_{t,3}\to{\mathcal O}(-2,0)\otimes 2V_{t,3}\to {\mathcal O}(-1,0)\otimes V_{t,3}\to 0, $$
$$0\to{\mathcal O}(-2,0)\otimes 2V_{t,3}\to{\mathcal O}(-1,0)\otimes 3V_{t,3}\to{\mathcal O}(1,0)\otimes V_{t,3}\to 0, $$
and
$$0\to{\mathcal O}(-1,0)\otimes 3V_{t,3}\to{\mathcal O}(0,0)\otimes 4V_{t,3}\to{\mathcal O}(3,0)\otimes V_{t,3}\to 0 .$$
That is 
$$E=({\mathcal O}(-1,0)\oplus{\mathcal O}(1,0)\oplus{\mathcal O}(3,0)) \otimes V_{t,3}.$$
Inside the second summand we consider the rank two subbundle 
$$0\to {\mathcal O}(1,-2)\otimes V_{t,1}\to{\mathcal O}(1,0)\otimes V_{t,3}$$
defined by multiplication by $t^2$, in the point $t=0$ on the line.

Inside the first summand we consider the rank three subbundle 
$$0\to {\mathcal O}(-1,-1)\otimes V_{t,2}\to{\mathcal O}(-1,0)\otimes V_{t,3}$$
defined by multiplication by $t$, in the point $t=0$ on the line.
The twisted cubic curve of pure powers in $\PP(V_{t,3})$ has as tangent line at $t$ the image of the line $\PP(V_{t,1})$ under multiplication by $t^2$.  Similarly, the osculating plane at $t$ the image of the plane $\PP(V_{t,2})$ under multiplication by $t^2$. 

The map $\phi:E^*\to E$ is a section of $(\wedge^2 E)(2,0)$, or equivalently a section of 
\begin{align*}
\wedge^2(E(-1,0))=&({\mathcal O}(-4,0)\oplus {\mathcal O}(0,0) \oplus {\mathcal O}(4,0))\otimes \wedge^2V_{t,3}\\
&\oplus ({\mathcal O}(-2,0)\oplus{\mathcal O}(0,0)\oplus {\mathcal O}(2,0))\otimes V_{t,3}\otimes V_{t,3}.
\end{align*}
 In fact it is a constant section $\omega:{\mathcal O}(0,0)\to \wedge^2(E(-1,0))$, since it is everywhere nondegenerate.  The image of the rank three bundle ${\mathcal O}(-1,-1)\otimes V_{t,2}$ in $E$ is isotropic for $\phi$.

The  rank four bundle ${\mathcal O}(-3,0)\otimes V_{t,3}$ parameterizes a $4$-space of cubics that is common to the ideal of any scheme in $V_{L_1}$. The quotient $6$-spaces of  cubics are found as subspaces in fibers of $E$.  They are the variety $\tilde {V}_{L_1}\subset LG_{\phi}(6,E)$ of 
%apolar length $10$ schemes are, modulo the %rank $4$-bundle ${\mathcal O}(-3,0)\otimes %V_{t,3}$, the 
isotropic $6$-spaces w.r.t. $\phi$ that contain the isotropic subspace 
${\mathcal O}(-1,-1)\otimes V_{t,2}$ and intersects the $6$-space 
$${\mathcal O}(-1,0)\otimes V_{t,3}\oplus{\mathcal O}(1,-2)\otimes V_{t,1}$$ in a $5$-space.
The variety $V_{L_1}$ in $V(1)$ is obtained
as the image of $\tilde {V}_{L_1}$ by the map into ${\rm LG}(6,V_{12})$.

We can now pass to the global description of $V(1)$.
% Those components of $\VAPS(q^2,10)$ are distinguished by the choice of the ruling on the quadric $Q^{-1}$ such that the descriptions are symmetric to each others and we denote them by $\H\tilde{V}(i)=V(i) \subset G(6,16)$ for $i=1,2$.
As we saw $V(1)$ is the image by the composition of natural maps $${\rm LG}(3,E_6)\to {\rm LG}(6,E)\to \GG(10,V_{16})$$ of a fibration $\tilde{V}(1)$ of cones over Veronese surfaces $ {\rm LG}(3,E_6)$ on a relative Lagrangian bundle. 
 
Let us consider the rank six vector bundle whose fibers over a point $p$ are the quotient of the $13$-dimensional space $U_{L,p}$ by the $7$-space $W_{L,p}$ of cubics as in Section \ref{compactificationofspecial}, i.e. the bundle $$E_6=\oo(-2,3)\oplus 4\oo(0,0)\oplus \oo(2,-3)$$ on the quadric $Q^{-1}$.
We consider the relative Grassmannian $\GG(3,E_6)\to Q^{-1}$ and denote by $\mathcal Q$ its relative quotient bundle on $\GG(3,E_6)$. The above isomorphism between the space $V_{16}$ of apolar cubic forms and $(3,3)$-forms on $Q^{-1}$, and fibers of $\GG(3,E_6)$ as subspaces of $(3,3)$-forms, defines a map $\GG(3,E_6)\to \GG(10,V_{16})$.
\begin{lemma}
    There exists a nowhere degenerate global section $\varphi$ in $H^0(\wedge^2 E_6)$.
    It defines a section of $\wedge^2 \mathcal Q$ on $\GG(3,E_6)$ whose zero locus is a relative Lagrangian Grassmannian $\pi\colon LG_{\varphi}(3,E_6)\to Q^{-1}$.
\end{lemma}
\begin{proof}We have 
$$\wedge^2 E_6 = 7\mathcal{O}(0,0) \oplus 4\mathcal{O}(-2,3) \oplus 4\mathcal{O}(-2,3),$$ 
so \(H^0(\wedge^2 E_6) = 7\), and all these sections are nonvanishing. Furthermore, the general ones are of maximal rank on all fibers by the discussion above.
   % The section can be found as before in one fiber of the Grassmann bundle and spread using the group action.
 %    Indeed it is enough to find a skew symmetric map $E_6()^{\vee}\to E_6()$ and consider its degeneracy locus.....
\end{proof}

 Now we consider the  section  of the Lagrangian bundle $LG_{\varphi}(3,E_6)$ associating to a point $p\in Q^{-1}$ the point in a the fiber  of $LG_{\varphi}(3,E_6)$ over $p$ that corresponds to the isotropic rank 3 subspace  $I_{ver,p}/K_L$ of $E_L=(E_6)_p$ from Proposition \ref{LG compactification of Vp1}. That induces a rank $3$ sub-bundle $V_3\subset E_6$ such that  $$V_3\simeq \oo_{Q^{-1}}(-2,3)\oplus 2\oo_{Q^{-1}} (0,-2).$$ 
Let $\mathcal U$ be the universal bundle on the relative Lagrangian Grassmannian with the exact sequence
  $$0\to \U\to \pi^{\ast}(E_6) \to \U^{\ast} \to 0.$$
The pull back $\pi^{\ast}(V_3)$ is an isotropic subbundle of the Lagrangian bundle $\pi^{\ast} (E_6)$.
\begin{proposition}\label{addComp}
    The additional components $\tilde{V}(i) \subset {\rm LG}(3, E_6)$, which are Lagrangian degeneracy loci, admit double covers that are projective rank $3$ bundles over a quadric. 
Thus, these components can be seen as quotients of a fiberwise involution on such a bundle.

\end{proposition}
\begin{proof}
It is enough to observe that the intersection of the tangent plane at $V_3$ with the Lagrangian Grassmannian is the set of three dimensional Lagrangian spaces that intersect $V_3$ along a line and use the definition of $V_3$ and Proposition \ref{LG compactification of Vp1}.
\end{proof}

It follows from  \cite[Chapter VII]{FP}
that the class of the second degeneracy locus along one Lagrangian fiber is $$(c_2c_1 -2c_3)(\U)+(c_2c_1 -2c_3)(\pi^* V_3).$$
The class of the third degeneracy locus is  $$( c_1 c_2 c_3 - 2 c_1^2 c_4 + 2 c_2 c_4 + 2 c_1 c_5 - 2 c_3^2)(\U) .$$

\begin{lemma}
    The map from $LG_{\varphi}(3,E_6)$ to $\GG(6,16)$ is given by a divisor with class $7 h_1+3h_2+ \xi$.

\end{lemma}
\begin{proof}
The Picard group of $LG_{\varphi}(3,E_6)$ is generated by three elements so the divisor giving the map is $$a.h_1+b.h_2+c.\xi.$$ 
     In order to determine $a,b,c$ we restrict this divisor to the section isomorphic to the base quadric $Q^{-1}$ defined by the vertices of the cones.  
\end{proof}

\begin{proposition}
    The degree of each of the special components \( V(1) \) in the Plücker embedding is 3392. Furthermore, the degrees of the images of the fourfolds corresponding to schemes with support on one line on the quadric are 112 or 392, depending on the chosen ruling \( h_1 \) or \( h_2 \), respectively.

\end{proposition}
\begin{proof}
%From \cite[Section 3.1]{FP} the cohomology ring of a relative Lagrangian Grassmannian $LG(3,E_6)$ is described as
%$$ \ZZ[ \sigma_1,\dots, \sigma_n ]/ (\sigma_i-2\sigma_{i+1}\sigma_{i-1}+\dots +(-1)^i2\sigma_{2i})_{i=1,2,3},$$
%where $\sigma_i=c_i(\U^{\vee})$ are Chern classes of the dual tautological bundle.
We work with the Schubert2 package of Macaulay 2 in the cohomology ring of the relative Grassmannian $\GG(3,E_6)$ in which the relative Lagrangian Grassmannian appears as a zero locus of a section of $\bigwedge^2 \U^{\ast}$, i.e. its class is $LG:=c_3(\bigwedge^2\U^{\ast})$. 
We need to find the sum of the following numbers
$$LG\cdot(7h_1+3h_2+\xi)^5\cdot(c_2c_1 -2c_3)(\U) =3392$$
$$LG\cdot(7h_1+3h_2+\xi)^5\cdot(c_2c_1 -2c_3)(\pi^* V_3) =0.$$

 Knowing that $$(7h_1+3h_2+\xi)^5=\xi^5+35h_1\xi^4+420\xi^3h_1h_2+15\xi^4h_2$$ we need to compute
$\xi^8$, $\xi^7h_1$ and $\xi^7h_2$.
 This is done with Schubert2 \cite{Macaulay_2} using the fact that the class of ${\rm LG}(3,E_6)$ 
%The section $\varphi$ defines a global section of the universal bundle $\wedge^2\mathcal U^{\vee}$.
in $\GG(3,E_6)$ is $c_3(\wedge^2\mathcal U^{\ast})$. 

Over the rulings we do the same computation but we intersect with $$(7h_1+3h_2+\xi)^4h_1$$ and $(7h_1+3h_2+\xi)^4h_2$ instead of
$(7h_1+3h_2+\xi)^5$.
%Recall that \cite{EisenbudHarris-AllThat} the Chow ring of the Grassmann bundle $G(k,\mathcal E)\to X$ where $\mathcal E$ has rank $n$ is 
%$$A (G)=A (X)[\zeta_1,..., \zeta_k]/ (\{\frac{c(\mathcal{E})}{1+\zeta+\zeta^2 +\dots} \}_l ,l<dim(G)-n+k)$$%where $\{-\}_l$ denotes the component of $-$ of dimension
%$l$, and $\zeta_k$ an element of degree $k$.
\end{proof}

\begin{remark}
  Note that the additional components $\tilde{V}(i) \subset {\rm LG}(3, E_6)$, which are Lagrangian degeneracy loci, admit double covers that are projective rank $3$ bundles over a quadric. 
Thus, these components can be seen as quotients of a fiberwise involution on such a bundle.

\end{remark}
\begin{remark} 
     For $i=1,2$ the map $\tilde{V}(i)\to V(i)$ is birational.
     After restricting to a line from the main ruling in $Q^{-1}$ the map is $1:1$ outside the locus $$\GG(2,4)\cap H\subset {\rm LG}(6,12)$$ (of bed limits) where $H\cap \GG(2,4)$ is the hyperplane section defined before in Proposition \ref{apolar_L}.
    Over this locus the map is $4:1$ branched over a surface that contains the twisted quartic in $H\cap \GG(2,4)$ that is the image of the curve of vertices of the cones over the Veronese surfaces in $H\cap \GG(2,4)$ (where it is generically  $3:1$). 
\end{remark}

\begin{remark}
    The main component intersects in $\GG(6,16)$ each of the additional components in divisors that are isomorphic to
  a fibration of quadric cones over $Q^{-1}$. We can obtain a more precise description by globalising 
  the approach in Proposition \ref{cone}. The common intersection of all the three components is empty.
\end{remark}

\bibliographystyle{alpha}
\bibliography{refs}
\section{Appendix}\label{Appendix}
In this Appendix we write down computer codes in Macaulay2 \cite{Macaulay_2} that find the equations that are referred to in the text.
\subsection{Equations when $n=2,r=3$}\label{equations23}
First, let us describe  the code to generate equations in the case of the third power of a ternary quadric:

We begin by introducing the Buchsbaum-Eisenbud matrix of syzygies for $(q^3)^{\bot}$.
\begin{verbatim}
A=QQ[x0,x1,x2,a_0..a_20,b_0..b_20]
q=(x0^2+x1*x2)^3;
BE=matrix{{0, 0, 0, 0, -2*7*x0, 2*7*x1, 0, 0, 0},
{0, 0, 0, 0, 2*7*x2, -3*7*x0, 7*x1, 0, 0},
{0, 0, 0, 0, 0, 7*x2, -x0, 3*x1, 0}, 
{0, 0,0, 0, 0, 0, 3*x2, -x0, x1},
{2*7*x0, -2*7*x2, 0, 0, 0, 0, 0, 0, 0},
{-2*7*x1, 3*7*x0, -7*x2, 0, 0, 0, 0, 0, 0},
{0, -7*x1, x0, -3*x2, 0, 0, 0, 0, 0},
{0, 0, -3*x1, x0, 0, 0, 0, 0, -x2}, 
{0, 0, 0, -x1, 0, 0, 0, x2, 0}}
psyz=transpose syz BE
diff(psyz,q3)==0  
            \end{verbatim}
The $VSP(s, 10)$ for a general sextic curve $\{s=0\}$ was originally described by Mukai using a the $(9 \times 9)$
Buchsbaum-Eisenbud matrix of syzygies for $(s)^{\bot}$.
The points in the VSP variety are defined as length $10$ codimension $2$ subschemes identified by their Hilbert-Burch matrix. This is a $(4 \times 5)$ submatrix of the Buchsbaum-Eisenbud matrix that is determined by a choice of a four-dimensional subspace in $\mathbb{C}^9$.
In order to have a more explicit description of this variety, we use the deformation method described in \cite{ranestad_schreyer_VSP}.

We first find a finite scheme $\Gamma_0$ that is apolar to $q^3$. It is described as a Hilbert-Burch matrix $\text{ps54}$, which is a submatrix of the above Buchsbaum-Eisenbud matrix BE.
\begin{verbatim}
ps54=submatrix(BE,{4,5,6,7,8}, {0,1,2,3})
Gamma0=submatrix(psyz,{0}, {4,5,6,7,8})
Gamma0*ps54==0 
syz=submatrix(psyz,{0}, {4,5,6,7,8,0,1,2,3})
  \end{verbatim}
In order to deform the ideal of $\Gamma_0$ generated by $Gamma0$ we need to control the syzygies. For this reason we introduce $Ja$ (this is $J(a)$ in the paper) an unfolding of $Gamma0$
and extend its syzygies to $Hb$ (denoted by $H(b)$ in the paper). 
\begin{verbatim}
F=genericMatrix(A,a_1,5,4)
D=diagonalMatrix{a_0,a_0,a_0,a_0,a_0}|F
Ja=D*transpose syz
H=diagonalMatrix{x2,x2,x2,x2,x2}*genericMatrix(A,b_1,5,4)
Hb=sub(ps54,A)+H
 \end{verbatim}
By \cite{artin_deform_of_sings}, in order to describe the ideal $EE$ of the flat family of deformations of the ideal $Gamma0$, as we saw in the paper, we need the condition $J(a)^t H(b) = 0$.
\begin{verbatim}
CA=(transpose Ja)*Hb
var5=gens (ideal(x0,x1,x2))^5;
coeff=diff(var5, flatten CA);
icoeff=ideal flatten coeff;
EE=eliminate(icoeff, {b_1,..,b_20});
betti EE
FC=ideal flatten submatrix(EE,{0}, {12..26}
codim FC==6
AB=QQ[a_0,a_1,a_6,a_11,a_16,a_17,a_18,a_19,a_20]
AA=sub(FC,AB);
JF=jacobian AA;
jf6=minors(6,JF)+AA;
minimalPrimes jf6
\end{verbatim}
We infer that the affine deformation of the ideal $I_{\Gamma_0}$ is parameterized by a surface $S\subset \AAA^8\subset \AAA^{20}$, where $\AAA^8$ is the common zeros of  the $12$ linear forms:
%{\tiny
%\begin{align*}
\begin{verbatim}
a_15-3a_19,a_14-a_18,a_13-9a_17,a_12-3a_16,a_10-7a_18,a_9-21a_17,
a_8-63a_16,a_7-7a_11,a_5-14,a_17,a_4-14a_16,a_3-14a_11,3a_2-2a_6
      \end{verbatim}
%\end{align*}
%}
and $S$ is the common zeros of the following $15$ quadratic forms (the ideal ${\rm AA}$ in the code above):
%\begin{tiny}
%\begin{align*}\label{affinedeformationequations}
\begin{verbatim}
3a_18^2-32a_17a_19+5a_16a_20,4a_17a_18-15a_16a_19+a_11a_20,
25a_16a_18-32a_11a_19+a_6a_20,63a_11a_1835a_6a_19+12a_1a_20,
60a_17^2-27a_11a_19+a_6a_20,105a_16a_17-7a_6a_19+3a_1a_20,
105a_11a_17-7a_6a_18+9a_1a_19,20a_0a_17+3a_19^2-a_18a_20,
875a_16^2-21a_6a_18+32a_1a_19,35a_11a_16-7a_6a_17+a_1a_18,
25a_0a_16+a_18a_19-3a_17a_20,63a_11^2-35a_6a_16+12a_1a_17,
3a_0a_11+a_17a_19-a_16a_20,4a_0a_6+9a_16a_19-3a_11a_20,
100a_0a_1+49a_11a_19-7a_6a_20
 \end{verbatim}
%\end{align*}
%  \end{tiny}    
\subsection{Equations when $n=3,r=2$}\label{unfolding32}
The method again consists of finding a point $\Gamma_0\in \VAPS(q^2,10)$, then finding an affine neighborhood in the deformation space.  
The point $\Gamma_0$ is found by taking the osculating spaces to the third Veronese embedding of the quadric $Q^{-1}$ (technically, we choose the first $10$ cubics among the $16$ generators in the Macaulay2 program).

The deformation of $I_{\Gamma_0}$ is described by matrices of the following shape.
$$J(a)=\begin{pmatrix}f_{000}\\
f_{001}\\
f_{011}\\
f_{111}\\
f_{002}\\
f_{012}\\
f_{112}\\
f_{022}\\
f_{122}\\
f_{222}\\
\end{pmatrix}=
\begin{pmatrix}y_{0}^3 -6y_0y_2y_3\\
y_{0}^2y_1 -2y_1y_2y_3\\
y_{0}y_1^2 +2y_0y_2y_3\\
y_1^3 +6y_1y_2y_3\\
y_0^2y_2 -y_2^2y_3\\
y_0y_1y_2\\
y_1^2y_2+y_2^2y_3\\
y_0y_2^2\\
y_1y_2^2\\
y_2^3\end{pmatrix}
+\begin{pmatrix}a_{00}\ldots a_{40}\\
a_{01}\ldots a_{41}\\
a_{02}\ldots a_{42}\\
\vdots \quad\quad\quad  \vdots\\
a_{09}\ldots a_{49}\end{pmatrix}\cdot 
\begin{pmatrix}y_0^2y_3-y_2y_3^2\\ y_0y_1y_3\\y_1^2y_3+y_2y_3^2\\y_0y_3^2\\y_1y_3^2%\\y_3^3
\end{pmatrix}$$
%   \end{tiny}
The following code in Macaulay2 \cite{Macaulay_2} using the package \emph{inverseSystems} find the matrix $Ja$ that is the transpose of $J(a)$.
\begin{verbatim}
loadPackage"inverseSystems"
Z=QQ[a_0..a_49]
S=Z[y0,y1,y2,y3]
Q=(y0^2-y1^2+y2*y3)^2
J=gens inverseSystem Q
I=ideal J_{6..15}
JS=(resolution I).dd_1 
F=(resolution I).dd_2
Ja=sub(JS,S)
   +sub(J_{5,4,3,2,1},S)*transpose sub(genericMatrix(Z,10,5),S)
   \end{verbatim}

 %  $$H(b)\cdot J(a)=$$
%   $$(-y_1+b_1y_3, y_0+b_2y_3, b_3y_3, b_4y_3,b_5y_3, -4y_3+b_6y_3, b_7y_3, b_8y_3,b_9y_3, b_{10}y_3)\cdot J(a)=$$
%   $$(b_1+a_{01})y_0^3y_3+(b_2+a_{11}-a_{00})y_0^2y_1y_3+(b_3-a_{10}+a_{12})y_0y_1^2y_3+(b_4-a_{20})y1^3y_3+(b_5)y_0^2y_2y_3+$$
%   $$(b_6)y_0y_1y_2y_3+(b_7)y_1^2y_3+(b_8)y_0y_2^2y_3+(b_9)y_0y_2^2y_3+(b_{10})y_2^3y_3$$
The extension of the syzygy matrix for $I_{\Gamma_0}$
%{\tiny
%$
%\left(\begin{array}{ccccccccccccccc}
 %   -y1& 0& 0& -y2& 0& 0& 0& 0& 0& 0& 0& 0& 0& 0& 0\\
 %   y0& -y1& 0& 0& -y2& 0& 0& 0& 0&
 %     0& 0& 0& 0& 0& 0\\
 %     0& y0& -y1&
 %     0& 0& 0& 0& 0& 0& 0& 0& 0& 0&
 %     0& 0\\
 %     0& 0& y0& 0& 0& 0& 0& 0& -y2& 0& 0&0& 0& 0& 0\\
 %     0& -2y3& 0& y0& y1& -y1& 0& 0& 0& -y2& & 0& 0&0& 0\\
 %     -4y3& 0& -4y3& 0& 0& y0& y1& -y1& 0& 0& -y2& 0& 0& 0&0\\
 %     0& -2y3& 0& 0& 0& 0& 0& y0& y1& & & 0& -y2& 0& 0\\
 %     0& 0& 0& -5y3& 0& 0& 2y3& -y3& 0& y0& y1& -y1& 0& -y2& 0\\
 %     0& 0& 0& 0& -y3& -y3& 0& 0& 5y3& 0& 0& y0& y1& 0& -y2\\
 %     0&0& 0& 0& 0& 0& 0& 0& 0& -y3& 0& 0& y3& y0& y1\\
%\end{array}
%\right)
%$
%}
%{\tiny
\[
H(0)=\left(\begin{array}{cccccccccc}
-y_1& y_0& 0& 0& 0& -4y_3& 0& 0& 0& 0\\ 0& -y_1& y_0& 0& -2y_3& 0& -2y_3& 0& 0& 0\\ 0& 0& -y_1& y_0& 0& -4y_3& 0& 0& 0& 0\\ -y_2& 0& 0& 0& y_0&
      0& 0& -5y_3& 0& 0\\ 0& -y_2& 0& 0& y_1& 0& 0& 0& -y_3& 0\\ 0& 0& 0& 0& -y_1& y_0& 0& 0& -y_3& 0\\ 0& 0& -y_2& 0& 0& y_1& 0& 2y_3& 0& 0\\ 0& 0& 0& 0& 0&
      -y_1& y_0& -y_3& 0& 0\\ 0& 0& 0& -y_2& 0& 0& y_1& 0& 5y_3& 0\\ 0& 0& 0& 0& -y_2& 0& 0& y_0& 0& -y_3\\ 0& 0& 0& 0& 0& -y_2& 0& y_1& 0& 0\\ 0& 0& 0& 0& 0&
      0& 0& -y_1& y_0& 0\\ 0& 0& 0& 0& 0& 0& -y_2& 0& y_1& y_3\\ 0& 0& 0& 0& 0& 0& 0& -y_2& 0& y_0\\ 0& 0& 0& 0& 0& 0& 0& 0& -y_2& y_1
\end{array}
\right)
\]
%}
is computed by considering for each $i$ an extension $H(b)_i$ of the $i$-th syzygy $H(0)_i$ in $H(0)$ (the $i$'th row in this matrix).  The first syzygy $H(0)_1$ extended with multiples of $y_3$ is 
$$H(b)_1=
   (-y_1+b_1y_3, y_0+b_2y_3, b_3y_3, b_4y_3,b_5y_3, -4y_3+b_6y_3, b_7y_3, b_8y_3,b_9y_3,b_{10}y_3).$$
 This extension is a syzygy for $J(a)$ if
$$H(b)_1\cdot J(a)=0.$$
%\begin{array}{rl}
% H(b)_1\cdot J(a)&= \\
% =&(-y_1+b_1y_3, y_0+b_2y_3, b_3y_3, b_4y_3,b_5y_3, -4y_3+b_6y_3, b_7y_3, b_8y_3,b_9y_3, b_{10}y_3)\cdot J(a)\\
%   =&0.\\
%   \end{array}
   In the product $H(b)_1\cdot J(a)$, the part that is linear in $y_3$ is
   \begin{align*}
   (b_1+a_{01})y_0^3y_3&+(b_2+a_{11}-a_{00})y_0^2y_1y_3+(b_3-a_{10}+a_{12})y_0y_1^2y_3+
   (b_4-a_{20})y_1^3y_3\\
   &+b_5y_0^2y_2y_3+
b_6y_0y_1y_2y_3+b_7y_1^2y_3+b_8y_0y_2^2y_3+b_9y_0y_2^2y_3+b_{10}y_2^3y_3.
   \end{align*}
   So the vanishing of $H(b)_1\cdot J(a)$ means
   $$(b_1,b_2,b_3,b_4,b_5,b_6,b_7,b_8,b_9,b_{10})=
   (-a_{01},a_{00}-a_{11},a_{10}-a_{21},a_{20},0,0,0,0,0,0).$$
Similarly we extend all syzygies in $H(0)$ to syzygies $H(a)=H(0)+y_3\cdot U(a)$, where 
%{\tiny
%$\left(\begin{array}{ccccccccccccccc}
%    -a_1& -a_2& -a_3& -a_4& 0& -a_5& 0& -a_6& 0& -a_7& 0& -a_8& 0& -a_9& 0\\
 %   a_0-a_{11}& a_1-a_{12}& a_2-a_{13}& -a_{14}& -a_4& a_4-a_{15}& -a_5& a_5-a_{16}& -a_6&
  %    -a_{17}& -a_7& a_7-a_{18}& -a_8& -a_{19}& -a_9\\
   %   a_{10}-a_{21}& a_{11}-a_{22}& a_{12}-a_{23}&
   %   -a_{24}& -a_{14}& a_{14}-a_{25}& -a_{15}& a_{15}-a_{26}& -a_{16}& -a_{27}& -a_{17}& a_{17}-a_{28}& -a_{18}&
    %  -a_{29}& -a_{19}\\
     % a_{20}& a_{21}& a_{22}& 0& -a_{24}& a_{24}& -a_{25}& a_{25}& -a_{26}& 0& -a_{27}&a_{27}& -a_{28}& 0& -a_{29}\\
      %0& 0& 0& y0+a_0&a_1& 0& a_2& 0& a_3& a_4& a_5& 0& a_6&a_7& a_8\\
     % 0& 0& 0& a_{10}& a_{11}& 0& a_{12}& 0& a_{13}& a_{14}& a_{15}& 0& a_{16}& a_{17}&a_{18}\\
     % 0& 0& 0& a_{20}& a_{21}& 0& a_{22}& 0& a_{23}& a_{24}& a_{25}& 0& a_{26}& a_{27}& a_{28}\\
     % 0& 0& 0& 0& 0& 0&0 &0 & 0&0& 0& 0& 0& 0& 0\\
     % 0& 0& 0& 0& 0& 0& 0& 0& 0& 0& 0&0&0& 0& 0\\
     % 0&0& 0& 0& 0& 0& 0& 0& 0& 0& 0& 0& 0& 0& 0\\
%\end{array}
%\right)$

\[
U(a)=\left(\begin{array}{cccccccccc} -a_{01}& a_{00}-a_{11}& a_{10}-a_{21}& a_{20}& 0& 0& 0& 0& 0& 0\\ -a_{02}& a_{01}-a_{12}& a_{11}-a_{22}& a_{21}& 0& 0& 0& 0& 0& 0\\ -a_{03}& a_{02}-a_{13}& a_{12}-a_{23}& a_{22}&
      0& 0& 0& 0& 0& 0\\ -a_{04}& -a_{14}& -a_{24}& 0& a_{00}& a_{10}& a_{20}& 0& 0& 0\\ 0& -a_4& -a_{14}& -a_{24}& a_1& a_{11}& a_{21}& 0& 0& 0\\ -a_{05}& a_{04}-a_{15}&
      a_{14}-a_{25}& a_{24}& 0& 0& 0& 0& 0& 0\\ 0& -a_{05}& -a_{15}& -a_{25}& a_{02}& a_{12}& a_{22}& 0& 0& 0\\ -a_{06}& a_{05}-a_{16}& a_{15}-a_{26}& a_{25}& 0& 0& 0& 0& 0& 0\\ 0&
      -a_{06}& -a_{16}& -a_{26}& a_{03}& a_{13}& a_{23}& 0& 0& 0\\ -a_{07}& -a_{17}& -a_{27}& 0& a_{04}& a_{14}& a_{24}& 0& 0& 0\\ 0& -a_7& -a_{17}& -a_{27}& a_{05}& a_{15}& a_{25}& 0& 0&
      0\\ -a_{08}& a_{07}-a_{18}& a_{17}-a_{28}& a_{27}& 0& 0& 0& 0& 0& 0\\ 0& -a_{08}& -a_{18}& -a_{28}& a_{06}& a_{16}& a_{26}& 0& 0& 0\\ -a_{09}& -a_{19}& -a_{29}& 0& a_{07}& a_{17}&
      a_{27}& 0& 0& 0\\ 0& -a_{09}& -a_{19}& -a_{29}& a_{08}& a_{18}& a_{28}& 0& 0& 0\\
\end{array}
\right)
\]
%}

\vskip .5cm

The transpose $Ha$ of the matrix $H(a)$ can be implemented in Macaulay2 by the following procedure (building on the previous code):

\begin{verbatim}

T=Ja*sub(F,S)
MN=sub(sub(JS,S), y3=>0)
Ha=(sub(F,S)
   -y3*(coefficients(diff(matrix{{sub(y3,S)}},T),Monomials=>MN))_1)
    \end{verbatim}
    
We can now find, the ideal \emph{radunfold} of the deformation space $(V_p^{aff})_{red}$ and then describe the components $V0p$, $V1p$ and $V2p$.

\begin{verbatim}

da4=ideal flatten (coefficients (flatten Ja*Ha))_1;
unfold=ideal mingens da4;
radunfold=radical unfold;
betti radunfold
V0p=(minimalPrimes radunfold)#0
V1p=(minimalPrimes radunfold)#1
V2p=(minimalPrimes radunfold)#2
\end{verbatim}

\end{document}